\documentclass[10pt]{amsart}
\usepackage{amscd,amssymb,graphics}

\usepackage{amsfonts}
\usepackage{amsmath}
\usepackage{amsxtra}
\usepackage{latexsym}
\usepackage[mathcal]{eucal}

\input xy
\xyoption{all}
\usepackage{epsfig}
\usepackage[hidelinks]{hyperref}

\newtheorem{thm}{Theorem}[section]

\newtheorem{cor}[thm]{Corollary}
\newtheorem{lem}[thm]{Lemma}

\newtheorem{prop}[thm]{Proposition}

\theoremstyle{definition}
\newtheorem{defin}[thm]{Definition}

\theoremstyle{remark}
\newtheorem{remark}[thm]{Remark}

\newtheorem{ex}[thm]{Example}
\newtheorem{exs}[thm]{Examples}

\newtheorem{question}[thm]{Question}
\newtheorem{problem}[thm]{Problem}

\numberwithin{equation}{section}

\newcommand{\delete}[1]{} 

\newcommand{\sig}{\sigma}

\def\eps{{\varepsilon}}
\newcommand{\sk}{\vskip 0.3cm}

\newcommand{\ov}{\overline}

\newcommand{\ben}{\begin{enumerate}}

\newcommand{\een}{\end{enumerate}}
\newcommand{\bit}{\begin{itemize}}

\newcommand{\eit}{\end{itemize}}

\newcommand{\TFAE}{The following are equivalent: }

\def\R {{\mathbb R}}
\def\N {{\mathbb N}}
\def\Z {{\mathbb Z}}
\def\Q {{\mathbb Q}}

\def\T {{\mathbb T}}

\def\RUC{{\hbox{RUC\,}^b}}

\def\Iso{{\mathrm{Iso}}\,}
\def\Aut{{\mathrm Aut}\,}

\def\diam{{\mathrm{diam}}}

\def\A{{\mathcal{A}}}
\def\B{{\mathcal{B}}}

\def\F{{\mathcal F}}

\def\K{{\mathcal K}}

\def\Iso{\operatorname{Iso}}
\def\Asp{\operatorname{Asp}}
\def\RUC{\operatorname{RUC}}
\def\RMC{\operatorname{RMC}}
\def\SUC{\operatorname{SUC}}
\def\LUC{\operatorname{LUC}}
\def\UC{\operatorname{UC}}
\def\WAP{\operatorname{WAP}}
\def\Tame{\operatorname{Tame}}

\def\QED{\nobreak\quad\ifmmode\roman{Q.E.D.}\else{\rm Q.E.D.}\fi}

\def\a{\alpha}
\def\om{\omega}
\def\Om{\Omega}

\def\s{\sigma}

\newcommand{\Ga}{A}
\newcommand{\ga}{\gamma}
\newcommand{\del}{\delta}

\newcommand{\br}{\vspace{4 mm}}

\newcommand{\cls}{{\rm{cls\,}}}

\newcommand{\al}{\alpha}
\newcommand{\ep}{\varepsilon}

\newcommand{\Ucal}{\mathcal{U}}

\newcommand{\OC}{\bar{\mathcal{O}}}

\newcommand{\sgn}{{\rm{sgn}}}

\newcommand{\card}{\rm{card\,}}

\newcommand{\Exp}{{\rm{Exp\,}}}

\newcommand{\proj}{{\rm{proj\,}}}

\newcommand{\rest}{\upharpoonright}

\begin{document}

\title[]
{Eventual nonsensitivity and tame dynamical systems}

\author[]{Eli Glasner}
\address{Department of Mathematics,
Tel-Aviv University, Ramat Aviv, Israel}
\email{glasner@math.tau.ac.il}
\urladdr{http://www.math.tau.ac.il/$^\sim$glasner}

\author[]{Michael Megrelishvili}
\address{Department of Mathematics,
Bar-Ilan University, 52900 Ramat-Gan, Israel}
\email{megereli@math.biu.ac.il}
\urladdr{http://www.math.biu.ac.il/$^\sim$megereli}

\date{September 22, 2016}

\begin{abstract}
In this paper we characterize tame dynamical systems and functions in terms of eventual non-sensitivity
and eventual fragmentability.
As a notable application we obtain a neat characterization of tame subshifts 
$X \subset \{0,1\}^{\mathbb Z}$:
for every infinite subset $L \subseteq {\mathbb Z}$ there exists an infinite subset $K \subseteq L$
such that $\pi_{K}(X)$ is a countable subset of $\{0,1\}^K$.
The notion of eventual fragmentability
is one of the properties we encounter which indicate some ``smallness" of a family.
We investigate a ``smallness hierarchy" for families of continuous functions on compact dynamical systems, and link the existence of a ``small" family which separates points of a dynamical system $(G,X)$
to the representability of $X$ on ``good" Banach spaces.
For example, for metric dynamical systems the property of admitting a separating family which is eventually fragmented is equivalent to being tame.
We give some sufficient conditions for coding functions to be tame
and, among other applications, show that certain multidimensional analogues of Sturmian sequences are  tame. We also show that linearly ordered dynamical systems are tame and discuss examples where universal dynamical systems associated with certain Polish groups are tame.
\end{abstract}
 


\thanks{This research was supported by a grant of Israel Science Foundation (ISF 668/13)}

\thanks{The first named author thanks the Hausdorff Institute at Bonn for the opportunity to participate in the Program ``Universality and Homogeneity" where part of this work was written, November 2013.}

\maketitle

\setcounter{tocdepth}{1}
\tableofcontents

\section{Introduction}

Tame dynamical systems were introduced by A. K\"{o}hler \cite{Ko}
(under the name ``regular systems") and
their theory was later developed in a series of works by several authors
(see e.g. \cite{Gl-tame}, \cite{GM1}, \cite{GM-rose}, \cite{H}, \cite{KL} and \cite{Gl-str}).   

More recently connections to other areas of mathematics like coding theory,
substitutions and tilings, and even model theory and logic were established 
(see e.g. \cite{Auj}, 
\cite{Ibar}, and the survey \cite{GM-survey} for more details). 
In the present work we introduce a new approach to the study of tame and other related
dynamical systems in terms of ``small families" of functions.

Given a topological semigroup $S$, one way to measure the complexity of a compact dynamical
$S$-system $X$ is to investigate its representability on ``good" Banach spaces, \cite{GM-survey, GM-AffComp,GM-rose}.
Another is to ask whether the points of $X$ can be separated by a norm bounded $S$-invariant family $F \subset C(X)$ of continuous functions on $X$, such that $F$ is ``small" in some sense or another.
Usually being ``small" means that the pointwise closure $\cls_p(F)$ of $F$ (the \emph{envelope of} $F$) in $\R^X$ is a ``small" compactum.
Two typical examples are (a) when $\cls_p(F) \subset C(X)$ and  (b) when
$\cls_p(F)$ consists of fragmented functions
(Baire 1, when $X$ is metrizable). 
It is equivalent to saying that $F$ does not contain independent subsequences.

It turns out (Theorem \ref{t:AspDuality})
that the first case (a) characterizes the reflexively representable dynamical systems,
(these are the dynamical analog of Eberlein compacta) or,
for metric dynamical systems $X$,
the class of WAP systems.

In the second case (b) we obtain a characterization of the class of Rosenthal representable dynamical systems, or, for metric dynamical systems $X$,
the class of tame systems
(a Banach space $V$ is said to be {\em Rosenthal} \cite{GM-rose, GM-AffComp}
if it does not contain an isomorphic copy of the Banach space $l_1$).

We also have a characterization of the intermediate class of Asplund representable,
or hereditarily nonsensitive (HNS) systems.
Namely, an $S$-system $X$ is HNS iff there exists a separating bounded family $F \subset C(X)$ which is a fragmented family (Definition \ref{d:sens-fr-f}).
See Theorem \ref{t:WRN} below.

These subjects are treated in Sections 2, 3 and 4. In section 5 we apply these results to symbolic $\Z$-systems, and, in Section 6, relate them to entropy theory.

In Section \ref{s:tame-type} we give some sufficient conditions on certain
coding functions which ensure that the associated dynamical systems are tame.
In particular, we conclude that some multidimensional Sturmian-like functions are tame (Theorem \ref{multi}).

In Section \ref{s:order} we show that order preserving dynamical systems are tame. The same is true for the system $(H_+(\T),\T)$, where
$H_+(\T)$ is the Polish group of orientation preserving homeomorphisms of the circle $\T$.
Recall that for every topological group $G$ there exists a universal minimal system $M(G)$
and also a universal irreducible affine $G$-system $I\!A(G)$. 
In Section \ref{sec,int}
we discuss some examples, where $M(G)$ and $I\!A(G)$ are tame. 
In particular, we show that this is the case for the group $G=H_+(\T)$, using a well known result of Pestov \cite{Pest98} which identifies $M(G)$ as the tautological action of $G$ on the circle $\T$.

In Section \ref{app}
we present the proof of Theorem \ref{main} communicated to
us by Stevo Todor\u{c}evi\'{c} which asserts that not every compact space is 
Rosenthal-representable. More precisely, $\beta \N$ cannot be $w^*$-embedded into the dual space $V^*$ of a Rosenthal Banach space $V$.

\sk
Our main results in this work are:

\bit 
\item
In Theorem \ref{t:AspDuality} we consider a hierarchy of smallness of bounded $S$-invariant
families $F \subset C(X)$ on a compact dynamical system $X$.
In particular, we investigate when the evaluation map $F \times X \to \R$ comes from the canonical bilinear evaluation map $V \times V^* \to \R$ for good classes of Banach spaces $V$.
We then show that fragmentability and eventual fragmentability of
a separating family $F$ characterize Asplund and Rosenthal representability respectively (Theorem \ref{t:WRN}).
\item
A characterization of tame systems and functions in terms of eventual non-sensitivity
(Theorems \ref{t:tame-f} and \ref{t:tame}).
\item
A characterization of tame symbolic dynamical systems 
(Theorems \ref{subshifts} and \ref{t:3}.2). 
A combinatorial characterization of tame subsets $D \subset \Z$ (i.e., subsets $D$ such that the associated subshift $X_D\subset \{0,1\}^{\Z}$ is tame), Theorem \ref{tame subsets in Z}.
\item Theorems \ref{2811} and \ref{t:circleWalk} (see also Remark \ref{applicBV}) 
give some useful sufficient conditions for the tameness of coding functions. 

\item 
Theorem \ref{t:OrderedAreTame} shows that any linearly ordered dynamical system is tame. 
 Moreover, such dynamical systems are representable on Rosenthal Banach spaces. 
 By Theorem \ref{GenH_+}, the universal minimal $G$-system $M(G)=\T$ for $G=H_+(\T)$ is tame. 
\eit


\subsection{Preliminaries}

We use the notation of \cite{GM-AffComp, GM-survey}.
By a topological space we mean mostly a Tychonoff (completely regular Hausdorff) space. 
The closure operator in topological spaces
will be denoted by $\cls$.  
A topological space $X$ is called \emph{hereditarily Baire} if every closed subspace of $X$ is a Baire space.
A function $f: X \to Y$ is \emph{Baire class 1 function} if the inverse image $f^{-1}(O)$ of every open set is $F_\sigma$ in $X$, \cite{Kech}.

When $(X,\mu)$ is a uniform space, a subset $\Gamma \subset \mu$ is said to be a (uniform) subbase of $\mu$ if 
the finite intersections of the elements of $\Gamma$ form a base of the uniform structure $\mu$. 
We allow uniform structures which are not necessarily Hausdorff (i.e. $\cap \{\al : \al \in \mu\}$
may properly contain the diagonal), such as uniform structures which are induced by a pseudometric. 
On the other hand, a compact space will mean ``compact and Hausdorff". Recall that any compact space $X$ admits a unique compatible 
uniform structure, namely the set of all neighborhoods of the diagonal in $X \times X$. 

For a pair of topological spaces $X$ and $Y$ we let $C(X,Y)$ denote the  
set of continuous functions from $X$ into $Y$.
We will take $C(X)$ to be the Banach 
algebra of {\em bounded} continuous real valued functions even when 
$X$ is not necessarily compact.
For a compact space $X$ we let
$$P(X)=\{\mu \in C(X)^*: \ \|\mu\|=\mu(\textbf{1})=1\}$$
be the $w^*$-compact  subset of $C(X)^*$ which, as usual, is identified with
the space of \emph{probability measures} on $X$.

All semigroups $S$ are assumed to be monoids, 
i.e., semigroups with a neutral element which will be denoted by $e$. 
An {\em action} of $S$ on a space $X$ is a map $\pi : S \times X \to X$ such that
$\pi(st,x) = \pi(s,\pi(t,x))$ for every $s, t \in S$ and $x \in X$. We usually simply write
$sx$ for $\pi(s,x)$.
Also actions are usually \emph{monoidal} (meaning $ex=x, \forall x \in X$). 
A topologized semigroup $S$ is said to be \emph{semitopological} if its multiplication is separately continuous. 

A \emph{dynamical $S$-system} $X$ 
(or an $S$-\emph{space}) 
is a topological space $X$ equipped with a separately continuous action $\pi: S \times X \to X$ of 
a semitopological semigroup $S$. 
As usual, a continuous map $\a : X \to Y$ between two $S$-systems is called an {\em $S$-map} 
when $\a(sx)=s\a(x)$ for every $(s,x) \in S \times X$.  

For every $S$-space $X$ we have a pointwise continuous monoid homomorphism
$j: S \to C(X,X)$, $j(s)=\tilde{s}$,
where $\tilde{s}: X \to X, x \mapsto sx=\pi(s,x)$ is the {\em $s$-translation} ($s \in S$). 

The {\em enveloping semigroup} $E(S,X)$ (or just $E(X)$) is
defined as the pointwise closure $E(S,X)={\cls}_p(j(S))$ of $\tilde{S}=j(S)$ in $X^X$. If $X$ is a compact $S$-system then $E(S,X)$ 
is always a right topological compact monoid.
Algebraic and topological properties of the families $j(S)$ and $E(X)$ reflect the asymptotic dynamical behavior of $(S,X)$. 
More generally, for a family $F \subset C(X,Y)$ we define its \emph{envelope} as the pointwise closure
$\cls_p(F)$ of $F$ in $Y^X$.

By an (invertible)
\emph{cascade} on $X$ we mean an $S$-action $S \times X \to X$, where 
$S:=\N \cup \{0\}$ is the additive semigroup of nonnegative integers
(respectively, $S=(\Z,+)$).
We write it sometimes as a pair $(X,\sigma)$ where $\sigma$ is the $s$-translation of $X$ corresponding to $s=1$
($0$ acts as the identity). 
By $\OC_S(x_0)$ we denote the closure of the orbit $Sx_0$ in $X$.   

Let $F, X,Y$ be topological spaces and $w: F \times X \to Y, w(f,x):=f(x)$ a function.
We say that $F$ has the {\em Double Limit Property} (DLP) on $X$ 
if for every sequence $\{f_n\} \subset F$ and every sequence
 $\{x_m\} \subset X$ the limits
 $$\lim_n \lim_m f_n(x_m) \ \ \  \text{and} \ \ \ \lim_m \lim_n f_n(x_m)$$
 are equal whenever they both exist. 
 We also say that $w$ has the DLP. 

Given a function $f \in C(X)$ on a compact $S$-space $X$ we consider its orbit
$fS : = \{f \circ \tilde{s}:  s \in S\} \subset C(X)$.
One can estimate the dynamical complexity of $f$ by considering the
pointwise closure
$$\cls_p(fS)=f E(S,X) := \{f \circ q : q \in E(S,X)\}$$ in $\R^X.$
Various degrees of ``smallness" of this compactum lead to a natural hierarchy. The classical examples are the almost periodic (AP) and weakly almost periodic (WAP) functions. The norm compactness of $\cls_p(fS)$
in $C(X)$ is the characteristic trait of a Bohr almost periodic function.
For the latter property we have:
\begin{defin} \label{d:WAP} Let $X$ be a compact $S$-system.
\ben
\item $f \in C(X)$ is said to be WAP if one of the following equivalent conditions is satisfied:
\ben
\item
$fS$ is weakly precompact in $C(X)$;
\item $\cls_p(fS) \subset C(X)$;
\item $fS$ has DLP on $X$.
\een
\item
$(S,X)$ is said to be WAP if one of the following equivalent conditions is satisfied:
\ben
\item every member $p \in E(S,X)$
is a continuous function $X \to X$;
\item $\WAP(X)=C(X)$.
\een

\een
\end{defin}

The equivalences can be verified using Grothendieck's classical results. See for example,
\cite[Theorem A4]{BJM} and \cite[Theorem A5]{BJM}.

\begin{remark} \label{r:JCont2} 
If $(S,X)$ is WAP then it is easy to see that $S \times X \to X$ has DLP
(assuming the contrary, choose $f \in C(X)$ which separates the corresponding double limits in $X$. Then $f \notin WAP(X)$).
If the compactum $X$ is metrizable (or, more generally, sequentially compact) then the converse is also true.
To see this use Definition \ref{d:WAP} and the diagonal arguments.

\end{remark}

\sk When $V$ is a Banach space we denote by $B$, or $B_V$,
the closed unit ball of $V$. $B^*=B_{V^*}$ and $B^{**}:=B_{V^{**}}$ will denote the
weak$^*$ compact unit balls in the dual $V^*$ and second dual
$V^{**}$ of $V$ respectively. 

The following DLP characterization of reflexive spaces combines Grothendieck's double limit 
criterion of weak compactness (see for example \cite[Theorem A5]{BJM}) and a well known fact that a Banach space $V$ is reflexive iff $B_V$ is weakly compact.  

\begin{thm} \label{t:ReflChar}
Let $V$ be a Banach space. The following conditions are
equivalent:
\ben
\item $V$ is reflexive.
\item $B$ has DLP on $B^*$.
\item every bounded subset $F \subset V$ has DLP on every bounded
 $X \subset V^*$.
\item $B \subset V$ is weakly compact.
\een
\end{thm}

\sk

Later we will recall 
and examine the definitions of Asplund and tame functions which are based on the concept of \emph{fragmentability} (Definition \ref{d:sens-fr-f} below). 
This concept has its roots in Banach space theory.
We will also introduce various definitions of non-sensitivity for invariant families of functions.
Such families are the main object of study in this paper.
As some recent results show this approach is quite effective and provides
the right level of generality. See for example the
study of HNS systems in \cite{GM1} and some applications to fixed point theorems
in \cite{GM-fp}.

\subsection{Some classes of dynamical systems}

We next briefly describe the two main classes of dynamical systems which will be analyzed.
We begin with a generalized version of the notion of dynamical sensitivity.

\begin{defin}  \label{d:HNS0} (See for example \cite{AuYo, GW1, GM1}.)
Let $(X,\tau)$ be a compact $S$ dynamical system endowed with its unique compatible
uniform structure $\mu$.
\begin{enumerate}
\item
The dynamical $S$-system $X$ has {\em sensitive dependence on initial conditions\/}
(or, simply is {\it sensitive}) if there exists an $\ep \in \mu$ such that for every open nonempty subset
$O \subset X$ there exist $s \in S$ and $x,y \in O$ such that $(sx,sy) \notin \ep$.
\item Otherwise we say that $(S,X)$ is {\em non-sensitive\/}, NS for short.
This means that for every $\ep \in \mu$ there exists an
open nonempty subset $O$ of $X$ such that $sO$ is $\ep$-small in
$(X, \mu)$ for all $s \in S$.
\end{enumerate}
\end{defin}

The following definition (for continuous group actions) originated in \cite{GM1}.

\begin{defin} \label{d:HNS1} \cite{GM1,GM-AffComp}
We say that a compact $S$-system $X$ is \emph{hereditarily non-sensitive} (HNS, in short)
if for every closed nonempty subset $A \subset X$ and for every entourage $\eps$ from the unique compatible uniformity on $X$ there exists
an open subset $O$ of $X$ such that $A \cap O$ is nonempty and $s (A \cap O)$ is $\eps$-small for every $s \in S$.
\end{defin}

\begin{thm} \label{t:HNSenv} \
\begin{enumerate}
\item \cite{GM-rose}
A dynamical system $(S,X)$ is HNS iff $E(S,X)$
(equivalently, $\tilde{S}$), as an $S$-invariant family,
 is fragmented.
\item \cite{GMU}
A metric dynamical system $(S,X)$ is HNS iff $E(S,X)$ is metrizable.
\end{enumerate}
\end{thm}

\sk

The second class of dynamical systems we will be interested in is the class of tame dynamical systems.
For the history of this notion, which is originally due to K\"{o}hler \cite{Ko}, we refer to
\cite{Gl-tame} and \cite{GM-rose}.
The following principal result is a dynamical analog
of a well known BFT dichotomy \cite{BFT,TodBook}.

\begin{thm} \label{D-BFT}
\cite{Ko, GM1, GM-AffComp} \emph{(A dynamical version of BFT dichotomy)}
Let $X$ be a compact metric dynamical $S$-system and let $E=E(X)$ be its
enveloping semigroup. We have the following alternative. Either
\begin{enumerate}
\item
$E$ is a separable Rosenthal compact (hence $E$ is Fr\'echet and ${card} \; {E} \leq
2^{\aleph_0}$); or
\item
the compact space $E$ contains a homeomorphic
copy of $\beta\N$ (hence ${card} \; {E} = 2^{2^{\aleph_0}}$).
\end{enumerate}
The first possibility
holds iff $X$ is a tame $S$-system.
\end{thm}

Thus, a metrizable dynamical system is tame iff
$\card(E(X)) = 2^{\aleph_0}$ iff $E(X)$ is a Rosenthal compactum (or a Fr\'echet space).
Moreover, by \cite{GMU} a metric $S$-system is tame iff every $p \in E(X)$
is a Baire class 1 map $X \to X$.
This result led us to the following definition for general (not necessarily, metrizable) systems.

\begin{defin} \label{d:tame} \cite{GM-rose, GM-AffComp}
A compact $S$-system $X$ is said to be \emph{tame} if every $p \in E(X)$ is a fragmented map
(equivalently, Baire 1, when $X$ is metrizable).
\end{defin}

There are several other well known characterizations of tameness and in the present work we will
obtain two more:  by Theorem \ref{t:tame-f} $(S,X)$ is tame iff
$fS$ is an eventually fragmented family for every $f \in C(X)$, and, in Theorem \ref{t:tame}
we show that $(S,X)$ is tame iff
$\tilde{S} \subset X^X$,
 as an $S$-invariant family, is eventually weakly fragmented (see Definitions \ref{d:sens-fr-f} and \ref{d:E-HNS}).

Note that, as it directly follows from the definitions (when considering the enveloping semigroup characterizations),
 every WAP system is HNS and every HNS is tame.

\sk

\subsection{Some classes of functions}

A \emph{compactification} of $X$ is a pair $(\nu,Y)$ where $Y$ is a
compact (Hausdorff, by our convention) 
space and $\nu: X \to Y$ is a continuous map with a dense range. 
When $X$ and $Y$ are $S$-spaces and $\nu$ is an $S$-map we say that $\nu$ is 
an \emph{$S$-compactification}.

\begin{defin} \label{d:funct} 
Let $X$ be a (not necessarily, compact) $S$-system and let $f \in C(X)$.
\ben
\item
We say that $f$ \emph{comes} from the $S$-compactification $q: X \to Y$ 
if there exists a continuous function $f': Y \to \R$ such that $f=f' \circ q$.
 \item
 We say that $f \in C(X)$ is \emph{RMC} ({\it right multiplicatively continuous}) if
$f$ comes from some $S$-compactification $q: X \to Y$.
For every compact $S$-system $X$ we have $\RMC(X) = C(X)$.
\item
If we consider only jointly continuous $S$-actions on $Y$ then the
functions $f: X \to \R$ which come from such $G$-compactifications
$q: X \to Y$ are
{\it right uniformly continuous}. Notation: $f \in \RUC(X)$.

\item
$f$ is said to be: a)
 \emph{WAP}; b) \emph{Asplund}; c)  \emph{tame}
if  $f$ comes from an $S$-compactification $q: X \to Y$
 such that $(S,Y)$ is: WAP, HNS or tame respectively.
For the corresponding classes of functions we use the notation:
$\WAP(X), \Asp(X), \Tame(X)$, respectively. Each of these is a
norm closed
$S$-invariant subalgebra of the $S$-algebra
$\RMC(X) \subset C(X)$
and
$$ \WAP(X) \subset \Asp(X) \subset \Tame(X).$$
For more details see \cite{GM-AffComp,GM-survey}.
\item
Note that as a particular case of (3) we have defined the algebras
$\WAP(S), \Asp(S), \Tame(S)$ corresponding to the left action of $S$ on $X:=S$.
\een
\end{defin}

Below we give also 
characterizations of tame and Asplund functions in terms of tame and fragmented families, respectively. 
See Theorems \ref{l:functCases}, \ref{t:AspF} and \ref{t:tame-f}.

\begin{defin} \label{d:cyclic} \cite{GM1,GM-AffComp}
We say that a compact dynamical $S$-system $X$ is \emph{cyclic} if there exists $f \in C(X)$ such that $(S,X)$
is topologically $S$-isomorphic to the Gelfand space $X_f$ of the $S$-invariant unital subalgebra $\A_f \subset C(X)$ generated by the orbit $fS$.
\end{defin}

\begin{remark} \label{r:cycl-comes}
Let $X$ be a (not necessarily compact) $S$-system and $f \in \RMC(X)$.
Then, as was shown in \cite{GM-AffComp}, there exist:
a cyclic $S$-system $X_f$, a continuous $S$-compactification $\pi_f: X \to X_f$,
and a continuous function $\tilde{f}: X_f \to \R$ such that
$f=\tilde{f} \circ \pi_f$;
that is, $f$ comes from the $S$-compactification $\pi_f: X \to X_f$.
The collection of functions $\tilde{f} S$ separates points of $X_f$.
Finally, $f \in \RUC(X)$ iff the action of $S$ on $X_f$ is jointly continuous.
\end{remark}

\section{Sensitivity and fragmentability of families}

The following topological definitions
were motivated by
the notions of sensitivity and tameness in topological dynamics (Definitions \ref{d:HNS0} and \ref{d:tame} above),
as well as by the fragmentability concepts which come from Banach space theory.

\begin{defin} \label{d:sens-fr-f}
Let $(X,\tau)$ be a topological space, $(Y,\xi)$ a uniform space and $\eps \in \xi$ is an entourage.
We say that a family of (not necessarily continuous) functions $F=\{f_i: X \to Y\}_{i \in I}$ is:
\bit
\item [($\varepsilon$-NS)]
$\eps$-\emph{Non-Sensitive} ($\eps$-NS in short) if there exists a non-void open subset $O$ in $X$
such that $f_i(O)$ is $\eps$-small for every $i \in I$.
\item [(NS)]
\emph{Non-Sensitive}
if $F$ is $\eps$-NS for every $\eps \in \xi$.
\item [(E-NS)]
{\em Eventually Non-Sensitive}
if for every
infinite subfamily $L \subset F$ and every $\eps \in \xi$
there exists an
infinite subfamily $K \subset L$ which is $\eps$-NS.

\sk

\item [($\eps$-Fr)]
$\eps$-\emph{Fragmented} if for every nonempty 
(closed) subset $A$ of $X$ the restriction
 $F|_A:=\{f|_A: A \to Y\}_{f \in F}$ is an $\eps$-NS family.
\item [(Fr)=(HNS)]
\emph{Fragmented}, or \emph{Hereditarily Non-Sensitive},
if $F$ is
$\eps$-Fr
for every $\eps \in \xi$.
\item [(E-wFr)]
{\em Eventually weakly fragmented}
if for every
infinite subfamily $L \subset F$ and every $\eps \in \xi$
there exists an
infinite
$\eps$-Fr
subfamily $K \subset L$.
 \item [(E-Fr)] {\em Eventually Fragmented} if for every
 infinite subfamily $L \subset F$ there exists an
 infinite fragmented subfamily $K \subset L$.

\sk

\item [(HAE)]
\emph{Hereditarily Almost Equicontinuous}, or \emph{barely continuous},
if for every nonempty closed subset $A$ of $X$ the family $F|_A$ has a point of equicontinuity.
\item [(E-HAE)]
{\em Eventually Hereditarily Almost Equicontinuous}
if for every infinite subfamily $L \subset F$ there exists an
infinite HAE subfamily $K \subset L$.
 \eit
\end{defin}

When the family $F=\{f\}$ consists a single function we
retrieve the definitions of NS,
fragmented, and barely continuous functions.
The definition of a fragmented function (as in \cite{JOPV})
is a slight generalization of the original one 
(for the identity function $f:=id_X: (X,\tau) \to (X,d)$ with a pseudometric $d$ on $X$;
the \emph{$(\tau,d)$-fragmentability}) 
which is due to Jayne and Rogers from Banach space theory.
It appears implicitly in a work of Namioka and Phelps \cite{NP} which provides
a characterization of Asplund Banach spaces $V$ in terms of
(weak$^*$,norm)-fragmentability. The set of fragmented maps from $X$ into $Y:=\R$
is denoted by ${\mathcal F}(X)$.
See \cite{N, me-fr, Me-nz, GM1, GM-rose} for more details.
Barely continuous maps are well known also as the maps with the \emph{point of continuity property} (i.e., for every closed
nonempty $A \subset X$ the restriction $f_{|A}: A \to Y$ has a continuity point)

Eventually fragmented families were introduced in \cite{GM-rose}, where they yield a new
characterization of Rosenthal Banach spaces (see below Theorem \ref{f:RosFr}.4).

In Example \ref{e:DistEx} we present some simple examples 
illustrating the definitions of families of
functions which are (or are not) fragmented, eventually fragmented, or satisfy DLP.

For some applications of the fragmentability concept
for topological transformation groups, see~\cite{me-fr, Me-op, Me-nz, GM1, GM-fp, GM-rose, GM-AffComp}.
For other research directions involving fragmentability see for example \cite{KM}.

\begin{lem} \label{r:fr1} \cite{GM1, GM-rose}
\ben
\item (Fr), (E-Fr) and (E-wFr) in Definition \ref{d:sens-fr-f} 
are hereditary conditions. 
It is enough to check these conditions only for $\ep \in \gamma$ from 
a {\it subbase} $\gamma$ of $\xi$ and for closed nonempty subsets $A \subset X$. 
\item
If $X$ is Polish and $Y$ is a separable metric space then
$f: X \to Y$ is fragmented iff $f$ is a Baire class 1 function (i.e., the inverse image of every open set is
$F_\sigma$).
\item
When $X$ is hereditarily Baire
and $(Y,\rho)$ is a pseudometric space then $f: X \to Y$ is fragmented iff $f$ has the
point of continuity property.
\item
A topological space $(X,\tau)$ is {\em scattered\/} (i.e., every
nonempty subspace has an isolated point) iff $X$ is
$(\tau,\xi)$-fragmented, for arbitrary uniform structure $\xi$ on the \emph{set} $X$.
\item Let $(X,\tau)$ be a separable metrizable space and $(Y,\rho)$ a
pseudometric space. Suppose that $f: X \to Y$ is a fragmented onto map. Then $Y$ is separable.
\item $F=\{f_i: X \to (Y,\xi)\}_{i \in I}$ is a fragmented family iff
the induced map $X \to (Y^F, \xi_U)$ is fragmented,
where $\xi_U$ is the uniformity of uniform convergence on $Y^F$. \item
Let $\a: X \to X'$ be a continuous onto map between compact
spaces. Assume that $(Y, \xi)$ is a uniform space, $F:=\{f_i: X \to Y\}_{i \in I}$
and $F':=\{f'_i: X' \to Y\}_{i \in I}$ are families such that $f'_i \circ \a=f_i$ for every $i \in I$. Then $F$ is a fragmented family iff $F'$ is a fragmented family.
\item
(See 
\cite[Cor. 1D]{BFT} or \cite[Lemma 3.7]{Dulst})
Let $X$ be a hereditarily Baire space and $f: X \to \R$ an arbitrary function. \TFAE
\ben
\item $f$ has the point of continuity property (equivalently: fragmented, by (3)).
\item For every nonempty closed subset $K \subset X$ and $a<b$ in $Y$ the sets
$K \cap \{f \leq a\}$, $K \cap \{f \geq b\}$
are not both dense in $K$.
\een
\item
Let $p: X \to Y$ be a map from a topological space $X$ into a
compact space $Y$. Suppose that $\{f_i: Y \to Z_i\}_{i \in I}$ is
a system of continuous maps from $Y$ into Hausdorff uniform spaces
$Z_i$ such that it separates points of $Y$ and $f_i \circ p \in
{\mathcal F}(X,Z_i)$ for every $i \in I$. Then $p \in {\mathcal F}(X,Y)$.
\item \cite[Lemma 2.3.4]{GM-rose} 
Let $(X,\tau)$ and $(X',\tau')$ be compact spaces, and let
$(Y, \mu)$ and $(Y', \mu')$ be uniform spaces. Suppose that:
$\a: X \to X'$ is a continuous onto map, $\nu: (Y, \mu) \to
(Y', \mu')$ is uniformly continuous, $\phi: X \to Y$ and
$\phi': X' \to Y'$ are maps such that the following diagram

\begin{equation*}
\xymatrix {
	(X, \tau) \ar[d]_{\a} \ar[r]^{\phi} & (Y, \mu)
	\ar[d]^{\nu} \\
	(X', \tau') \ar[r]^{\phi'} & (Y', \mu') }
\end{equation*}
commutes. If $X$ is fragmented by $\phi$ then $X'$ is
fragmented by $\phi'$.
\een
\end{lem}
\delete{\begin{proof} 	
		(1) 
		Let $F=\{f_i: X \to Y\}_{i \in I}$ be a family of functions, $\mu_B \subset \mu$ is a base of $\mu$, $\emptyset \neq A \subset X$. Suppose that $F$ is $\eps$-NS for every $\eps \in \mu_B$. We have to show that 
		$F$ is $\eps$-NS for every $\eps \in \mu$. It is enough to show that 
		$F$ is $(\eps_1 \cap \eps_2)$-NS whenever $F$ is $\eps_i$-NS for every $i \in \{0,1\}$. 
		There exists an open subset $O_1 \subset X$ such that 
		$A \cap O_1 \neq \emptyset$ and $f(A \cap O_1)$ is $\eps_1$-small for every $f \in F$. Now for $A \cap O_1$ we can choose an open subset $O_2 \subset X$ such that $(A \cap O_1) \cap O_2$ is nonempty and $f(A \cap O_1 \cap O_2)$ is $\eps_2$-small. Then $f(A \cap (O_1 \cap O_2))$ is $\eps$-small.   
		
	\end{proof}
	}

\begin{lem} \label{l:propFamil}
{} \

\ben
\item Always, HAE $\subset$ HNS $\subset$ NS and E-HAE $\subset$ E-Fr $\subset$
E-wFr $\subset$ E-NS.
\item
In the definitions (E-Fr), (E-HAE), (E-HAE) and (E-NS)
one can assume that the infinite sets $K$ (and $L$) are \emph{countable}.
\item
When $(Y,\xi)$ is a pseudometric uniformity and every $f_i \in F$ is a fragmented map,
then E-Fr=E-wFr.
\item
When $X$ is a hereditarily Baire space and $(Y,\xi)$ is a pseudometric uniformity then
HAE $=$ Fr.
\item
If, in addition
to the conditions in (4),
every $f_i \in F$ is a fragmented map then
E-HAE $=$ E-Fr $=$ E-wFr.
\een
\end{lem}
\begin{proof}
(1) and (2) are trivial.

To get (3) use a diagonal argument. This is possible because the pseudometric uniformity $\xi$
has a countable basis for the uniform structure.
We have a decreasing sequence of entourages
$\eps_1 \supset \eps_2 \supset \cdots$ which
form a basis of the uniform structure $\xi$.
Using the E-wFr condition for a given infinite sequence in $F$ we extract a subsequence which is $\eps_1$-Fr subsequence.
Now for this subsequence one can extract an $\eps_2$-Fr subsequence, and so on.
Consider the diagonal sequence. This will be an $\eps$-Fr family for all $\eps \in \xi$.
In the verification it is important to note that, by our assumption, every individual $f_i \in F$ is a fragmented map. This guarantees that every finite subfamily is fragmented. So, in particular, each initial finite segment of the diagonal sequence is fragmented.
Finally note that the union of two fragmented
families is a fragmented family.

For (4) we recall some observations from \cite{GM1}.
By Lemma \ref{r:fr1}.6,
a family $F=\{f_i: X \to (Y,\xi)\}_{i \in I}$ is a fragmented family iff
the induced map $X \to (Y^F, \xi_U)$ is fragmented.
Now, when $X$ is hereditarily Baire and $\xi$ is pseudometrizable we obtain, using Lemma \ref{r:fr1}.3,
that $F$ is fragmented iff $F$ is HAE.

For (5) combine parts (3) and (4) to get  E-HAE $=$ E-wFr.
\end{proof}

Of course the condition that $(Y,\xi)$ be a pseudometric uniform space is satisfied when $Y = \R$,
so that this assumption is automatically fulfilled for families in $C(X)$.

\begin{lem} \label{l:FrFa} \cite{GM-rose, GM-AffComp}
\begin{enumerate}
\item
Suppose $F$ is a compact space, $X$ is \v{C}ech-complete,
$Y$ is a uniform space
and we are given a separately continuous map $w: F \times X \to Y$.
Then the naturally associated family
$\tilde{F}:=\{\tilde{f}: X \to Y\}_{f \in F}$ (where $\tilde{f}(x) = w(f,x)$) is fragmented.
\item
Suppose $F$ is a compact metrizable space, $X$ a hereditarily Baire space,
(e.g., \v{C}ech-complete, compact or Polish),
and $M$ separable
and metrizable.
Assume that we are given a map $w: F\times X \to M$ such that (i)
$\tilde{x}: F \to M, f \mapsto w(f,x)$ is continuous for every $x \in X$, and
(ii) $\tilde{f}: X \to M, x \mapsto w(f,x)$ is continuous for
every $f \in Y$ for a dense subset $Y$ of $F$.
Then the family $\tilde{F}$ is
HAE
(hence, fragmented).
\item
\emph{(A version of Osgood's theorem)}
Let $f_n: X \to \R$ be a pointwise convergent sequence of continuous functions on a
hereditarily Baire space $X$.
Then $\{f_n\}_{n \in \N}$ is a fragmented family.
\end{enumerate}
\end{lem}

\begin{proof}
(1): There exists a collection of uniformly continuous maps $\{\varphi_i: Y \to M_i\}_{i \in I}$
into metrizable uniform spaces $M_i$ which generates the uniformity on $Y$.
Now for every closed subset $A \subset X$ apply Namioka's joint continuity theorem 
to the separately continuous map
$\varphi_i \circ w: F \times A \to M_i$ and take into account Lemma \ref{r:fr1}.1.

(2): Since every $\tilde{x}: F \to M$ is continuous, the natural map $j: X
\to C(F, M), \ j(x)=\tilde{x}$ is well defined.
By assumption every closed nonempty subset $A \subset X$ is Baire.
By \cite[Proposition 2.4]{GMU}, 
$j |_{A}: A \to
C(F, M)$ has a point of continuity, where $C(F,M)$ carries the
sup-metric.
Hence, $\tilde{F}_{A} =\{\tilde{f}
\upharpoonright_{A}: A \to M \}_{f \in F}$ is equicontinuous at
some point $a\in A$. This implies 
that the family $\tilde{F}$ is HAE.

(3) : Follows from (2) applied to the evaluation map $w : F \times X \to \R$, where
$F:=\{f\} \cup \{f_n: n \in \N\} \subset \R^X$ with $f:=\lim f_n$, the pointwise limit.
\end{proof}

For other properties of fragmented maps and fragmented families
we refer to \cite{Me-nz, GM1, GM-rose}.

\subsection{Sensitivity conditions in dynamical systems}

\begin{defin} \label{d:E-HNS}
Let $(X,\tau)$ be a compact $S$-dynamical system endowed with its unique compatible uniform structure $\xi$.
The set of translations
$\tilde{S}$ can be treated as a family of functions
$(X,\tau) \to (X,\xi)$. We say that the $S$-system $X$ is
NS, HNS, HAE, E-wFr, E-Fr E-HAE whenever the family $\tilde{S}$
has the same
property in the sense of Definition \ref{d:sens-fr-f}.
\end{defin}

It is easy to see that 
this definition, in the case of HNS, 
coincides with the class described in Definition \ref{d:HNS1}.
We will see below that tameness  (Definition \ref{d:tame}) is equivalent to E-wFr, and to E-Fr if $X$, in addition, is metrizable.

In the list of classes in Definitions \ref{d:E-HNS} and \ref{d:sens-fr-f}
the 
most important for the present work are the 
classes: 
HNS=Fr, E-Fr and E-wFr. 

\subsection{Fragmentability and Banach spaces}

\subsubsection{Asplund Banach spaces}

Recall that a Banach space $V$ is an  {\it Asplund space}\index{Asplund space}\index{Banach space!Asplund} if the dual of every
separable linear subspace is separable.

In the following result the equivalence of (1), (2) and (3) is well known
and (4) is a reformulation of (3) in terms of fragmented families.

\begin{thm} \label{f:Asp} \emph{\cite{NP, N}}
Let $V$ be a Banach space. The following conditions are
equivalent:
\ben
\item $V$ is an Asplund space.
\item $V^*$ has the Radon-Nikod\'ym property.
\item
Every bounded subset $A$ of the dual $V^*$ is (weak${}^*$,norm)-fragmented.
\item $B$ 
is a fragmented family of real valued maps on the compactum $B^*$.
\een
\end{thm}

For separable $V$, the assertion (4) can be derived from Lemma \ref{l:FrFa}.2.


Reflexive spaces and spaces of the type
$c_0(A)$ are Asplund.
By \cite{NP} the Banach space $C(K)$ for compact $K$ is Asplund iff $K$ is a scattered compactum
(see also Lemma \ref{r:fr1}.4).
Namioka's joint continuity theorem implies that
every weakly compact set in a Banach space is norm fragmented, \cite{N}. This
explains why every reflexive space is Asplund.

\subsubsection{Banach spaces not containing $l_1$}
\label{s:ind}

\begin{defin}\label{d:l1}
Let $f_n: X \to \R$ be a uniformly bounded sequence of functions on a \emph{set} $X$. Following Rosenthal we say that
this sequence is an \emph{$l_1$-sequence} on $X$ if there exists a real constant $a >0$
such that for all $n \in \N$ and choices of real scalars $c_1, \dots, c_n$ we have
$$
a \cdot \sum_{i=1}^n |c_i| \leq ||\sum_{i=1}^n c_i f_i||.
$$
\end{defin}


For every $l_1$-sequence $f_n$ the closed linear span in $l_{\infty}(X)$
is linearly homeomorphic to the Banach space $l_1$.
In fact, in this case the map $l_1 \to l_{\infty}(X), \ (c_n) \to \sum_{n \in \N} c_nf_n$ is a linear homeomorphic embedding.

\begin{defin}\label{d:ind}
A sequence $f_n$ of
real valued functions on
a set
$X$ is said to be \emph{independent} if
there exist real numbers $a < b$ such that
$$
\bigcap_{n \in P} f_n^{-1}(-\infty,a) \cap  \bigcap_{n \in M} f_n^{-1}(b,\infty) \neq \emptyset
$$
for all finite disjoint subsets $P, M$ of $\N$.
\end{defin}

Clearly every subsequence of an independent sequence is again independent.

\begin{defin} \label{d:tameF} 
Let us say that a family $F$ of real valued (not necessarily, continuous)  functions on a set $X$ is {\it tame} if $F$ does not contain an independent sequence.
\end{defin}
A word of warning is in place here. In Definition \ref{d:funct} (4) the notion of a tame function 
$f : X \to \R$ was introduced. 
Note that when a function $f$ is tame in the sense of Definition \ref{d:funct} (4) this, of course, 
does not mean that the singleton family $\{f\}$ is tame.
However, as we will see (Theorem \ref{l:functCases}.3) it is true that, when $X$ is a compact $S$-system,
$f : X \to \R$ is a tame function iff the family 
$fS=\{f \circ s : s \in S\}$ is a tame family.


\sk

\begin{defin}
A Banach space $V$ is said to be {\em Rosenthal} if it does
not contain an isomorphic copy of $l_1$.
\end{defin}

Every Asplund space is Rosenthal (because $l_1^*= l_{\infty}$ is nonseparable).

\sk

\begin{defin} \label{d:Ros-F} \cite{GM-rose}
Let $X$ be a topological space. We say that a subset $F\subset
C(X)$ is a \emph{Rosenthal family}
(for $X$) if $F$ is norm bounded and
the pointwise closure ${\cls_p}(F)$ of $F$ in $\R^X$ consists of fragmented maps,
that is,
${\cls_p}(F) \subset {\mathcal F}(X).$
\end{defin}

The following useful result synthesizes some known results.
It is based on results of Rosenthal \cite{Ro}, Talagrand \cite[Theorem 14.1.7]{Tal} and van Dulst \cite{Dulst}.
In \cite[Prop. 4.6]{GM-rose} we show why eventual fragmentability of $F$ can be included in this list.

\begin{thm} \label{f:sub-fr}
Let $X$ be a compact space and $F \subset C(X)$ a bounded subset.
The following conditions are equivalent:
\begin{enumerate}
\item  $F$ is a tame family. 
\item
$F$ does not contain a subsequence equivalent to the unit basis of $l_1$.
\item
Each sequence in $F$ has a pointwise convergent subsequence in $\R^X$.
\item
$F$ is a Rosenthal family for $X$.
\item $F$ is an eventually fragmented family.
\end{enumerate}
\end{thm}

Note that the compactness in Theorem \ref{f:sub-fr} is essential even for Polish spaces $X$. 

We will also need some characterizations of Rosenthal spaces.

\begin{thm} \label{f:RosFr}
Let $V$ be a Banach space. The following conditions are
equivalent:
\begin{enumerate}
\item
$V$ is a Rosenthal Banach space.
\item \emph{(E. Saab and P. Saab \cite{SS})}
Each $x^{**} \in V^{**}$ is a fragmented map when restricted to the
weak${}^*$ compact ball $B^*$. Equivalently, $B^{**} \subset \F(B^*)$.
\item $B$ is a \emph{Rosenthal family} for the weak$^*$ compact unit ball $B^{*}$.
\item $B$ is an eventually fragmented family of maps on $B^{*}$.
\end{enumerate}
\end{thm}

Condition (2) is a reformulation (in terms of fragmented maps) of a
criterion from \cite{SS} which was originally stated in terms of the point of continuity property.
The equivalence of (1), (3) and (4) follows from Theorem \ref{f:sub-fr}.

\subsection{More properties of fragmented families}

Here we demonstrate a general principle:
the fragmentability of a family of continuous maps defined on a compact space is ``countably-determined".
 The following theorem 
is inspired by results of Namioka and it can be deduced after some reformulations from \cite[Theorems 3.4 and 3.6]{N}. See also \cite[Theorem 2.1]{CNO}.

\begin{thm} \label{t:countDetermined}
Let $F=\{f_i: X \to Y\}_{i \in I}$ be a
bounded
family of \textbf{continuous} maps from a compact (not necessarily metrizable)
space $(X,\tau)$ into a pseudometric space $(Y,d)$.
The following conditions are equivalent:
\ben
\item
$F$ is a fragmented family of functions on $X$.
\item
Every \emph{countable} subfamily $K$ of $F$ is fragmented.
\item
For every countable subfamily $K$ of $F$
the pseudometric space $(X,\rho_{K,d})$ is separable,
where
$$
\rho_{K,d}(x_1,x_2):=\sup_{f \in K} d(f(x_1),f(x_2)).
$$
\een
\end{thm}
\begin{proof}
(1) $\Rightarrow$ (2) is trivial.

(2) $\Rightarrow$ (3):
Let $K$ be a countable subfamily of $F$.
Consider the natural map $$\pi: X \to Y^K, \pi(x)(f):=f(x).$$
By (2), $K$ is a fragmented family.
Thus by Lemma \ref{r:fr1}.6 the map $\pi$ is $(\tau,\mu_K)$-fragmented, where $\mu_K$ is the uniformity
of $d$-uniform convergence on $Y^K:=\{f: K \to (Y,d)\}$.
Then the map
$\pi$ is also $(\tau,d_K)$-fragmented, where $d_K$ is the pseudometric on $Y^K$ defined by
$$d_K(z_1,z_2):=\sup_{f \in K} d(z_1(f), z_2(f)).$$
Since $d$ is bounded, $d_K(z_1,z_2)$ is finite and $d_K$ is well-defined.
Denote by $(X_K,\tau_p)$ the subspace $\pi(X) \subset Y^K$ in pointwise topology.
Since $K \subset C(X)$, the induced map $\pi_0: X \to X_K$ is a continuous
map onto the compact space $(X_K, \tau_p)$. Denote by $i: (X_K,\tau_p) \to (Y^K,d_K)$ the inclusion map.
So, $\pi=i \circ \pi_0$, where the map $\pi$ is $(\tau,d_K)$-fragmented.
Then by Lemma \ref{r:fr1}.7 
we obtain that $i$ is $(\tau_p,d_K)$-fragmented. It immediately follows that the identity map
$id: (X_K,\tau_p) \to (X_K,d_K)$ is $(\tau_p,d_K)$-fragmented.

Since $K$ is countable, $(X_K, \tau_p) \subset Y^K$ is metrizable.
Therefore, $(X_K, \tau_p)$ is second countable (being a metrizable compactum). Now, since $d_K$ is a pseudometric on $Y^K$,
and $id: (X_K,\tau_p) \to (X_K,d_K)$ is $(\tau_p,d_K)$-fragmented, we can apply Lemma \ref{r:fr1}.5. 
It directly implies that the set $X_K$ is a separable subset of $(Y^K, d_K)$.
This means that $(X,\rho_{K,d})$ is separable.

(3) $\Rightarrow$ (1) :
Suppose that $F$ is not fragmented.
Thus, there exists a non-empty closed subset $A \subset X$ and an $\eps >0$ such that
for each non-empty open subset $O \subset X$ with $O \cap A \neq \emptyset$ there is
some $f \in F$ such that
$f(O \cap A)$ is not $\eps$-small in $(Y,d)$.
Let $V_1$ be an arbitrary non-empty relatively open subset in $A$. There are
$a,b \in V_1$ and
$f_1 \in F$ such that $d(f_1(a),f_1(b)) > \eps$. Since $f_1$ is continuous we can choose
relatively open subsets $V_2,V_3$ in $A$
with $\cls (V_2 \cup V_3) \subset V_1$
such that $d(f_1(x),f_1(y)) > \eps$ for every $(x,y) \in V_2 \times V_3$.

By induction we can construct a sequence $\{V_n\}_{n \in \N}$ of non-empty relatively open subsets in $A$ and a sequence $K:=\{f_n\}_{n \in \N}$ in $F$ such that:

\bit
\item [(i)]
$V_{2n} \cup V_{2n+1} \subset V_n$
for each $n \in \N$;
\item [(ii)] $d(f_n(x),f_n(y)) > \eps$ for every $(x,y) \in V_{2n} \times V_{2n+1}$.

\eit

\sk

We claim that $(X,\rho_{K,d})$  is not separable,
where $$\rho_{K,d}(x_1,x_2):=\sup_{f \in K} d(f(x_1),f(x_2)).$$

In fact, for each \emph{branch}
$$
\a:=V_1 \supset V_{n_1} \supset V_{n_2} \supset \cdots
$$
where for each $i, n_{i+1}=2n_i$ or $2n_i+1$, by compactness of $X$ one can choose an element
$$
x_{\a} \in \bigcap_{i \in \N} {\cls}(V_{n_i}).
$$
If $x=x_\alpha$ and $y=x_\beta$ come from different branches,
then there is an $n \in \N$ such that
$x \in {\cls}(V_{2n})$ and
$y \in {\cls}(V_{2n+1})$ or (vice versa).
In any case it follows from (ii) and the continuity of $f_n$ that $d(f_n(x),f_n(y)) \geq \eps$,
hence $\rho_{K,d}(x,y) \geq \eps$. Since there are uncountably many branches we conclude
that $A$ and hence also $X$ are not $\rho_{K,d}$-separable.
\end{proof}

\begin{defin} \label{d:AspSet}
\cite{Fa, Me-nz}
Let $X$ be a compact space and $F \subset C(X)$ a norm bounded family of continuous real valued functions on $X$. Then $F$ is said to be an
\emph{Asplund family for} $X$ if for every countable subfamily $K$ of $F$
the pseudometric space $(X,\rho_{K,d})$ is separable,
where
$$
\rho_{K,d}(x_1,x_2):=\sup_{f \in K} |f(x_1) - f(x_2)|.
$$
\end{defin}


\begin{cor} \label{c:AspSet}
Let $X$ be a compact space and $F \subset C(X)$ a norm bounded family of continuous real valued functions on $X$. Then $F$ is fragmented if and only if $F$ is an Asplund family for $X$.
\end{cor}

\begin{thm} \label{c:countDetermined2}
Let $F=\{f_i: X \to Y\}_{i \in I}$ be a family of continuous maps
from a compact (not necessarily metrizable) space $(X,\tau)$ into a uniform space $(Y,\mu)$.
Then $F$ is fragmented if and only if every countable subfamily $A \subset F$ is fragmented.
\end{thm}
\begin{proof} The proof can be reduced to Theorem \ref{t:countDetermined}.
Every uniform space can be uniformly approximated by pseudometric spaces.
Using Lemma \ref{r:fr1}.1
we can suppose that $(Y,\mu)$ is pseudometrizable; i.e.
there exists a pseudometric $d$ such that $\mathrm{unif}(d)=\mu$.
Moreover, replacing $d$ by the uniformly equivalent pseudometric $\frac{d}{1+d}$ we can suppose that $d \leq 1$.
\end{proof}

\section{Banach representations of dynamical systems and of functions}
\label{s:repr}

A \emph{representation} of a semigroup $S$ (with identity element $e$) on a Banach space $V$
is a co-homomorphism $h: S \to \Theta(V)$, where
$\Theta(V):=\{T \in L(V): \ ||T|| \leq 1\}$,
with $h(e)=id_V$. Here $L(V)$ is the 
set of continuous
linear operators $V\to V$ and $id_V$ is the identity operator.
This is equivalent to the requirement that
$h: S \to \Theta(V)^{op}$ be a monoid homomorphism, where $\Theta(V)^{op}$
is the opposite semigroup of $\Theta(V)$.
If $S=G$ is a group then $h(G) \subset \Iso(V)$, where
$\Iso(V)$ is the group of all linear isometries from $V$ onto $V$.

Since $\Theta(V)^{op}$ acts from the right on $V$
and from the left on $V^*$
we sometimes write $vs$ for $h(s)(v)$ and $s \varphi$ for $h(s)^* (\varphi)$, where $h(s)^*: V^* \to V^*$ is the adjoint of $h(s): V \to V$.
Then $\langle vs, \varphi \rangle=\langle v, s\varphi  \rangle$. In this way we get the \emph{dual action} (induced by $h$)
$$S \times V^* \to V^*, \ (s \varphi)(v):=\varphi(vs)= \langle vs, \varphi \rangle.$$

\begin{defin} \label{d:repr}  \cite{Me-nz,GM1}
Let $X$ be
an $S$-space.
\ben
\item
A \emph{representation} of $(S,X)$ on a Banach space $V$ is a pair
$$h: S \to \Theta(V), \ \ \a: X \to V^*$$
where $h: S \to \Theta(V)$ is a weakly continuous representation (co-homomorphism) of semigroups
and $\a: X \to V^*$ is a weak$^*$ continuous bounded
$S$-mapping with respect to the dual action
$S \times V^* \to V^*, \ (s \varphi)(v):=\varphi(vs).$  

\begin{equation*} \label{diag2}
\xymatrix{ S \ar@<-2ex>[d]_{h} \times  X \ar@<2ex>[d]^{\a} \ar[r]  & X \ar[d]^{\a} \\
\Theta^{op} \times V^* \ar[r]  &  V^* }
\end{equation*}

We say that the representation $(h,\a)$ is \emph{strongly continuous} if
$h$ is strongly continuous. \emph{Faithful} will mean that $\a$ is a topological embedding.


\item In particular, if in (1) $S:=G$ is a group then,
necessarily, $h(G)$ is a subgroup of $\Iso(V)$.

\item
If $\K$ is a subclass of the class of Banach spaces, we say
that a dynamical system $(S,X)$ is
{\em (strongly) $\K$-representable}
if there exists a weakly (respectively, strongly) continuous
faithful representation of $(S,X)$ on a Banach space $V \in \K$.
\item
A dynamical system $(S,X)$ is said to be {\em (strongly)} \emph{$\K$-approximable}
if $(S,X)$ can be embedded in a product of (strongly) $\K$-representable $S$-spaces.
\item
For a topological group $G$ $\K$-\emph{representability} will mean that there exists
an embedding (equivalently, a co-embedding) of $G$ into the
group $\Iso(V)$ where $V \in \K$ and $\Iso(V)$ is endowed with the strong operator topology. 
\een
\end{defin}

Note that when $X$ is compact then every weak-star continuous $\a: X \to V^*$ is necessarily bounded.

\begin{remark}
The notion of a reflexively (Asplund) representable compact dynamical system is a dynamical version of
the purely topological notion of an \emph{Eberlein} (respectively,
\emph{Radon-Nikod\'ym} (RN)) compactum, in the sense of Amir and
Lindenstrauss (respectively, in the sense of Namioka).
\end{remark}

\begin{defin} \label{d:DSclasses}
We say that a dynamical system $(S,X)$ is: (i) {\it Eberlein} when it is reflexively representable;
(ii) Radon-Nikod\'ym (RN) when it is Asplund representable;
and, (iii) \emph{Weakly Radon-Nikod\'ym} (WRN), when it is Rosenthal representable.
\end{defin}

As a word of warning note that 
if $X$, as a compactum, satisfies one of the properties above (Eberlein, RN or WRN) this does not mean that $(S,X)$ has the same property.
In fact {\it every} compact metrizable space is obviously (uniformly) Eberlein as it can be embedded in a separable Hilbert space.
However,
it is not hard to find metric dynamical systems which distinguish the classes of dynamical systems mentioned above.

\begin{ex} \label{e:DistEx} \
\ben
\item Let $X=[0,1]$ be the unit interval. Consider the cascade $(\Z,X)$ generated by the homeomorphism $\s(x)=x^2$. Then $(\Z,X)$,  as a dynamical system, is RN and not Eberlein.
To see this observe that the pair of sequences $x_n=1-\frac{1}{n}$ in $X=[0,1]$
and $\s^m \in G$ with $\s^m(x)=x^{2^m}$ does not satisfy DLP. The corresponding limits are 0 and 1.
This means that $(\Z,X)$ is not Eberlein (Remark \ref{r:JCont2} and Theorem \ref{t:tame=R-repr}).
The enveloping semigroup $E(\Z,X)$ is metrizable, being homeomorphic to the two-point compactification of $\Z$.
Hence, by \cite{GMU}, $(\Z,X)$ is RN.
The sequence $\{\s^m: [0,1] \to [0,1]\}_{m \in \N}$ is a fragmented family which does not satisfy DLP.
\item The Sturmian symbolic dynamical system $(X,\s)$ from Example \ref{e:tameNOThns}.3 is WRN but not RN.
The sequence $\{\s^n: X \to X\}_{n \in \N}$ of (positive) iterations is an eventually fragmented but not a fragmented family.
\item The natural action of the Polish group $H_+[0,1]$ of increasing homeomorphisms of $[0,1]$ on $[0,1]$ is tame but not HNS; this system is WRN but not RN. 
See for example \cite{GM-AffComp} or Section \ref{s:orderOnI} below. The family $H_+[0,1]$, as a family  of functions, (or any dense subsequence
in $H_+[0,1]$)
is eventually fragmented (equivalently, tame) but not fragmented.
\item The Bernoulli shift $(\Z,\{0,1\}^{\Z})$ is not WRN (equivalently, not tame). In fact,
the enveloping semigroup of this system can be identified with
$\beta \Z$.
Now use the dynamical version of BFT dichotomy (Theorem \ref{D-BFT}).
Another way to see that the shift system is not tame is by the well known
fact (see for example \cite{TodBook}) that
the sequence of projections on the Cantor cube $$\{\pi_m: \{0,1\}^{\N} \to \{0,1\}\}_{m \in \N}$$ is independent.
Hence by Theorem \ref{f:sub-fr} this family fails to be eventually fragmented.
\een
\end{ex}

\sk

Note that the classes from Definition \ref{d:DSclasses} are closed under countable products (see \cite{GM-AffComp}).

For a compact space $X$ we denote by $H(X)$ the topological group of
self-homeomorphisms of $X$ endowed with the compact open topology.
%

\begin{lem} \label{GrRepr}
Let $X$ be a compact $G$-space, where $G$ is a topological subgroup of $H(X)$. Assume that $(h,\a)$
is a
strongly continuous
faithful representation of $(G,X)$ on a Banach space $V$
(that is, $h: G \to \Iso(V)$ is strongly continuous and $\a: X \to (V^*,w^*)$ is an embedding, see Definition \ref{d:repr}).
Then $h: G \to \Iso(V)$ is a topological group embedding.
\end{lem}
\begin{proof}
Recall that the strong operator topology on $\Iso(V)^{op}$ is identical with the compact open topology inherited from the action of this group on the weak-star compact unit ball $(B^*,w^*)$.
\end{proof}

We recall the following theorems.

\begin{thm} \label{t:tame=R-repr}
\emph{(\cite[Theorem 11]{GM-AffComp} and \cite{Me-nz})}
Let $X$ be a compact $S$-system.
\ben
\item $(S,X)$ is a tame system
iff $(S,X)$ is Rosenthal-approximable.
\item $(S,X)$ is a HNS system
iff $(S,X)$ is Asplund-approximable.
\item $(S,X)$ is a WAP system
iff $(S,X)$ is reflexively-approximable. 
\een
(*) If $X$ is metrizable then in (1), (2) and (3) ``approximable" can be replaced by ``representable" and
the corresponding Banach space can be assumed to be separable.
\end{thm}

\begin{remark} \label{r:JCont1} \
If the given action $S \times X \to X$ is jointly continuous then the representations in Theorem
\ref{t:tame=R-repr} can be assumed to be strongly continuous.
The same is true for the results below: Theorems \ref{t:AspDuality},
\ref{t:WRN} \ref{t:AspF}, \ref{t:HNS}, \ref{t:tame-f}, \ref{t:tame} and \ref{t:OrderedAreTame}.1.
 \end{remark}

Since every reflexive Banach space is Asplund and every Asplund Banach space is Rosenthal, 
Theorem \ref{t:tame=R-repr} immediately implies that every WAP system is HNS and every HNS system is tame. 

\br

Of course not every $\K$-approximable system  is $\K$-representable.
For example, $(S,X)$ with $S:=\{e\}$ and $X:=[0,1]^{\R}$
is clearly reflexively-approximable but not reflexively-representable (because $X$, as a compactum, is not Eberlein).

\begin{remark} \label{r:subrepresentations} 
	(subrepresentations) 
	\ben 
	\item Let $(h,\a)$ be a representation of an $S$-space $X$ on $V$. For every $S$-invariant closed linear subspace $V_0 \subset V$ we have a natural induced representation $(h_0,\a_0)$ on $V_0$. Indeed, $h_0: S \to \Theta(V_0)^{op}$ is uniquely defined using the induced action $V_0 \times S \to V_0$. Define $\a_0$ as the composition $\a_0:=i^* \circ \a: X \to V_0^*$, where $i^*: V^* \to V_0^*$ is the adjoint of the embedding $i: V_0 \hookrightarrow V$.
\item If in (1),   
$V_0$ separates the points of $\a(X)$ then $\a_0$ is an injection iff $\a$ is injective. So, if $X$ is compact then we get a faithful representation (in the sense of Definition \ref{d:repr}). This argument implies that if a compact {\it metrizable} $S$-system $X$ is strongly representable 
on $V$ then it is strongly representable on a {\it separable} Banach subspace $V_0$ of $V$.

\een    
\end{remark} 

Next we deal with the representability of families of real-valued functions on compact systems.
This topic is closely related to the ``smallness" of
the family $F$ in terms of its pointwise closure in the spirit of Theorem \ref{f:sub-fr}.

\begin{defin} \label{d:dualities}
Let $\K \subset \mathbf{Ban}$ be a subclass of Banach spaces.
\ben

\item
Let $X$ be
an $S$-space
and $(h,\a)$ a representation of $(S,X)$ on a Banach space $V$.
Let $F \subset C(X)$ be a bounded
$S$-invariant
family of continuous functions on $X$ and
$\nu: F \to V$ a
bounded mapping. We say that $(\nu,h,\a)$ is an $F$-\emph{representation} of the triple
$(F,S,X)$ if
$\nu$ is an $S$-mapping (i.e., $\nu(fs)=\nu(f)s$ for every $(f,s) \in F \times S$) and
$$
f(x)= \langle \nu(f), \a(x) \rangle
\ \ \ \forall \ f \in F, \ \ \forall \ x \in X.
$$
\begin{equation}  \label{diag1}
\xymatrix{ F \ar@<-2ex>[d]_{\nu} \times X
\ar@<2ex>[d]^{\a} \ar[r]  & \R \ar[d]^{id_{\R}} \\
V \times V^* \ar[r]  &  \R }
\end{equation}
\item
We say that a family $F \subset C(X)$ is
\emph{$\K$-representable}
if there exists a Banach space $V \in \K$
and a representation $(\nu,h,\a)$ of the triple $(F,S,X)$.
A function $f \in C(X)$ is said to be \emph{$\K$-representable} if
the orbit $fS$ is $\K$-representable.

Note that we do not assume in (1) or (2) that $\a$ is injective.
However, when the family $F$ separates points on $X$ it follows that the map $\alpha$ is necessarily an injection.
\item
In particular, we obtain the definitions of reflexively,  Asplund and Rosenthal representable families of functions on dynamical systems.
\een
\end{defin}

Clearly, every bounded
$S$-invariant family
$F \subset C(X)$ on an $S$-system $X$ is Banach representable via the canonical representation on $V=C(X)$.

\begin{remark} \label{Appr=Repr}
In some particular cases
$\K$-approximability and $\K$-representability are equivalent.
This happens for example in the following important cases:
\ben
\item $X$ is metrizable and $\K$ is closed under countable $l_2$-sums;
\item  $(S,X_f)$ is a cyclic system, $\K$ is closed under subspaces and
$f$ is $\K$-representable.
\item $(S,X_f)$ is a cyclic system and $\K$ is one of the following classes: reflexive, Asplund, Rosenthal.

The assertions
(1) and (2) are quite straightforward. In order to verify (3) (by reducing it to the case (2))
we note that every WAP (Asplund, tame) function $f \in C(X)$
is reflexive (Asplund, Rosenthal) representable.
See Theorems \ref{t:AspF} and \ref{t:tame-f} below.
\een
\end{remark}

The following result is very close to \cite[Theorem 12]{GM-AffComp}. It serves as a central
ingredient in the proof of Theorem \ref{t:tame=R-repr}.
We give a sketch of the proof just to illustrate the definitions.

\begin{thm} \label{t:AspDuality} \emph{(Small families of functions)}

Let $X$ be a compact $S$-space and $F \subset
C(X)$ a norm bounded $S$-invariant subset of $C(X)$.
\ben
\item
$(F,S,X)$ admits a Rosenthal representation iff $F$ is an eventually fragmented family iff $\cls_p(F) \subset \F(X)$.
\item
$(F,S,X)$ admits an Asplund representation iff $F$ is a fragmented family iff the envelope $\cls_p(F)$ of $F$ is a fragmented family.
\item
$(F,S,X)$ admits a reflexive representation iff
$\cls_p(F) \subset C(X)$ iff  $F$ has
DLP on $X$.
\een
\end{thm}
\begin{proof}
The ``if part" in the proof of (1) and (2) is based on a dynamical version of the well known DFJP construction, \cite{DFJP}.
See the proof of Theorem 12 in \cite[Sect. 8]{GM-AffComp}.
The ``only if part" is a direct consequence of the characterizations of
Asplund and Rosenthal spaces in terms of fragmented and eventually fragmented families,
Theorems \ref{f:Asp}.4 and \ref{f:RosFr}.4.

For the ``if part" in (3) we use \cite[Theorem 4.11]{Me-nz} which we reformulate here in
terms of Definition \ref{d:dualities}.
Let $w: F \times X \to \R$ be a separately continuous bounded map with compact spaces $F$ and $X$
such that $S$ acts on $X$ (from left) and on $F$ (from right) via separately continuous actions
such that $w(f,sx)=w(fs,x)$. Assume that $F$ (regarded as a (bounded) family of maps $X \to \R$)
separates the points of $X$. Then according to \cite[Theorem 4.11]{Me-nz} there exists a reflexive Banach space
$V$ and a faithful representation $(\nu,h,\a)$ of the triple $(F,S,X)$.

Another important part of the proof of (3) (which is close to Grothendieck's double limit theorem) is Theorem A.4 in \cite{BJM}. It asserts that for every compact space $X$ a bounded family $F \subset C(X)$
 has the Double Limit Property on $X$ iff
$\cls_p(F) \subset C(X)$.

For the ``only if" part in
(3) observe that if $V$ is a reflexive Banach space then every bounded subset $F$ of the dual $V^*$
has DLP on every bounded subset $X \subset V$
(Theorem \ref{t:ReflChar}).
\end{proof}

Note that, in Definition \ref{d:dualities}, when the family $F$ separates points on $X$ it follows that the map $\alpha$ is necessarily an injection. In view of this remark, Theorem \ref{t:AspDuality} implies Theorem \ref{t:tame=R-repr} and also
the essential part in
the following useful result 
which is interesting even for trivial actions.

\begin{thm} \label{t:WRN}
A compact $S$-system $X$ is RN (WRN, Eberlein) iff there exists
a bounded $S$-invariant $X$-separating family $F \subset C(X)$ which is  fragmented (resp.: eventually fragmented, DLP).
\end{thm}
\begin{proof}
The ``if part" follows from Theorem \ref{t:AspDuality}.

The ``only if part" follows by
considering the family $F$ of functions on $X \subset V^*$ induced by $B_V$ taking into account
the
characterization properties of reflexive, Asplund and Rosenthal Banach spaces. See
Theorems \ref{t:ReflChar}.3, \ref{f:Asp}.4 and \ref{f:RosFr}.4, respectively.
\end{proof}

\begin{thm} \label{l:functCases}
Let $X$ be a compact $S$-system and $f \in C(X)$.
\ben
\item $f \in \WAP(X)$ if and only if $fS$ has DLP on $X$.
\item $f \in \Asp(X)$ if and only if $fS$ is a fragmented family.
\item $f \in \Tame(X)$ if and only if $fS$ is eventually fragmented 
if and only if $fS$ is a tame family.
\een
\end{thm}
\begin{proof} (1) This was mentioned in Definition \ref{d:WAP}.

(2)
If $f \in \Asp(X)$ then it comes from a HNS factor $q: X \to Y$ and
$f' \circ q=f$ for some $f' \in C(Y)$.
Since $(S,Y)$ is HNS, the family of translations $\{\tilde{s}: Y \to Y\}$ is
fragmented. It follows that $f'S$ and hence also $fS$ are fragmented.

(3)
If $f \in \Tame(X)$ then it comes from a tame factor $q: X \to Y$ and $f' \circ q=f$ for some $f' \in C(Y)$.
Since $(S,Y)$ is tame, every $p \in E(S,Y)$ is a tame function. Therefore, $f'E(S,Y)=\cls_p (f'S) \subset \F(Y)$.
That is, $f'S$ is a Rosenthal family.
Then $f'S$ is eventually fragmented by Theorem \ref{f:sub-fr}. Now, Lemma \ref{r:fr1}.7 guarantees
that $fS$ is also eventually fragmented. 
Finally note that by Theorem \ref{f:sub-fr} the family $fS$ is eventually fragmented iff $fS$ is tame. 

The "only if" parts of (2) and (3) come from Theorems \ref{t:AspF} and \ref{t:tame-f}.
\end{proof}

\subsection{The purely topological case}

Note that the definitions and results of Section \ref{s:repr} (for instance, Theorems \ref{t:AspDuality} 
and \ref{t:WRN})
make sense in the purely topological setting, for trivial
$S=\{e\}$ actions,
yielding characterizations of ``small families" of  functions, and of RN, WRN and Eberlein compact spaces.

The ``only if" parts of these results, in the cases of Eberlein and RN compact spaces
(with trivial actions), are consequences of known characterizations of reflexive and Asplund Banach spaces.
The Eberlein case yields a well-known result: a compact space $X$
is Eberlein iff there exits a pointwise compact subset $Y \subset C(X)$ which separates the points of $X$.
The RN case is very close to results of Namioka
\cite{N} (up to some reformulations).
The case of WRN spaces seems to be new.

Recall that by a classical result of 
Benyamini-Rudin-Wage \cite{BRW},
(which answered a question posed by Lindenstrauss)
continuous surjective maps preserve the class of Eberlein compact spaces.
The same is true for uniformly Eberlein (that is, Hilbert representable) compacta.
Recently, Aviles and Koszmider \cite{AK} proved that this is not the case for the class RN
of Asplund representable compacta, answering a long standing open problem posed by
Namioka \cite{N}.
In view of these results
the following question seems to be interesting.

\begin{question}
Is the class of WRN compact spaces closed under continuous onto maps
(in the realm of Hausdorff spaces) ? 
\end{question}

We posed this question in early versions of this article. Very recently a counterexample was found by G. Martinez-Cervantes \cite{Mart-Cerv}. 

\begin{remark} \label{r: Todorc} \
\ben
\item
$\beta \N$ is an example of a compact space which is not WRN.
We thank Stevo Todor\u{c}evi\'{c} for communicating to us a beautiful
proof of this fact which is presented as an appendix to this work. See Theorem \ref{main} in Section \ref{app}.
\item
Below, in Corollary \ref{2arrows}, we show that the two arrows space is WRN. This space is not RN by a result of Namioka \cite[Example 5.9]{N}.
\een
So we have 
$$
RN \subsetneq WRN \subsetneq Comp
$$
\end{remark}


One may show that a compact $G$-space $K$ is WRN iff the Banach $G$-space $C(K)$ 
is Rosenthal generated. Meaning that there exists a Rosenthal $G$-space $V$ and a
linear dense (injective) continuous $G$-operator $V \to C(K)$. For compact spaces $K$ it is a WRN analog of Stegall's result for RN compacts and Asplund spaces (see for example \cite[Theorem 1.5.4]{Fa}).

\begin{thm} \label{t:WRNE(X)}  
	A metric dynamical $S$-system $X$ whose enveloping semigroup 
	$E(S,X)$ is a WRN compactum is tame.
\end{thm}  
\begin{proof}
	If $E(S,X)$ is a WRN compactum, then, by Todor\u{c}evi\'{c}'s Theorem \ref{main} $X$ can not contain a copy of $\beta \N$. Hence, by the dynamical version of the BFT Theorem \ref{D-BFT}, the system $(S,X)$ is tame.   
\end{proof}

\begin{question}
	Is the converse true ? I.e. is it true that the enveloping semigroup
	of a metric tame system is necessarily a WRN compactum ?
\end{question}


\sk

One can also derive the following corollary of Theorem \ref{t:AspDuality}.

\begin{cor} \label{c:DLPisFr}
Let $X$ be a compact space and $F \subset C(X)$ a bounded family. If
$F$ has DLP on $X$ then $F$ is a fragmented family.
\end{cor}
\begin{proof}
Theorem \ref{t:AspDuality}.3 guarantees that there exists a representation of $(F,\{e\},X)$ on a reflexive space $V$.
Since $V$ is Asplund we easily obtain by Theorem \ref{f:Asp}.4 that $F$ is fragmented. Alternatively, one can derive the latter statement from Lemma \ref{l:FrFa}.1.
\end{proof}

\begin{remark} \label{GenerRayn}
For trivial actions, Theorem \ref{t:AspDuality}.3 yields the following corollary.
Let $w: F \times X \to \R$ be a separately continuous bounded function with
compact spaces $F$ and $X$. Then there exists a reflexive Banach space $V$ and weakly continuous maps $\nu: F \to V, \ \a: X \to V^*$ such that
$\langle \nu(f),\a(x) \rangle =w(f,x)$.
This is a result of Raynaud \cite[Prop. 1.1]{Rayn} who
generalized an earlier result by Krivine and Maurey \cite[Theorem II.3]{KMa} which dealt with metrizable $X$ and $F$.
One can refine these results (even in the general action setting) as follows.
The fundamental DFJP-factorization construction from \cite{DFJP} has an ``isometric modification" \cite{LNO}.
Taking into account this modification
(which is
compatible with our $S$-space setting)
note that in Theorem \ref{t:AspDuality} we can prove more.
Namely,
if the given family $F \subset C(X)$ is \emph{bounded by the constant} $1$,
then we can assume that $\nu(F) \subset B$ and $\a(X) \subset B^*$.
Hence the following
sharper
(than the diagram \ref{diag1} in Definition \ref{d:dualities}) diagram commutes:
$$\xymatrix{ F \ar@<-2ex>[d]_{\nu} \times X
\ar@<2ex>[d]^{\a} \ar[r]  & [-1,1] \ar[d]^{id} \\
B \times B^* \ar[r]  &  [-1,1] }$$
For more details see \cite{GM-AffComp}.
\end{remark}

\section{WAP, HNS and tame systems}

\subsection{WAP systems and functions}

Besides the three equivalent conditions in
Definition \ref{d:WAP}.1 (for being a WAP function) we can now say
a bit more. In the proof below we twice use the following observation:
if a continuous function on a compact $S$-system $X$ comes from an $S$-factor $q: X \to Y$ with $\tilde{f} \in C(X)$, $f=\tilde{f} \circ q$, then
$fS$ has DLP on $X$ iff $\tilde{f}S$ has DLP on $Y$. This means, by Definition \ref{d:WAP}.1, that $f \in \WAP(X)$ iff
$\tilde{f} \in \WAP(Y)$.

\begin{lem} \label{t:WAP-f}
Let $X$ be a compact
$S$-space and $f \in C(X)$. The following conditions are equivalent:
\ben
\item $f \in \WAP(X)$.
\item $f$ is reflexively representable.
\item The cyclic $S$-space $X_f$ is reflexively representable (i.e., Eberlein).
\item $f$ comes from an Eberlein factor.
\item  $f$ comes from a WAP factor.
\een
\end{lem}
\begin{proof}

(1) $\Leftrightarrow$ (2):
Use Theorem \ref{t:AspDuality}.3.

The implications (2) $\Rightarrow$ (3) and (5) $\Rightarrow$ (1) follow from the above observation.
In the first case use also the facts that because $\tilde{f}S$ has DLP on
$X_f$ and $\tilde{f}S$ separates the points of $X_f$,
we can apply Theorem \ref{t:tame=R-repr}.

(3) $\Rightarrow$ (4):  Is trivial because $f$ comes from the factor $X \to X_f$.

(4) $\Rightarrow$ (5): Every Eberlein system is WAP by Theorem \ref{t:tame=R-repr}.3.
\end{proof}


\sk

\subsection{HNS systems and functions}

Recall the definition of HNS systems (see Definition \ref{d:HNS1} and Theorem \ref{t:HNSenv}.1). 

\begin{defin} \label{d:HNS} \cite{GM1,GM-AffComp}
A compact $S$-system $X$ is \emph{hereditarily non-sensitive} (HNS, in short)
if one of the following equivalent conditions are satisfied:
\ben
\item
For every closed nonempty 
(not necessarily $S$-invariant) 
subset $A \subset X$ and for every entourage $\eps$ from the unique compatible uniformity on $X$ there exists
an open subset $O$ of $X$ such that $A \cap O$ is nonempty and $s (A \cap O)$ is $\eps$-small for every $s \in S$.
\item
The family of translations $\widetilde{S}:=\{\tilde{s}: X \to X\}_{s \in S}$ is a fragmented family of maps.
 \item
 $E(S,X)$  is a fragmented family of maps from $X$ into itself.
 \een
\end{defin}

The equivalence of (1) and (2) is evident from the definitions.
Clearly, (3) implies (2) because $\widetilde{S} \subset E(S,X)$. As to the implication (2) $\Rightarrow$ (3),
observe that the pointwise closure of a fragmented family is again
a fragmented family, \cite[Lemma 2.8]{GM-rose}.

\begin{remark}
Note that if $S=G$ is a group then in Definition \ref{d:HNS}.1 one can
consider only {\it $G$-invariant} closed subsets $A$
(see the proof of \cite[Lemma 9.4]{GM1}).
\end{remark}

\begin{lem} \label{l:HNS-prop} {} \
\ben
\item
The class of HNS $S$-dynamical systems is closed under subsystems, products and factors.
\item
A compact dynamical $S$-system $X$ is HNS iff $\Asp(X)=C(X)$.
\een
\end{lem}
\begin{proof}
For group actions this was proved in
\cite[Sect. 2]{GM1}. The same method works for general semigroup actions.
\end{proof}

We collect here some characterizations of Asplund functions and HNS systems.
For continuous group actions they can be found in
\cite{GM1, GMU, GM-rose}.

\begin{thm} \label{t:AspF}
Let $X$ be a compact $S$-space and $f \in C(X)$.
The following conditions are equivalent:
\ben
\item $f \in \Asp(X)$. 
\item $fS$ is a fragmented family.
 \item $f$
  is Asplund representable.
 \item The cyclic $S$-system $X_f$ is RN.
  \item $f$ comes from an RN-factor.
\een
\end{thm}

\begin{proof}
(1) $\Rightarrow$ (2): This is Theorem \ref{l:functCases}.1.

(2) $\Leftrightarrow$ (3):
By Theorem \ref{t:AspDuality}.2.

(2) $\Rightarrow$ (4):
 Let $\tilde{f}$ be the function on the cyclic $S$-factor $X_f$ such that $f=\tilde{f} \circ \pi_f$
(see Remark \ref{r:cycl-comes}).
By Lemma \ref{r:fr1}.7, $fS$ is a fragmented family on $X$ iff $\tilde{f} S$ is a fragmented family on $X_f$.
In this case, since $\tilde{f} S$ separates the points of $X_f$, we can conclude by Theorem \ref{t:WRN} that $X_f$ is Asplund-representable.

(4) $\Rightarrow$ (5): This is trivial because $f$ comes from the factor
$X \to X_f$.

(5) $\Rightarrow$ (1):
An Asplund-representable (that is, RN) system is Asplund-approximable,
and therefore it is HNS by Theorem \ref{t:tame=R-repr}.2.
\end{proof}

\begin{thm} \label{t:HNS} \cite{GM1,GMU,GM-AffComp}
Let $X$ be a compact $S$-space.
The following conditions are equivalent:
\ben
\item
$(S,X)$ is \rm{HNS}.
\item
The family $\tilde{S} \subset X^X$ of translations
(equivalently, $E(S,X)$)
 is a fragmented family.
\item
For every countable subset $A \subset S$ the family $\tilde{A}$ is fragmented.
\item
$(S,X)$ has sufficiently many representations on Asplund Banach spaces. 
\item $\Asp(X)=C(X)$.

\sk

\noindent{
If $X$ is metrizable then each of the conditions above is equivalent also to any of the following conditions:}

\item  The enveloping semigroup $E(X)$ of $(S,X)$ is metrizable.

\item $(S,X)$ is RN (that is, $(S,X)$ admits a faithful representation on a (separable) Asplund space).

\item $(S,X)$ is HAE.

\item  The pseudometric
$$
d_{S}(x,y):=\sup\{d(sx,sy):  \ \ s \in S\}
$$
on $X$ is separable (where $d$ is a compatible metric on $X$).

 \item For every $f \in C(X)$ the pointwise closure ${\cls}_p(fS)$ is a metrizable subset of $\R^X$.
 \een
\end{thm}

\begin{proof}
(1) $\Leftrightarrow$ (2): Directly from Definition \ref{d:HNS}.

(2) $\Leftrightarrow$ (3): Directly from Theorem \ref{c:countDetermined2}.

(1) $\Leftrightarrow$ (4): Theorem \ref{t:tame=R-repr}.2.

(1) $\Leftrightarrow$ (5): Lemma \ref{l:HNS-prop}.2.

\sk
If $X$ is metrizable then

(1) $\Leftrightarrow$ (6): By \cite[Theorem 7]{GM-AffComp} (for group actions see \cite{GMU}).

(1) $\Leftrightarrow$ (7): By Theorem \ref{t:tame=R-repr}.

(2) $\Leftrightarrow$ (8): Lemma \ref{l:propFamil}.4.

(2) $\Leftrightarrow$ (9):
The family $\tilde{S}$ of the corresponding translation maps is a fragmented family.
This means that the natural map $X \to X^{S}$ is fragmented,
where $X^{S}$ carries the uniformity of $d$-uniform convergence.
It then follows, by Lemma \ref{r:fr1}.5,
that the image of $X$ into $X^{S}$ is separable.
This exactly means that $(X,d_{S})$ is separable.

(2) $\Leftrightarrow$ (10): Apply \cite[Lemma 4.4]{GM-rose}  to the family $fS$ for every $f \in C(X)$.
\end{proof}

\subsection{Tame systems and functions}

Recall that, by Definition \ref{d:tame}, a compact dynamical $S$-system $X$ is said to be tame
if every $p \in E(X)$ is a fragmented map.
For HNS systems, $E(X)$ is a fragmented family, hence, every HNS system is tame.

\begin{lem} \label{l:tame-prop} {} \
\ben
\item
The class of tame $S$-systems is closed under subsystems,
arbitrary products and factors.
\item
A compact $S$-dynamical system $X$ is tame iff $\Tame(X)=C(X)$.
\een
\end{lem}
\begin{proof}
For group actions this was proved in
\cite{GM-rose}. The same method works for general semigroup actions.
\end{proof}

In the following theorem we collect, for the reader's convenience, the various characterizations we have for tame functions.   
\begin{thm} \label{t:tame-f}
Let $X$ be a compact
$S$-space and $f \in C(X)$.
The following conditions are equivalent:
\ben
\item
$f \in \Tame(X)$.
\item
$fS$ is a tame family. 
\item
${\cls}_p(fS) \subset {\mathcal F}(X)$.  
\item $fS$ is an eventually fragmented family.
\item For every countable infinite subset $A \subset S$ there exists a countable infinite subset $A' \subset A$ such that
the corresponding pseudometric
$$
\rho_{f, A'}(x,y):=\sup\{|f(gx)-f(gy)|:  \ \ g \in A'\}
$$
on $X$ is separable.
\item $f$ is Rosenthal representable.
\item The cyclic $S$-space $X_f$ is Rosenthal representable (i.e., WRN).
\item $f$ comes from a WRN-factor.
\een
\end{thm}

\begin{proof}
(1) $\Rightarrow$ (4): Theorem \ref{l:functCases}.3.

The equivalence of
(2), (3) and (4) follows from Theorem \ref{f:sub-fr}.

(4) $\Leftrightarrow$ (5):
Use Theorem \ref{t:countDetermined}.

(4) $\Leftrightarrow$ (6):
By Theorem \ref{t:AspDuality}.1.

(4) $\Rightarrow$ (7):
 Let $\tilde{f}$ be the function on the cyclic $S$-factor $X_f$ such that $f=\tilde{f} \circ \pi_f$
(Remark \ref{r:cycl-comes}).
By Lemma \ref{r:fr1}.7,
$fS$ is an eventually fragmented family on $X$ iff $\tilde{f} S$
is an eventually fragmented family on $X_f$.
In this case, since $\tilde{f} S$ separates the points of $X_f$ we can conclude by Theorem \ref{t:WRN} that $X_f$ is 
Rosenthal-representable.

(7) $\Rightarrow$ (8): Is trivial because $f$ comes from the factor $X \to X_f$.

(8) $\Rightarrow$ (1):
 By Theorem \ref{t:tame=R-repr}.1 every WRN system is tame.
\end{proof}

\begin{thm} \label{t:tame} \cite{GM1,GMU,GM-AffComp}
Let $X$ be a compact dynamical $S$-system.
The following conditions are equivalent:
\ben
\item
$(S,X)$ is tame
(that is, every $p \in E(X)$ is a fragmented map).
\item
$(S,X)$ is 
eventually weakly fragmented (E-wFr in Definition \ref{d:E-HNS}).
\item
$(S,X)$ has sufficiently many representations on Rosenthal Banach spaces.
\item $\Tame(X)=C(X)$.

\sk

\noindent{If $X$ is metrizable then each of the conditions above is equivalent also to any of the following conditions:}

\item
Every $p \in E(X)$ is a Baire 1 map.

\item
$(S,X)$ is WRN (that is, $(S,X)$ admits a faithful representation on a (separable) Rosenthal space).

\item
$(S,X)$ is E-HAE.

\item
$(S,X)$ is 
eventually fragmented. 

\item
For every infinite subset $A \subset S$ there exists a (countable) infinite subset $A' \subset A$ such that
the corresponding pseudometric
$$
\rho_{A'}(x,y):=sup\{d(sx,sy):  \  s \in A'\}
$$
on $X$ is separable, where $d$ is a compatible metric on $X$.

\een
\end{thm}

\begin{proof}
(1) $\Leftrightarrow$ (3): Theorem \ref{t:tame=R-repr}.1.

(1) $\Leftrightarrow$ (4):
Lemma \ref{l:tame-prop}.2.

(1) $\Rightarrow$ (2):
We modify a proof from \cite[Prop. 4.1]{GM1} where we dealt with jointly continuous group actions.
By Definitions \ref{d:sens-fr-f}, \ref{d:E-HNS} and Lemma \ref{l:propFamil}.2,
it suffices to show that $(S_1,X)$ is E-wFr for every \emph{countable} subsemigroup $S_1 \subset S$.
Clearly,
$(S_1,X)$ remains tame.
Since $S_1$ is countable, metric factors separate points on $X$
(and therefore $(S_1,X)$ can be embedded in a product of metrizable $S_1$-systems).
Now we use \emph{cyclic compactifications} (Remark \ref{r:cycl-comes}). Observe that since $S_1$ is countable, every cyclic $S_1$-factor $X_f$ of $X$ is metrizable for every $f \in C(X)$.

It is easy to see that the class of E-wFr systems is
closed under products and passing to subsystems.
So, it suffices to show that every metrizable $S_1$-factor $M$ of $X$ is E-wFr.
By Lemma \ref{l:tame-prop}.1, $(S_1,M)$ is also tame.
Therefore $E(S_1,M)$ is a Rosenthal compactum. Hence, the topological space
$E(S_1,M)$ has the Fr\'{e}chet property, \cite{BFT}.
Let $j: S_1 \to E(S_1,M), s \mapsto \tilde{s}$ be the canonical Ellis compactification. Then, 
given an infinite subset  $j(L) \subset j(S_1)$,
there exists a countable subset
$\{t_n\}_{n \in \N}$ in $L$ such that $K:=\{j(t_n)\}_{n \in \N}$ is infinite and the sequence
$j(t_n)$ converges in
$E(S_1,M)$.
We apply Lemma \ref{l:FrFa}.3 to conclude
that $K$, as a family of maps from $M$ into itself, is
fragmented, so that $(S_1,M)$ is indeed E-wFr.

(2) $\Rightarrow$ (1):
Let $(S,X)$ be E-wFr. In order to show that $X$ is tame it suffices to prove,
by Lemma \ref{l:tame-prop}.2,
that $f \in \Tame(X)$ for every $f \in C(X)$.
By our assumption (2) $\tilde{S}$
is E-Fr. Since $f: (X,d) \to \R$ is uniformly continuous we obtain that $fS$ is E-wFr.
By Lemma \ref{l:propFamil}.3 we conclude that $fS$ is
eventually fragmented. Thus, $f$ is tame by Theorem \ref{t:tame-f}.

\sk
If $X$ is metrizable then

(1) $\Leftrightarrow$ (5): Lemma \ref{r:fr1}.2.

(1) $\Leftrightarrow$ (6): By Theorem \ref{t:tame=R-repr}.

(2) $\Leftrightarrow$ (7) $\Leftrightarrow$ (8): By Lemma
\ref{l:propFamil}.5.

(8) $\Rightarrow$ (9):
For every countable infinite subset $A \subset S$ there exists a countable infinite subset $A' \subset A$ such that
$A'$ is a fragmented family. This means, by Lemma \ref{r:fr1}.6, that the induced map $X \to X^{A'}$ is fragmented,
where $X^{A'}$ carries the uniformity of uniform convergence. Since $X$ is second countable,
Lemma \ref{r:fr1}.5 implies
that the image of $X$ into $X^{A'}$ is separable. This exactly means that $(X,\rho_{A'})$ is separable.

(9) $\Rightarrow$ (4): Let $f \in C(X)$. Since $f: (X,d) \to \R$ is uniformly continuous
it is easy to see
that the map
$1_X: (X,\rho_{A'}) \to (X,\rho_{f, A'})$ is uniformly continuous.
This implies that $(X,\rho_{f, A'})$ is also separable.
By Theorem \ref{t:tame-f} we conclude that $f \in \Tame(X)$.
\end{proof}

Since every E-Fr is E-wFr, the implication (8) $\Rightarrow$ (1) holds for every (not necessarily metric) compact system.

\section{A characterization of tame symbolic systems}

\subsection{Symbolic systems}

The classical \emph{Bernoulli shift system} is defined as the cascade $(\Z,\Omega)$, where $\Omega:=\{0,1\}^{\Z}$.
We have the natural
$\Z$-action on 
the compact metric space  
$\Omega$ induced by the left $T$-shift:
$$
\Z \times \Omega \to \Omega, \ \ \ T^m (\omega_i)_{i \in \Z}=(\omega_{i+m})_{i \in \Z} \ \ \ \forall (\omega_i)_{i \in \Z} \in \Omega, \ \ \forall m \in \Z.
$$

More generally, for a
discrete monoid $S$ and a finite alphabet $A:=\{0,1, \dots,n\}$
the compact space $\Omega:=A^S$ is an $S$-space under the action
$$ S \times \Omega \to \Omega, \ \ (s \omega)(t)=\omega (ts),\ \omega \in A^S, \ \ s,t \in S.$$
A closed $S$-invariant subset $X \subset A^S$ defines a
subsystem $(S,X)$. Such systems are called {\em subshifts\/} or
{\em symbolic dynamical systems\/}.
For a nonempty $L \subseteq S$ define the natural projection $$\pi_L: A^S \to A^{L}.$$
The compact zero-dimensional space $A^S$ is metrizable iff $S$ is countable
(and, in this case, $A^S$ is homeomorphic to the Cantor set).

It is easy to see
that the full shift system $\Omega=A^S$ (hence also every subshift) is  \emph{uniformly expansive}. This means that
there exists an entourage $\eps_0 \in \mu$ in the natural uniform structure of $A^S$ such that
for every distinct $\omega_1 \neq \omega_2$ in $\Omega$ one can find
$s \in S$ with
$(s\omega_1, s\omega_2) \notin \eps_0$. Indeed, take
$$\eps_0:=\{(u,v) \in \Omega \times \Omega: \  u(e)=v(e)\},$$
where $e$, as usual, is the neutral element of $S$.

\begin{lem} \label{subshiftsAREcyclic}
Every symbolic dynamical $S$-system $X \subset \Omega=A^S$ is cyclic (Definition \ref{d:cyclic}).
\end{lem}
\begin{proof} It suffices to find $f \in C(X)$ such that the orbit $fS$ separates the points of $X$ since then, by the Stone-Weierstrass theorem, $(S,X)$ is isomorphic to its cyclic $S$-factor $(S,X_f)$.
The family
$$\{\pi_{s}: X \to A=\{0,1, \dots,n\} \subset \R\}_{s \in S}$$
of basic projections clearly separates points on $X$ and we let
$f:=\pi_e: X \to \R$.
Now observe that $fS=\{\pi_{s}\}_{s \in S}$.
\end{proof}

\begin{prop} \label{p:scattered} \cite[Prop. 7.15]{Me-nz}
Every scattered compact jointly continuous $S$-space $X$ is RN.
\end{prop}
\begin{proof}
A compactum $X$ is scattered iff $C(X)$ is Asplund, \cite{NP}. Now use the canonical
$S$-representation  $S \to \Theta(V)_s,  \a: X \hookrightarrow B^*$  of $(S,X)$ on the Asplund space $V := C(X)$.
\end{proof}

The following result recovers and extends
\cite[Sect. 10]{GM1} and \cite[Sect. 7]{Me-nz}.

\begin{thm} \label{f}
For a discrete monoid $S$ and a finite alphabet $A$ let $X \subset A^S$ be a subshift.
The following conditions are equivalent:
\begin{enumerate}
\item
$(S,X)$ is Asplund representable (that is, RN).
\item $(S,X)$ is HNS.
\item $X$ is scattered.
\sk
If, in addition, $X$ is metrizable (e.g., if $S$ is countable) then
 each of the conditions above is equivalent also to:

\item $X$ is countable.
\end{enumerate}
\end{thm}
\begin{proof}
(1) $\Rightarrow$ (2):  Follows directly from Theorem \ref{t:tame=R-repr}.2.

 (2) $\Rightarrow$ (3):
Let $\mu$ be the natural uniformity on $X$ and $\mu_S$ the (finer) uniformity of uniform convergence on $X \subset X^S$
(we can treat $X$ as a subset of $X^S$ under the assignment $x \mapsto \hat{x}$, where $\hat{x}(s)=sx$).
If $X$ is HNS then the family $\tilde{S}$ is fragmented. This means that $X$ is $\mu_S$-fragmented.
As we already mentioned, every subshift $X$ is uniformly $S$-expansive.
 Therefore, $\mu_S$ coincides with the discrete uniformity $\mu_{\Delta}$ on $X$ (the largest possible uniformity on the \emph{set} $X$).
 Hence, $X$ is also $\mu_{\Delta}$-fragmented.
This means, by Lemma \ref{r:fr1}.4, that $X$ is a scattered compactum.

 (3) $\Rightarrow$ (1):
 Use Proposition \ref{p:scattered}.

\sk
If $X$ is metrizable then

 (4) $\Leftrightarrow$ (3): A scattered compactum is metrizable iff it is countable.
\end{proof}

Every zero-dimensional compact $\Z$-system $X$
can be embedded into a
product $\prod X_f$ of (cyclic) subshifts
$X_f$ (where, one may consider only continuous functions $f: X \to \{0,1\}$) of the Bernoulli system $\{0,1\}^{\Z}$.

\sk

For more information about countable
(that is, HNS)
subshifts see e.g.  \cite{Sh} and \cite{C-W}.
In \cite{AG-16} the authors show that the class of WAP $\Z$-subshifts is very rich and
provide a powerful method for constructing such systems with various properties.
However, no concrete combinatorial characterisation of this class is known to us.

\begin{problem}
Find a nice characterization for WAP (necessarily, countable) $\Z$-subshifts.
\end{problem}


Next we consider tame subshifts.

\begin{thm} \label{subshifts}
Let $X$ be a subshift of $\Omega=A^S$.
The following conditions are equivalent:

\ben
\item
$(S,X)$ is a tame system.
\item
For every infinite subset $L \subseteq S$ there exists an infinite subset
$K \subseteq L$
and a countable subset $Y \subseteq X$ such that
$$
 \pi_{K}(X)=\pi_{K}(Y).
$$
That is,
$$
\forall x=(x_s)_{s \in S} \in X,  \    \exists y=(y_s)_{s \in S} \in Y \quad {\text{with}}\quad
 x_{k}=y_{k} \ \ \forall k \in K.
$$
\item
For every infinite subset $L \subseteq S$ there exists an infinite subset $K \subseteq L$
such that
$\pi_{K}(X)$ is a countable subset of $A^K$.
\item
$(S,X)$ is Rosenthal representable (that is, WRN).
\een
\end{thm}

\begin{proof}
(1) $\Leftrightarrow$ (2): As in the proof of Lemma \ref{subshiftsAREcyclic}
define $f:=\pi_e \in C(X)$. Then $X$ is isomorphic to the cyclic $S$-space $X_f$.
By Theorem \ref{t:tame}, $(S,X)$ is a tame system iff $C(X)=\Tame(X)$. By Lemma \ref{subshiftsAREcyclic}, $C(X)=\A_f$,
so we have only to show that $f \in \Tame(X)$.

By Theorem \ref{t:tame-f}, $f:=\pi_e: X \to \R$ is a tame function iff
for every infinite subset
$L \subset S$ there exists a countable infinite subset 
$K \subset L$ such that the corresponding pseudometric
$$
\rho_{f, K}(x,y):=
\sup_{k \in K} \{| (\pi_e)(kx) - (\pi_e)(ky)|\}
= \sup_{k \in K}\{|x_{k} - y_{k}|\}
$$
on $X$ is separable.
The latter assertion means that there exists a countable subset $Y$ which is
$\rho_{f, K}$-dense in $X$.
Thus for every $x \in X$ there is a point $y \in Y$ with $\rho_{f, K}(x,y) < 1/2$.
As the values of the function $f=\pi_{0}$ are in the set $A$, we
conclude that  $\pi_{K}(x)=\pi_{K}(y)$, whence
$$
\pi_{K}(X)=\pi_{K}(Y).
$$

The equivalence of (2) and (3) is obvious.

(1) $\Rightarrow$ (4): $(S,X)$ is Rosenthal-approximable (Theorem \ref{t:tame=R-repr}.1).
On the other hand, $(S,X)$ is cyclic (Lemma \ref{subshiftsAREcyclic}).
By Theorem \ref{t:tame-f}.7 we can conclude that $(S,X)$ is WRN.

(4) $\Rightarrow$ (1): Follows directly by Theorem \ref{t:tame=R-repr}.1.
\end{proof}

\begin{remark}
From Theorem \ref{subshifts}  we can deduce the following peculiar fact.
If $X$ is a tame subshift of $\Omega=\{0,1\}^{\Z}$ and $L \subset \Z$ an infinite set,
then there exist an infinite subset $K \subset L$,  $k \ge 1$, and
$a \in \{0,1\}^{2k+1}$ such that
$ X \cap [a] \not=\emptyset$ and
$\forall x, x' \in X \cap [a]$ we have $x|_K = x'|_K$. Here
$[a] =\{z \in \{0,1\}^\Z : z(j) = a(j),\ \forall |j| \le k\}$.
In fact,
since $\pi_K(X)$ is a countable closed
set it contains an isolated point, say $w$, and then the open set
$\pi_K^{-1}(w)$ contains a subset $[a] \cap X$ as required.
\end{remark}

\subsection{Tame and HNS subsets of $\Z$}

We say that a subset $D \subset \Z$ is \emph{tame} if
the characteristic function $\chi_D: \Z \to \R$ is a tame function on the group $\Z$.
That is, when this function \emph{comes}
from a pointed compact tame
$\Z$-system $(X,x_0)$.
Analogously, we say that $D$ is {\em HNS} (or \emph{Asplund}),  \emph{WAP}, or
\emph{Hilbert} if $\chi_D: \Z \to \R$ is an Asplund, WAP or Hilbert function on $\Z$, respectively.
By basic properties of the {\em cyclic system} $X_D : =\cls \{\chi_D \circ T^n : n \in \Z\} \subset \{0,1\}^\Z$
(see Remark \ref{r:cycl-comes}), 
the subset $D \subset \Z$ is tame (Asplund, WAP)
iff the associated subshift $X_D$ is tame (Asplund, WAP).

Surprisingly it is
not known whether
$X_f: =\cls \{f \circ T^n : n \in \Z\} \subset \R^\Z$
is a Hilbert system when $f: \Z \to \R$ is a Hilbert function (see \cite{GW-Hi}).
The following closely related question from \cite{Me-opit} is also open:
Is it true that Hilbert representable compact metric $\Z$-spaces are closed under factors?

\begin{remark} \label{r:Rup}
The definition of WAP sets was introduced by Ruppert \cite{Rup-WAPsets}.
He has the following characterisation (\cite[Theorem 4]{Rup-WAPsets}):
\sk

$D \subset \Z$ is a WAP subset if and only if every infinite subset $B \subset \Z$ contains a finite
subset $F \subset B$ such that the set
$$
\cap_{b \in F} (b+D) \setminus \cap_{b \in B \setminus F} (b+D)
$$
is finite.
See also \cite{Gl-tf}.
\end{remark}

\begin{thm} \label{tame subsets in Z}
	Let $D$ be a subset of $\Z$. 
The following conditions are equivalent:
\ben
\item
$D$ is a tame subset (i.e., the associated subshift $X_D\subset \{0,1\}^{\Z}$ is tame).
\item
For every infinite subset $L \subseteq \Z$ there exists an infinite subset
$K \subseteq L$
and a countable subset $Y \subseteq \beta \Z$ such that
for every $x \in \beta \Z$ there exists $y \in Y$ such that
$$
n+D \in x \Longleftrightarrow n+D \in y \ \ \ \forall n \in K
$$
(treating $x$ and $y$ as ultrafilters on the set $\Z$).

\een
\end{thm}
\begin{proof}
By the universality of the greatest ambit $(\Z, \beta \Z)$ it suffices to check when
the function
$$
f=\chi_{\overline{D}}: \beta \Z \to \{0,1\}, \ f(x) =1 \Leftrightarrow x \in \overline{D},
$$
the natural extension function of $\chi_{D}: \Z \to  \{0,1\}$,
is tame (in the usual sense, as a function on the compact cascade $\beta \Z$),
where we denote by $\overline{D}$ the closure of $D$ in $\beta \Z$ (a clopen subset).
Applying Theorem \ref{t:tame-f} to $f$ we see that the following condition is both necessary and sufficient:
For every infinite subset $L \subseteq \Z$ there exists an infinite subset
$K \subseteq L$ and a countable subset $Y \subseteq \beta \Z$ which is dense in the
pseudometric space
$(\beta \Z, \rho_{f, K})$.
Now saying that $Y$ is dense is the same as the requiring that $Y$ be
$\eps$-dense for every $0< \eps < 1$.
However, as $f$ has values in $\{0,1\}$ and $0 <\eps <1$ we conclude that for every $x \in \beta\Z$ there is $y \in Y$ with
$$
x \in n+\overline{D}  \Longleftrightarrow y \in n+\overline{D}  \ \ \ \forall n \in K,
$$
and the latter is equivalent to
$$
n+D \in x \Longleftrightarrow n+D \in y \ \ \ \forall n \in K.
$$
\end{proof}

\begin{thm} \label{Asplund subsets in Z}
		Let $D$ be a subset of $\Z$. 
The following conditions are equivalent:
\ben
\item
$D$ is an Asplund subset (i.e., the associated subshift $X_D\subset \{0,1\}^{\Z}$ is Asplund).
\item
There exists a \emph{countable} subset $Y \subseteq \beta \Z$ such that
for every $x \in \beta \Z$ there exists $y \in Y$ such that
$$
n+D \in x \Longleftrightarrow n+D \in y \ \ \ \forall n \in \Z.
$$

\een
\end{thm}
\begin{proof} Similar to Theorem \ref{tame subsets in Z} using Theorem \ref{t:HNS}.
\end{proof}

\begin{ex} \label{ex:AspnotWAP}
$\N$ is an Asplund subset of $\Z$ which is not a WAP subset.
In fact, let $X_{\N}$ be the corresponding subshift. Clearly $X_{\N}$ is homeomorphic
to the two-point compactification of $\Z$, with $\{\bf{0}\}$ and $\{\bf{1}\}$ as minimal subsets.
Since a
transitive WAP
system admits a unique minimal set, we conclude that $X_{\N}$ is not WAP
(see e.g. \cite{Gl-03}).
On the other hand, since $X_{\N}$ is countable we can apply Theorem \ref{f} to show that it is HNS.
Alternatively, using Theorem \ref{Asplund subsets in Z}, we can take $Y$ to be $\Z \cup \{p,q\}$,
where we choose $p$ and $q$ to be any two non-principal ultrafilters such that $p$ contains $\N$
and $q$ contains $-\N$.
\end{ex}

\section{Entropy and null systems} \label{s:entropy}

We begin by recalling the basic definitions
of topological (sequence) entropy.
Let $(X,T)$ be a cascade,
i.e., a $\Z$-dynamical system, and $A=\{a_0<a_1<\ldots\}$ a
sequence of integers.
Given an open cover $\Ucal$ define
$$
h^A_{top}(T,\Ucal)=\limsup_{n\to\infty} \frac{1}{n}N(\bigvee
_{i=0}^{n-1}T^{-a_i}(\Ucal))
$$
The {\it topological entropy along the sequence $A$} is then defined by
$$
h^A_{top}(T)= \sup
\{h^A_{top}(T,\Ucal) :  \Ucal\ \text{an open cover of $X$}\}.
$$
When the phase space $X$ is zero-dimensional, one can replace
open covers by clopen partitions.
We recall that a dynamical system $(T,X)$ is called {\em null}
if $h^A_{top}(T) =0$ for every infinite $A \subset \Z$.
Finally when $Y \subset \{0,1\}^{\Z}$, and $A \subset \Z$
is a given subset of $\Z$, we say that $Y$ {\em is free on} $A$
or that $A$ {\em is an interpolation set for} $Y$,
if $\{y|_A : y \in Y\} = \{0,1\}^A$.

By theorems of Kerr and Li \cite{KL05}, \cite{KL}
every null $\Z$-system is tame, and every tame system has zero topological entropy.
From results of Glasner-Weiss \cite{GW} (for (1)) and Kerr-Li \cite{KL}
(for (2) and (3)), the following results can be easily deduced.
(See Propositions 3.9.2, 6.4.2 and 5.4.2 of \cite{KL} for the positive topological entropy,
the untame, and the nonnull claims, respectively.)

\begin{thm} \label{t:3}  \
\begin{enumerate}
\item
A subshift $X \subset \{0,1\}^{\Z}$ has positive topological entropy iff
there is a subset $A \subset \Z$ of positive density such that
$X$ is free on $A$.
\item
A subshift $X \subset \{0,1\}^{\Z}$ is not tame iff
there is an infinite subset $A \subset \Z$ such that
$X$ is free on $A$.
\item
A subshift $X \subset \{0,1\}^{\Z}$ is not null iff for every
$n \in \N$ there is a finite subset $A_n \subset \Z$ with $|A_n| \ge n$
such that $X$ is free on $A_n$.
\end{enumerate}
\end{thm}
\begin{proof}
We consider the second claim; the other claims are similar.

Certainly if there is an infinite $A \subset \Z$ on which $X$ is free then
$X$ is not tame (e.g. use Theorem \ref{subshifts}).
Conversely, if $X$ is not tame then, by Propositions 6.4.2 of \cite{KL},
there exists a non diagonal IT pair $(x,y)$.
As $x$ and $y$ are distinct there is
an $n$ with, say, $x(n) =0, y(n)=1$. Since $T^n(x,y)$ is also an IT pair
we can assume that $n=0$. Thus $x \in U_0$ and $y \in U_1$, where these are
the cylinder sets $U_i = \{z \in X : z(0) = i\}, i=0,1$.
Now by the definition of an IT pair there is an infinite set $A \subset \Z$
such that the pair $(U_0,U_1)$ has $A$ as an independence set. This
is exactly the claim that $X$ is free on $A$.
\end{proof}


The following theorem was proved (independently) by
Huang \cite{H}, Kerr and Li \cite{KL}, and Glasner \cite{Gl-str}.
See also Remark \ref{r:11} below.
\begin{thm} \label{almost-auto}  \emph{(A structure theorem for minimal tame dynamical systems)}
Let $(G,X)$ be a tame minimal metrizable dynamical system with $G$ an
 abelian group.
Then:
\begin{enumerate}
\item
$(G,X)$ is an almost one to one extension $\pi : X \to Y$ of a minimal equicontinuous system $(G,Y)$.
\item
$(G,X)$ is uniquely ergodic and the factor map $\pi$ is, measure
theoretically, an isomorphism of the corresponding measure preserving
system on $X$ with the Haar measure on the equicontinuous factor $Y$.
\end{enumerate}
\end{thm}

\begin{exs} \label{IP} \
\begin{enumerate}
\item
According to Theorem \ref{almost-auto} the Morse minimal system,
which is uniquely ergodic and has zero entropy,
is nevertheless
 not tame as it fails to be an almost 1-1 extension of its adding machine factor.
We can therefore deduce that, a fortiori, it is not null.
\item
Let $L = IP\{10^t\}_{t=1}^\infty \subset \N$ be the IP-sequence generated by
the powers of ten, i.e.
$$
L =\{10^{a_1} + 10^{a_2} + \cdots + 10^{a_k} :
1 \le a_1 < a_2 < \cdots < a_k\}.
$$
Let $f= 1_L$ and let $X= \OC_T(f) \subset \{0,1\}^\Z$, where $T$ is
the shift on $\Om = \{0,1\}^\Z$.
The subshift $(T,X)$ is not tame. In fact it can be shown that
$L$ is an interpolation set for $X$.
\item
  
Take $u_n$ to be the concatenation of the words $a_{n,i} 0^n$,
where $a_{n,i},\ i=1,2,3, \dots,2^n$ runs over $\{0,1\}^n$. Let $v_n = 0^{|u_n|}$,
$w_n = u_nv_n$ and $w_\infty$ the infinite concatenation
$\{0,1\}^\N \ni w_\infty = w_1w_2w_3\cdots$. Finally define $w \in \{0,1\}^\Z$
by $w(n) = 0$ for $n \le 0$ and $w(n) = w_\infty(n)$.
Then $X=\OC_T(w) \subset \{0,1\}^\Z$ is a countable subshift, hence HNS
and a fortiori tame, but for an appropriately chosen sequence the sequence entropy of $X$ is $\log 2$.
Hence, $X$ is not null.
Another example of a countable nonnull subshift can be found in \cite[Example 5.12]{H}.
\item
In \cite[Section 11]{KL} Kerr and Li construct a Toepliz ( =  a minimal almost one-to-one extension
of an adding machine) which is tame but not null.
\item
In \cite[Theorem 13.9]{GY} the authors show that for interval maps
being tame is the same as being null.
\end{enumerate}
\end{exs}

\begin{remark}
Let $\sigma: [0,1] \to [0,1]$ be a continuous self-map on the closed interval.
In an unpublished paper \cite{Mich+} the authors show that
the enveloping semigroup $E(X)$ of the cascade ($\N \cup \{0\}$-system) $X=[0,1]$ is either metrizable or
it contains a topological copy of $\beta \N$.
The metrizable enveloping semigroup case occurs exactly when the system is HNS.
This was proved in \cite{GMU} for group actions but it remains true for semigroup actions,
\cite{GM-AffComp}.
The other case occurs iff $\sigma$ is Li-Yorke chaotic.
Combining this result with Example
\ref{IP}.4
one gets:  HNS\ =\ null\ = \ tame, for any cascade
$([0,1],\sigma)$.
\end{remark}


Combining Theorem \ref{subshifts}(2) and Theorem \ref{t:3}(2) we obtain the following surprizing dichotomy:

\begin{thm}\label{dich}
	For a subshift $X \subset A^Z$ we have the following dichotomy:
	\begin{enumerate}
		\item
		Either there exists an infinite subset $L \subset \Z$ such that $X$ is free on $L$,
		or 
		\item
		for every infinite subset $L \subseteq S$ there exists an infinite subset $K \subseteq L$
		such that
		$\pi_{K}(X)$ is a countable subset of $A^K$.
	\end{enumerate}
\end{thm}

\section{Some examples of tame functions and systems}
\label{s:tame-type}

The class of tame dynamical systems is quite large and contains the class of HNS (hence also of WAP) systems.
Also, as was mentioned above, every null $\Z$-system is tame.

\begin{ex} \label{e:tameNOThns} \
\ben
\item
In his paper \cite{Ellis93} Ellis, following Furstenberg's
classical work \cite{Furst-63-Poisson}, investigates the projective
action of $GL(n,\R)$
on the projective space $\mathbb{P}^{n-1}$. It follows from his
results that the corresponding enveloping semigroup is
not first countable. However,
in a later work \cite{Ak-98}, Akin studies the action of $G=GL(n,\R)$
on the sphere $\mathbb{S}^{n-1}$ and shows that here the enveloping
semigroup is first countable (but not metrizable).
It follows that the dynamical systems
$D_1=(G, \mathbb{P}^{n-1})$ and $D_2=(G, \mathbb{S}^{n-1})$
are tame but not HNS. Note that $E(D_1)$ is Fr\'echet,
being a continuous image of a first countable compact space, namely $E(D_2)$.

\item (Huang \cite{H})
An almost 1-1 extension $\pi: X \to Y$ of
a minimal equicontinuous metric $\Z$-system
$Y$ with $X \setminus X_0$ countable,
where $X_0=\{x \in X : |\pi^{-1}\pi(x)|=1\}$, is tame.

\item (See \cite{GM1}) Consider an irrational rotation $(\T,R_\alpha)$. Choose
$x_0 \in \T$ and split each point of the orbit $x_n=x_0 + n \alpha$
into two points  $x_n^{\pm}$.
This procedure
results is a {\em Sturmian}
(symbolic)
dynamical system $(X,\s)$ which is a minimal almost
1-1 extension of $(\T,R_\alpha)$.
Then $E(X,\s) \setminus \{\s^n\}_{n \in \Z}$ is homeomorphic to the two arrows space, a basic example of a non-metrizable Rosenthal
compactum. It follows that $E(X,\s)$ is also a Rosenthal compactum. Hence, $(X,\s)$ is tame but not HNS (by Theorems \ref{D-BFT} and \ref{t:HNSenv}.2).
\item
Let $P_0$ be the set
$[0, c)$ and $P_1$ the set $[c, 1)$; let $z$ be a point in $[0, 1)$ (identified with $\T$) via the rotation $R_{\a}$ we get the binary bisequence $u_n$,
$n \in \Z$
defined by $u_n = 0$ when $R_{\a}^n(z)  \in P_0, u_n = 1$ otherwise.
These are called 
\emph{Sturmian like codings}.
With $c = 1 - \al$ we retrieve the previous example.
For example, when $\a:=\frac{\sqrt{5}-1}{2}$ and $c = 1 -\a$ the corresponding sequence,
computed at $z =0$, is called the \emph{Fibonacci bisequence}.
\een
\end{ex}

Motivated by the Example \ref{e:tameNOThns}.4
we next present a new class of generalized Sturmian
systems.

\begin{ex}[A class of generalized Sturmian systems]\label{e:spheres}
Let $\a =(\al_1, \dots, \al_d)$ be a vector in $\R^d,\ d \ge 2$
with $1,\al_1, \dots, \al_d$ independent over $\Q$.
Consider the minimal equicontinuous dynamical system $(R_\al,Y)$,
where $Y = \T^d = \R^d / \Z^d$ (the $d$-torus) and $R_\al y = y +\al$.
Let $D$ be a small closed $d$-dimensional ball in $\T^d$ and let $C =
\partial D$ be its boundary, a  $d-1$-sphere.
Fix $y_0 \in {\rm{int}} D$ and let
$X = X(D,y_0)$ be the symbolic
system generated by the function
$$
x_0 \in \{0,1\}^\Z \ {\text{defined by}}\  x_0(n) = \chi_D(R_\al^ny_0),
\qquad
X = \overline{\mathcal{O}_{\sig} x_0} \subset  \{0,1\}^\Z,
$$
where $\sig$ denotes the shift transformation.
This is a well known construction and it is not hard to check that
the system $(\sig, X)$
is minimal and admits $(R_\al,Y)$ as an almost 1-1 factor:
$$
\pi : (\sig, X) \to (R_\al, Y).
$$
\end{ex}

\begin{thm}\label{sphere}
There exists a ball $D \subset \T^d$ as above such that the corresponding symbolic dynamical system $(\sig, X)$ is tame.
\end{thm}

\begin{proof}
{\bf 1.}\
First we show that a sphere $C \subset [0,1)^d \cong \T^d$ can be chosen
so that for every $y \in \T^d$  the set
$(y + \{n \al : n \in Z\}) \cap C$ is finite.
We thank Benjamin Weiss for providing the following proof of this fact.
\begin{enumerate}
\item
For the case $d=2$ the argument is easy.
If $A$ is any countable subset of the square
$[0,1) \times [0,1)$ there are only a countable
number of circles that contain three points of $A$. These circles
have some countable collection of radii. Take any circle with a radius
which is different from all of them and no translate of it will contain
more than two points from the set $A$. Taking $A = \{n\al : n \in \Z\}$
we obtain the required circle.
\item
We next consider the case $d =3$, which easily generalizes to the general case $d \ge 3$.
What we have to show is that there can not be infinitely many points in
$$
A = \{(n\al_1 - [n\al_1],\al_2 - [n\al_2],\al_3 - [n\al_3]): n \in \Z\}
$$
that lie on a plane.
For if that is the case, we consider all $4$-tuples of elements from the set
$A$ that do not lie on a plane to get a countable set of radii
for the spheres that they determine. Then taking a sphere with radius different from that collection we obtain our required sphere.
In fact, if a sphere contains infinitely many points of
$A$ and no $4$-tuple from $A$ determines it then they all lie on a single plane.

So suppose that there are infinitely many points in $A$ whose inner
product with a vector $v =(z,x,y)$  is always equal to $1$.
This means that there are infinitely many
equations of the form:
\begin{equation}\tag{*}
 z\al_1 + x\al_2 + y\al_3 =
 1/n + z[n\al_1]/n + x[n\al_2]/n + y[n\al_3]/n.
\end{equation}
Subtract two such equations with the second using $m$ much bigger than
$n$ so that the coefficient of $y$ cannot vanish.
We can express $y = rz + sx + t$ with $r, s$  and $t$ rational.
This means that we can replace (*) by
\begin{equation}\tag{**}
z \al_1 + x\al_2 + y\al_3 =
 1/n + t[n\al_3]/n + z([n\al_1]/n + r[n\al_3]/n)  +
 x([n\al_2]/n +s[n\al_3]/n).
\end{equation}
Now $r, s$ and $t$ have some fixed denominators and (having infinitely
many choices) we can take another equation like (**) where $n$
(and the corresponding $r,s, t$) is replaced by some much bigger $k$,
then subtract again to obtain an equation of the form $x = pz + q$ with $p$ and $q$ rational.
Finally one more step will show that $z$ itself is rational.
However, in view of (*), this contradicts the independence of
$1, \al_1, \al_2, \al_3$ over $\Q$ and our proof is complete.
\end{enumerate}

{\bf 2.}\
Next we show that for $C$ as above
\begin{quote}
for every converging sequence $n_i\al$, say $n_i\al \to \beta \in \T^d
\cong E(R_\al,\T^d)$, there exists a subsequence $\{n_{i_j}\}$ such
that for every $y \in \T^d$, $y + n_{i_j}\al$ is either
eventually in the interior of $D$ or eventually in its exterior.
\end{quote}
Clearly we only need to consider points $y \in C - \beta$.
Renaming we can now assume that $n_i \al \to 0$ and that $y \in C$.
Passing to a subsequence if necessary we can further assume that
the sequence of unit vectors $\frac{n_i\al}{\|n_i \al\|}$ converges,
\begin{equation*}
\frac{n_i\al}{\|n_i \al\|} \to v_0 \in \mathbb{S}^{d-1}.
\end{equation*}
In order to simplify the notation we now assume that $C$ is centered
at the origin.
For every point $y \in C$ where $\langle y, v_0 \rangle \ne 0$
we have that $y + n_i\al$ is either eventually in the interior of $D$
or eventually in its exterior. On the other hand, for the points $y \in C$
with $\langle y, v_0 \rangle = 0$ this is not necessarily the case.
In order to deal with these points we need a more detailed information
on the convergence of $n_{i}\al$ to $\beta$. At this stage we consider
the sequence of orthogonal projections of the vectors $n_{i}\al$ onto
the subspace $V_1 = \{u \in \R^d : \langle u , v_0 \rangle =0\}$, say
$u_i= \proj_{v_0}(n_i \al) \to u = \proj_{v_0}(\beta)$.
If it happens that eventually $u_i =0$ this means that all
but a finite number of the $n_i \al$'s are on the line defined by $v_0$
and our required property is certainly satisfied
\footnote{Actually this possibility can not occur,
as is shown in the first step of the proof.}.
Otherwise
we choose a subsequence (again using the same index) so that
\begin{equation*}
\frac{u_i}{\|u_i\|} \to v_1 \in \mathbb{S}^{d-2}.
\end{equation*}
Again (as we will soon explain) it is not hard to see that for points
$y \in C \cap V_1$ with
$\langle y, v_1 \rangle \ne 0$ we have that $y + n_{i}\al$ is either eventually in the interior of $D$ or eventually in its exterior.
For points $y \in C \cap V_1$ with $\langle y, v \rangle = 0$ we have
to repeat this procedure. Considering the subspace
$V_2 = \{u \in V_1 : \langle u , v_1 \rangle =0\}$,
we define the sequence of projections $u'_i= \proj_{v_1}(u_i) \in V_2$
and pass to a further
subsequence which converges to a vector $v_2$
\begin{equation*}\label{dir}
\frac{u'_i}{\|u'_i\|} \to v_2 \in \mathbb{S}^{d-3}.
\end{equation*}
Inductively this procedure will produce an {\bf ordered orthonormal basis}
$\{v_0, v_1, v_2, \dots,v_{d-1}\}$ for $\R^d$ and a final
subsequence (which for simplicity we still denote
as $n_i$) such that

\begin{quote}
for each $y \in \T^d$,
$y + n_i\al$ is either eventually in the \\
interior of $D$
or it is eventually in its exterior.
\end{quote}

This is clear for points $y \in \T^d$ such that $y + \beta \not\in C$.
Now suppose we are given a point $y$ with $y + \beta \in C$.
We let $k$ be the first index with
$\langle y + \beta, v_k \rangle \ne 0$. As $\{v_0, v_1, v_2, \dots,v_{d-1}\}$
is a basis for $\R^d$ such $k$ exists.
We claim that
the sequence $y + n_i \al$ is either eventually in the interior of $D$ or it is eventually in its exterior.
To see this consider the affine hyperplane
which is tangent to $C$ at $y +\beta$
(which contains the vectors $\{v_0,\dots, v_{k-1}\}$).
Our assumption implies that the sequence $y + n_i \al$ is either
eventually on the opposite side of this 
hyperplane from the sphere, in which
case it certainly lies in the exterior of $D$, or it eventually lies on the same side as the sphere.
However in this latter case it can not be squeezed
in between the sphere and the tangent hyperplane, as this would imply
$\langle y + \beta, v_k \rangle = 0$, contradicting our assumption.
Thus it follows that in this case the sequence
$y + n_i \al$ is eventually in the interior of $D$.

{\bf 3.}\
Let now $p$ be an element of $E(\sig, X)$.
We choose a {\bf net} $\{n_\nu\} \subset \Z$ with $\sig^{n_\nu} \to p$.
It defines uniquely an element $\beta \in E(Y) \cong \T^d$ so that
$\pi(px) = \pi(x) + \beta$ for every $x \in X$.
Taking a subnet if necessary we can assume
that the net $\frac{\beta - n_\nu\al}{\|\beta - n_\nu \al\|}$ converges  to some $v_0 \in S^{d-1}$. And, as above, proceeding by induction, we assume likewise that all the corresponding limits $\{v_0,\dots, v_{k-1}\}$ exist.

Next we choose a {\bf sequence} $\{n_i\}$ such that
$n_i \al \to \beta$,
$\frac{\beta - n_i\al}{\|\beta - n_i \al\|} \to v_0$ etc.,
We conclude that $\sig^{n_i} \to p$.
Thus every element of $E (\sig,X)$ is obtained as
a limit of a sequence in $\Z$ and is therefore of Baire class 1.
\end{proof}

\begin{remark}
From the proof we see that the elements of $E(\sigma,X) \setminus
\Z$ can be parametrized by the set
$\T^d \times \mathcal F$, where $\mathcal F$ is the collection of
ordered orthonormal bases for $\R^d$,
$\ p \mapsto (\beta, \{v_0,\dots, v_{d-1}\})$.
\end{remark}

For further recent results on tame systems see \cite{Pikula} and \cite{Auj}.
Below we will study the question whether some coding functions are tame.

\begin{defin} \label{d:tametype1} \
\begin{enumerate}
\item
Let $S \times X \to X$ be an action on a (not necessarily compact) space
$X$,
$f: X \to \R$ a bounded (not necessarily continuous)
function, $h: S_0 \to S$ a homomorphism of semigroups and $z \in X$.
The following function will be called a \emph{coding function}:
$$
m(f,z) : S_0 \to \R, \ s \mapsto f(h(s)z).
$$

\item
When $S_0=\Z^k$ and $f(X)=\{0,1,\dots,d\}$ we say
that $f$ is a 
\emph{$(k,d)$-code}. Every such code generates a point transitive subshift of $A^{\Z^k}$,
where $A=\{0,1, \dots, d\}$.
In the particular case of the characteristic function $\chi_D: X \to \{0,1\}$
for a subset $D \subset X$ and $S_0=\Z$ we get a $(1,1)$-code, i.e.
a binary function $m(D,z): \Z \to \{0,1\}$
which generates a $\Z$-subshift of the Bernoulli shift on $\{0,1\}^{\Z}$.
\end{enumerate}
\end{defin}

\begin{question}
When is a coding function tame ?
\end{question}

It follows from results in \cite{GM-rose} that a coding bisequence
$c: \Z \to \R$ (with $S_0:=\Z$) is tame iff it can be represented
as a generalized matrix coefficient of a Rosenthal Banach space representation.
That is, iff there exist: a Rosenthal Banach space $V$, a linear isometry $\s \in \Iso(V)$ and
two vectors $v \in V$, $\varphi \in V^*$ such that
$$
c_n=\langle \s^n(v),\varphi \rangle = \varphi(\s^n(v)) \ \ \ \ \forall n \in \Z.
$$

We will see that many coding functions are tame, including some
multidimensional analogues of Sturmian sequences.
The latter are defined on the groups $\Z^k$ and instead of the characteristic function
$f:=\chi_D$ (with $D=[0,c)$) one may consider coloring of the space leading to shifts with finite alphabet.
Here we give a precise definition which (at least in some partial cases)
was examined in several papers.
Regarding some dynamical and combinatorial aspects of coding functions (like multidimensional Sturmian sequences)
see for example \cite{BV,Fernique,Pikula}, and the survey paper \cite{BFZ}.

\begin{defin} \label{d:mSturm} (Multidimensional Sturmian sequences)
	Consider an arbitrary finite partition 
	$$\T=\cup_{i=0}^{d} [c_i,c_{i+1})$$ of $\T$
	by the ordered $d$-tuple of points $c_0=0,c_1, \dots, c_{d}, c_{d+1}=1$ and 
	define the natural function
	$$
	f: \T \to A:=\{0, \dots ,d\}, \ \ \ f(t)=i \ \text{iff} \ t \in [c_i,c_{i+1}).
	$$
	Now for a given $k$-tuple
	$(\a_1,\dots ,\a_k) \in \T^k$  and a given point $z \in \T$ consider the corresponding coding function
	$$
	m(f,z): \Z^k \to \{0, \dots ,d\} \ \ \ (n_1, \dots,n_k) \mapsto f(z + n_1 \a_1 + \cdots + n_k \a_k).
	$$
	We call such a sequence a \emph{multidimensional 
		$(k,d)$-Sturmian like sequence}. 
\end{defin}

Lemma \ref{l:tametype} and Remark \ref{r:coding} below demonstrate
the relevance of Definition \ref{d:tametype1} for coding functions.
By Theorem \ref{t:tame-f}
a continuous function $f: X \to \R$ on a compact $S$-system $X$ is tame iff
$fS$ does not contain an independent sequence.
This fact justifies our terminology in the following definition.
For the definition of an independent sequence of functions see Definition \ref{d:ind}. 

\begin{defin} \label{d:tametype2}
Let $S$ be a semigroup, $X$ a (not necessarily compact) $S$-space and
$f: X \to \R$ a bounded (not necessarily continuous) function.
We say that $f$ is of
\emph{tame-type} if the orbit $fS$ of $f$ in
$\R^X$ 
is a tame family 
(Definition \ref{d:tameF}). 
\end{defin}

An example of a Baire 1 tame-type function
which is not tame, being discontinuous,
is the characteristic function $\chi_D$ of an arc $D=[a,a+s)\subset \T$ defined on the system $(R_{\a}, \T)$, where $R_{\a}$
is an irrational rotation of the circle $\T$.
See Theorem \ref{t:tame-typeOften}.3.

\begin{lem} \label{l_1} \
\ben
\item
Let $q: X_1 \to X_2$ be a map between sets and
$\{f_n: X_2 \to \R\}$ a bounded sequence of functions \emph{(with no continuity assumptions on $q$ and $f_n$)}.
If $\{f_n \circ q\}$ is an independent sequence on $X_1$
then $\{f_n\}$ is an independent sequence on $X_2$.
\item
If $q$ is onto then the converse is also true. That is $\{f_n \circ q\}$ is independent if and only if $\{f_n\}$ is independent.
\item
Let $\{f_n\}$ be a bounded sequence of continuous
functions on a topological space $X$.
Let $Y$ be a \emph{dense} subset of $X$. Then $\{f_n\}$ is an independent sequence on $X$ if and only if
the sequence of restrictions
$\{f_n|_Y\}$ is an independent sequence on $Y$.
\een
\end{lem}
\begin{proof}
Claims (1) and (2) are straightforward.

(3)
 Since $\{f_n\}$ is an independent sequence 
for every pair of finite disjoint sets
$P, M \subset \N$,
the set
$$
\bigcap_{n \in P} f_n^{-1}(-\infty,a) \cap  \bigcap_{n \in M} f_n^{-1}(b,\infty)
$$
is non-empty. 
This set is open because every $f_n$ is continuous. 
Hence, each of them meets the dense set $Y$.
As
$f_n^{-1}(-\infty,a) \cap Y= f_n|_Y^{-1}(-\infty,a)$ and
$f_n^{-1}(b,\infty) \cap Y= f_n|_Y^{-1}(b,\infty)$,
this implies that $\{f_n|_Y\}$ is an independent sequence on $Y$.

Conversely if $\{f_n|_Y\}$ is an independent sequence on a subset $Y \subset X$ then by (1)
(where $q$ is the embedding $Y \hookrightarrow X$),
$\{f_n\}$ is an independent sequence on $X$.
\end{proof}

\begin{lem} \label{l:tametype}  In terms of Definition \ref{d:tametype2} we have:
\ben
\item
Every tame function $f: X \to \R$ is of tame-type.
\item
If $X$ is compact every continuous tame-type function $f$ is tame.
\item
Let $f \in \RUC(X)$; then
$f \in \Tame(X)$ if and only if $f$ is tame-type. Moreover,
there exists an $S$-compactification $\nu: X \to Y$ where the action $S \times Y \to Y$ is continuous,
$Y$ is tame and $f=\tilde{f} \circ \nu$ for some $\tilde{f} \in C(Y)$.
\item
Let $G$ be a topological group and $f \in \RUC(G)$. Then  
$f \in \Tame(G)$ 
if and only if $fG$ is a tame family. 
\item
Let $L$ be a discrete semigroup and $f: L \to \R$ a bounded function. Then $f \in \Tame(L)$ if and only if $fL$ is a tame family. 
\item
Let $h: L \to S$ be a homomorphism of semigroups,  
$S \times Y \to Y$ be an action (without any continuity assumptions) on a set $Y$ and $f: Y \to \R$ be 
a bounded function such that $fL$ is
a tame family. 
Then for every point $y \in Y$ the corresponding coding function $m(f,y): L \to \R$ is tame on the discrete semigroup
$(L,\tau_{discr})$.
\een
\end{lem}
\begin{proof}
For (1), (2) and  (3) use Lemma \ref{l_1} and Theorem \ref{t:tame-f}.

For (3)
consider the cyclic $S$-compactification $\nu: X \to Y=X_f$
(see Definition \ref{d:cyclic}). Since $f \in \RUC(X)$ the action $S \times X_f \to X_f$ is jointly continuous (Remark \ref{r:cycl-comes}).
By the basic property of the cyclic compactification there exists a continuous function $\tilde{f}: X_f \to \R$
such that $f=\tilde{f} \circ \nu$. Since $f$ is of tame-type the family $fS$ has no independent sequence.
By Lemma \ref{l_1}.3 we conclude that also $\tilde{f} S$ has no independent sequence.
This means, by Theorem \ref{t:tame-f}, that $\tilde{f}$ is tame. Hence (by Definition \ref{d:funct}) so is $f$.
The converse follows from (1) (or, directly from Lemma \ref{l_1}.1).

(4) and (5)
follow easily from (3) (with $X=G=L$)
 taking into account (1) and the fact that on a
discrete semigroup $L$ every bounded function
$L \to \R$ is in $\RUC(L)$.

(6)
 By (5) it is enough to show that
the coding function $f_0:=m(f,y)$ is of tame-type on $L$. That is, we have to show that $f_0L$ has no independent subsequence.
Define $q: L \to Y, s \mapsto h(s)y$.
Then $f_0t=(ft) \circ q$ for every $t \in L$. If $f_0t_n$ is an independent sequence for some sequence $t_n \in L$ then Lemma \ref{l_1}.1
implies that the sequence of functions $ft_n$ on $Y$ is independent.
This contradicts the assumption that $fL$ has no independent
subsequence.
\end{proof}

\begin{remark} \label{r:coding} Regarding Lemma \ref{l:tametype}.6
note that the following conditions are equivalent:
\ben
\item The coding function $f_0=m(f,z): L \to \R$ is a tame function on the semigroup $(L,\tau_{discr})$.
\item
The
cyclic
system $X_{f_0} \subset {A}^L$, induced by the function $f_0$ with $A:=f(h(L)z$),
is a tame $L$-system
(note that in particular for the characteristic function $f=\chi_D$ of $D \subset X$ we get a (cyclic) subshift $X_{f_0}  \subset \{0,1\}^L$ induced by the function $f_0:=m(\chi_D,z)$).

\item The orbit $f_0L$ is a tame family. 
\een
\end{remark}

For  (1) $\Leftrightarrow$ (2) observe that any cyclic space $X_{f_0}$ is a subshift of
${A}^L$ for any bounded function $f_0: L \to A$.
By the basic minimality property of the cyclic system $X_{f_0}$ we obtain that it is a factor of any tame system $(L,Y)$ which
realizes $f_0$.
Hence, $f_0$ is tame iff $(L,X_{f_0})$ is tame. Alternatively, use Theorem \ref{t:tame-f}.
For (1) $\Leftrightarrow$ (3) apply Lemma \ref{l:tametype}.4.

\sk

Let $f: X \to Y$ be a function between topological spaces. We denote by $cont(f)$ and $disc(f)$ the points of continuity and discontinuity for $f$
respectively.

\begin{defin} \label{d:EvCont}
Let $F$ be a family of functions on $X$. We say that $F$ is:
\ben
\item \emph{Strongly almost continuous}
if for every $x \in X$
we have $x \in cont(f)$  for
almost all $f \in F$ (i.e. with the exception of at most a finite set of elements
which may depend on $x$).
\item
 \emph{Almost continuous}
if for every infinite (countable) subset $F_1 \subset F$ there exists an
infinite subset $F_2 \subset F_1$ such that $F_2$ is strongly almost continuous on $X$.
\een
\end{defin}

\begin{ex} \label{ex:event} \
\ben
\item Let $G \times X \to X$ be a group action, $G_0 \leq G$
a
subgroup and $f: X \to \R$
a function such that
$$
G_0x \cap disc(f) \ \text{and} \ St(x) \cap G_0 \ \  \text{are \ finite}  \ \forall x \in X,
$$
where $St(x) \leq G$ is the stabilizer subgroup of $x$.
Then the family $fG_0$ is strongly almost continuous
(indeed, use the following equality $g^{-1}cont(f)=cont(fg), \ g \in G$).
\item
A coarse sufficient condition for (1) is:
$disc(f)$ is finite and $St(x) \cap G_0$ is finite \ $\forall x \in X$. 
\item As a particular case of (2)
we have the following example.
For every compact group $G$ and a function $f: G \to \R$ with finitely many discontinuities,
$fG_0$ is strongly almost continuous on $X=G$ for every subgroup $G_0$ of $G$.
\een
\end{ex}

\begin{thm} \label{2811}
Let $X$ be a compact metric space and
$F$ a bounded family of real valued functions on $X$ such that $F$ is almost continuous.
Further assume that:
\bit
\item [(*)]
for every sequence $\{f_n\}_{n \in \N}$ in $F$ there exists a subsequence $\{f_{n_m}\}_{m \in \N}$ and a countable
   subset $C \subset X$ such that
   $\{f_{n_m}\}_{m \in \N}$ pointwise converges on $X \setminus C$
   to a function $\phi: X \setminus C \to \R$ such that $\phi \in \B_1(X \setminus C)$.
\eit
Then $F$ is a tame family. 
\end{thm}

%
\begin{proof} 
Assuming the contrary let $\{f_n\}$ be an independent sequence in $F$. Then,
by assumption,
there exists a countable subset $C \subset X$ and a subsequence $\{f_{n_m}\}$ such that
$\{f_{n_m}: X \setminus C \to \R\}$ pointwise converges on $X \setminus C$
   to a function $\phi: X \setminus C \to \R$ such that $\phi \in \B_1(X \setminus C)$.

Independence is preserved by subsequences so
this subsequence $\{f_{n_m}\}$ remains independent. For simplicity of notation assume that $\{f_n\}$ itself has the properties of $\{f_{n_m}\}$.
Moreover we can suppose in addition,
by Definition \ref{d:EvCont}, that $\{f_n\}$ is strongly almost continuous.

By the definition of independence, for every pair of disjoint finite sets
$P,M \subset \N$, there exist  $a < b$ such that
$$
\bigcap_{n \in P} A_n \cap  \bigcap_{n \in M} B_n \neq \emptyset,
$$
where $A_n:=f_n^{-1}(-\infty,a)$ and $B_n:=f_n^{-1}(b,\infty)$.
Now define
a tree of nested sets
as follows:

\hskip 6.5cm $\Om_1:=X$

\hskip 2cm $\Om_2:=\Om_1 \cap A_1=A_1$ \hskip 4cm  $\Om_3:=\Om_1 \cap B_1=B_1$

$\Om_4:=\Om_2 \cap A_2  \hskip 2cm \Om_5:=\Om_2 \cap B_2 \hskip 1.2cm \Om_6:=\Om_3\cap A_2 \hskip 1cm \Om_7:=\Om_3 \cap B_2$,

and so on.
In general, $$\Om_{2^{n+1} + 2k}:=\Om_{2^n+k} \cap A_{n+1},  \hskip 0.3cm  \Om_{2^{n+1} + 2k+1}:=\Om_{2^n+k} \cap B_{n+1}$$
for every $0 \leq k < 2^n$ and every $n \in \N$.

We obtain a system $\{\Om_n\}_{n \in \N}$ which satisfies:

$$
\Om_{2n} \cup \Om_{2n+1} \subset \Om_n
\ {\text{ and}}\   \Om_{2n} \cap \Om_{2n+1} =\emptyset
\  {\text{for each}}\  n \in \N.
$$

\sk

Since $\{(A_n,B_n)\}$ is independent, every $\Om_n$ is nonempty.

For every binary sequence $u=(u_1,u_2, \dots) \in \{0,1\}^{\N}$ we have the corresponding uniquely defined
\emph{branch}
$$
\a_u:=\Om_1 \supset \Om_{n_1} \supset \Om_{n_2} \supset \cdots
$$
where for each $i \in \N$ with $2^{i-1} \leq n_i < 2^{i}$ we have
$$n_{i+1}=2n_i \ \text{iff} \ u_i=0 \ \text{and} \ n_{i+1}=2n_i+1 \ \text{iff} \ u_i=1.$$

Let us say that $u, v \in \{0,1\}^{\N}$ are \emph{essentially distinct} if they have infinitely many different coordinates.
Equivalently, if $u$ and $v$ are in different cosets of the Cantor group $\{0,1\}^{\N}$  with respect to the subgroup $H$ consisting of
the binary sequences with finite support. Since $H$ is countable there are uncountably many pairwise essentially distinct
elements in the Cantor group. We choose a subset $T \subset \{0,1\}^{\N}$ which intersects each coset in exactly one point.  Clearly, 
$card (T)=2^{\omega}$. Now 
for every branch $\a_u$ where $u \in T$
choose one element
$$x_{u} \in \cap_{i \in \N} cl(\Om_{n_i}).$$
Here we use the compactness of $X$ which guarantees that $\cap_{i \in \N} cl(\Om_{n_i}) \neq \emptyset$.
We obtain a set $X_T:=\{x_{u}: \ u \in T\} \subset X$ and a function $T \to X_T, \ u \mapsto x_u$.

\sk

\textbf{Claim:}
\ben
\item  The function $T \to X_T, \ u \mapsto x_u$ is injective. In particular,
$X_T$ is uncountable.
\item
$|\phi(x_u) - \phi(x_v)| \geq \eps:=b-a$ for every
distinct $x_u, x_v \in X_T \setminus C$.
\een

\sk
\textbf{Proof of the Claim:} (1)
Let $u=(u_i)$ and $v=(v_i)$ are distinct elements in $T$. Denote by
$\a_u:=\{\Omega_{n_i}\}_{i \in \N}$ and $\a_v:=\{\Omega_{m_i}\}_{i \in \N}$ the corresponding branches.
Then, by the definition of $X_T$, we have the uniquely defined points
$x_u \in \cap_{i \in \N} cl(\Om_{n_i})$ and
$x_v \in \cap_{i \in \N} cl(\Om_{m_i})$ in $X_T$.

Since $u,v \in T$ are essentially distinct they have infinitely many different indices.

As $\{f_n\}$ is strongly almost continuous
there exists a sufficiently large $t_0 \in \N$ such that
the points $x_u$ and $x_v$ are both points of continuity of
$f_{n}$ for every $n \geq t_0$.

Now note that if
$u_{i} \neq v_{i}$ then
the sets $\Omega_{n_{i}+1}$ and $\Omega_{m_{i}+1}$ are contained
(respectively) in the
pair of disjoint sets
$A_k:=f_k^{-1}(-\infty,a)$ and $B_k:=f_k^{-1}(b,\infty)$.
Since $u$ and $v$ are essentially distinct we can assume that $i$ is sufficiently large in order to ensure that $k \geq t_0$.
That is, we necessarily have exactly one of the cases:
 $$(a) \ \  \ \Omega_{n_{i}+1} \subset A_k, \quad \Omega_{m_{i}+1} \subset B_k$$ or
 $$(b) \ \ \  \Omega_{n_{i}+1} \subset B_k, \quad \Omega_{m_{i}+1} \subset A_k.$$

 For simplicity we only check the first case (a). For (a) we have
$x_u \in \cls(\Omega_{n_{i}+1}) \subset \cls(f_{k}^{-1}(-\infty,a))$ and $x_v \in \cls(\Omega_{n_{i}+1}) \subset \cls(f_{k}^{-1}(b,\infty))$.
Since $\{x_u,x_v\} \subset cont(f_{n})$ are continuity points for every $n \geq t_0$ and since $k \geq t_0$ by our choice,
we obtain $f_{k}(x_u) \leq a$ and $f_{k}(x_v) \geq b$. So, we can conclude
that $|f_{k}(x_u) -f_{k}(x_v)| \geq \eps:=b-a$ for every $k \geq t_0$. In particular, $x_u$ and $x_v$ are distinct.
This proves (1).
Furthermore, if our distinct $x_u, x_v \in X_T$ are in addition from $X_T \setminus C$ then
$\lim f_k(x_u)=\phi(x_u)$ and $\lim f_k(x_v)=\phi(x_v)$. It follows that
$|\phi(x_u) - \phi(x_v)| \geq \eps$ and
the condition (2) of our claim is also proved.

\sk

Since $X_T \setminus C$ is an uncountable subset of a Polish space $X$ there exists an uncountable subset $Y \subset X_T \setminus C$ such that any point of $y$ is a condensation point in $X_T \setminus C$ (this follows from the proof of Cantor-Bendixson theorem, 
\cite{Kech}). 
For every open subset $U$ in $X$ with
$U \cap Y \neq \emptyset$ we have $\diam(\phi(U \cap Y)) \geq \eps$. This
means that $\phi: X \setminus C \to \R$ is not fragmented. Since $C$ is countable and $X$ is compact metrizable the subset $X \setminus C$ is
Polish. On Polish spaces fragmentability and Baire 1 property are the same for real valued functions (Lemma \ref{r:fr1}.2).
So, we obtain that $\phi: X \setminus C \to \R$ is not Baire 1. This contradicts
the assumption that $\phi \in \B_1(X \setminus C)$.
\end{proof}

\begin{thm} \label{c:2811}
Let $X$ be a compact metric space,
and $F$ a bounded family of real valued functions on $X$ such that $F$ is almost continuous.
Assume that $\cls_p(F) \subset \B_1(X)$.
Then $F$ is a tame family. 
\end{thm}
\begin{proof} Assuming the contrary let $F$ has an independent sequence $F_1:=\{f_n\}$.
Since $\cls_p(F) \subset \B_1(X)$ we have $\cls_p(F_1) \subset \B_1(X)$.
By the BFT theorem \cite[Theorem 3F]{BFT} the compactum $\cls_p(F_1)$ is a Fr\'echet topological space.
Every (countably) compact Fr\'echet
space is sequentially compact, \cite[Theorem 3.10.31]{Eng},
hence
$\cls_p(F_1)$ is sequentially compact.
Therefore the sequence $\{f_n\}$
contains a pointwise
convergent subsequence, say $f_n \to \phi \in \B_1(X)$.
Now apply Theorem \ref{2811}
to get a contradiction, taking into account that
the properties almost continuity and independence are both inherited
by subsequences.
\end{proof}


\begin{thm} \label{t:tame-typeOften} \
\ben
\item Let $X$ be a compact metric $S$-space, $S_0$
a subsemigroup of $S$.  Let $f: X \to \R$
be a bounded function such that $\cls_p(fS_0) \subset \B_1(X)$
and  with $fS_0$ almost continuous. 
Then $fS_0$ is 
a tame family. 

\item
Let $X$ be a compact metric $G$-space
 and $G_0 \leq G$
 a subgroup of $G$ such that (i) $(G_0,X)$ is tame, (ii)
 for every $p \in E(G,X)$ and every $x \in X$ the preimage $p^{-1}(x)$ is countable
 and (iii) $G_0 \cap St(x)$ is finite for every $x \in X$.
 Suppose further that $f: X \to \R$ is a bounded function with only finitely many points of discontinuity. 
 Then $fG_0$ is
 a  tame family. 
 \item
 In particular, (2) holds in the following useful situation: $X=G$ is a compact metric group,
 $h: G_0 \to X$ is a homomorphism of groups and $f: X \to \R$ has finitely many discontinuities. 
 Then $fG_0$ is 
 a tame family.
 
 \item
In all the cases above (1), (2), (3), a coding function $m(f,z): S_0 \to \R$ is a tame function
on the discrete semigroup $S_0$ for every $z \in X$ and every homomorphism
$h: S_0 \to S$ (or, $G_0 \to G$). Also the corresponding subshift $(S_0, X_f)$ is tame.

\een
\end{thm}
\begin{proof} (1) Apply Theorem \ref{c:2811}.

(2)
First note that $fG_0$ is strongly almost continuous (Example \ref{ex:event}.2).
Now assuming the contrary $fG_0$ has an independent subsequence $\{fg_n\}$. Since
$(G_0, X)$ is tame we can assume, with no loss in generality,
that the sequence $\{g_n\}$ converges to an element $p \in E(G_0,X)$.
Then $f(g_n(x))$ converges to $f(px)$ for every $x \in X \setminus C$, where
$C: = p^{-1}( disc(f))$ is a countable set.
Since $X$ is a tame system, $p: X \to X$ is a fragmented function.
Then also the restricted function $p_0: X \setminus C \to X$ is fragmented.
Since $f: X \to \R$ is uniformly continuous we obtain that the composition $f \circ p_0: X \setminus C \to \R$ is fragmented.
Since $C$ is countable, $X \setminus C$ is Polish. Therefore, by Lemma \ref{r:fr1}.2,
$f \circ p_0: X \setminus C \to \R$ is Baire 1.
This however is in contradiction with Theorem \ref{2811}.

(4) Follows from (1), Lemma \ref{l:tametype} and Remark \ref{r:coding}.
\end{proof}

Theorem \ref{t:tame-typeOften}.4 directly implies the following:

\begin{ex}
For every irrational rotation $\alpha$ of the circle $\T$ and an
arc \ $D:=[a,b) \subset \T$
the function $$\varphi_D:=\Z \to \R, \quad n \mapsto \chi_D(n \alpha)$$
is a tame function on the group $\Z$.
In particular, for $D:=[-\frac{1}{4},\frac{1}{4})$ we get that $\varphi_D(n)=\sgn \cos (2 \pi n \a)$ is a tame function on $\Z$.
\end{ex}

\begin{thm} \label{multi}
The multidimensional Sturmian $(k,d)$-sequences
$\Z^k \to \{0,1, \dots, d\}$ (Definition \ref{d:mSturm}) are tame.
\end{thm}
\begin{proof} In terms of Definition \ref{d:mSturm} consider the homomorphism
$$h: \Z^k \to \T, \ \ \ (n_1, \dots,n_k) \mapsto n_1 \a_1 + \dots + n_k \a_k.$$
The function $f$ induced by a given partition $\T=\cup_{i=0}^{d} [c_i,c_{i+1})$
$$
f: \T \to A:=\{0, \dots, d\}, \ \ \ f(t)=i \ \text{iff} \ t \in [c_i,c_{i+1}).
$$
has only finitely many discontinuities.
Now Theorem \ref{t:tame-typeOften} (items (3) and (4)) guarantees that the corresponding
$(k,d)$-coding function $m(f,z): \Z^k \to A \subset \R$ is tame for every $z \in \T$.
\end{proof}

At least for $(1,d)$-codes, Theorem \ref{multi}, can be derived also from results of Pikula \cite{Pikula},
and of Aujogue \cite{Auj}.  
See also Remark \ref{applicBV}.

\begin{lem} \label{claim}
Let
$\pi: X \to Y$ be a continuous onto $S$-map of compact metric $S$-systems. Set
$$
X_0:=\{x \in X: \ \ |\pi^{-1}(\pi(x))|=1\}.
$$
Then the restriction map $\pi: X_0 \to Y_0$ is a topological homeomorphism of $S$-subspaces,
where $Y_0:=\pi(X_0)$.
\end{lem}
\begin{proof} First observe that $X_0$ and $Y_0$ are $S$-invariant
and that $\pi: X_0 \to Y_0$ is an onto, continuous, 1-1 map.
For every converging sequence $y_n \to y$, where $y_n,y \in Y_0$ the preimage $\pi^{-1}(\{y\} \cup \{y_n\}_{n \in \N})$ is a compact subset of $X$. On the other hand, $\pi^{-1}(\{y\} \cup \{y_n\}) \subset X_0$ by the definition of $X_0$. It follows that the restriction of $\pi$ to $\pi^{-1}(\{y\} \cup \{y_n\})$ is a homeomorphism. In particular, $\pi^{-1}(y_n)$ converges to $\pi^{-1}(y)$.
\end{proof}
\
Recall that a map $\pi: X \to Y$ as above
is said to be an \emph{almost one-to-one extension} if $X_0$
is a residual subset of $X$.
As a corollary of Theorem \ref{2811} one can derive the following result which
generalizes
the above mentioned result of W. Huang from Example \ref{e:tameNOThns}.2.

\begin{thm} \label{Huang}
Let $\pi: X \to Y$ be a homomorphism of compact metric $S$-systems
such that
$X \setminus X_0$ is countable,
where
$$
X_0:=\{x \in X: \ \ |\pi^{-1}(\pi(x))|=1\}.
$$
Assume that $(S,Y)$ is tame and that
the set
$p^{-1}(y)$ is (at most) countable for every $p \in E(Y)$ and $y \in Y$
(e.g., this latter condition is always satisfied when $Y$ is distal).
Then $(S,X)$ is also tame.
\end{thm}

\begin{proof}
We have to show that every $f \in C(X)$ is tame. Assuming the contrary, suppose $fS$
contains an independent sequence $fs_n$.
Since $Y$ is metrizable and tame,
one can assume (by Theorem \ref{D-BFT}) that
the sequence $s_n$  converges pointwise to some element $p$ of $E(S,Y)$.
Consider the set $Y_0 \cap p^{-1}Y_0$, where $Y_0 = \pi(X_0)$. Since $p^{-1}(y)$ is countable for every $y \in Y \setminus Y_0$ it follows that
$Y \setminus (Y_0 \cap p^{-1}Y_0)$ is countable.
Therefore, by the definition of $X_0$ and the countability of $X \setminus X_0$, we see that
$X \setminus \pi^{-1}(Y_0 \cap p^{-1}Y_0)$ is also countable.
Now observe that the sequence
$(fs_n)(x)$ converges 
for every $x \in \pi^{-1}(Y_0 \cap p^{-1}Y_0)$. 
Indeed if we denote $y=\pi(x)$ then
$s_ny$ converges to $py$ in $Y$. In fact we have $py \in Y_0$ (by the choice of $x$) and $s_n y \in Y_0$. 
By Lemma \ref{claim}, 
$\pi: X_0 \to Y_0$ is an $S$-homeomorphism. So we obtain that $s_n x$ converges
to $\pi^{-1}(py)$ in $X_0$.
Since $f: X \to \R$ is continuous, $(fs_n) (x)$ converges to $f(\pi^{-1}(py))$ in $\R$.
Every $fs_n$ is a continuous function, hence so is also its restriction to
$\pi^{-1}(Y_0 \cap p^{-1}Y_0)$.
Therefore the limit function 
$\phi: \pi^{-1}(Y_0 \cap p^{-1}Y_0) \to \R$ 
is Baire 1. Since C:=$X \setminus \pi^{-1}(Y_0 \cap p^{-1}Y_0)$ 
is countable and $fs_n$ is an independent sequence, 
Theorem \ref{2811} provides the sought-after contradiction.
\end{proof}

\section{Order preserving systems are tame}
\label{s:order}

In this section all group actions are jointly continuous
and representations of systems (and groups) are strongly continuous.

\subsection{Order preserving action on the unit interval}
\label{s:orderOnI}

Recall that for the group $G=H_+[0,1]$ (of orientation preserving self-homeomorphisms)
the $G$-system $X=[0,1]$  with the obvious $G$-action is tame \cite{GM-AffComp}.
One way to see this is to observe that the enveloping semigroup of this dynamical
system naturally embeds into the Helly compactum (and hence is a Rosenthal compactum).
By Theorem \ref{D-BFT}, $(G,X)$ is tame.
By Theorem \ref{t:tame}  this means that every $p \in E(X)$ is a Baire class 1 map. In fact we can say more.
As every monotonic map $[0,1] \to [0,1]$ has at most countably many discontinuities,
this holds also for every $p \in E(G,X)$.

We list here some other properties of $H_+[0,1]$.


\begin{thm}\label{H_+[0,1]}
Let $G:=H_+[0,1]$. Then
\ben
\item  \emph{(Pestov \cite{Pest98})} \ $G$ is extremely amenable.
\item \cite{GM-suc} \ $\WAP(G)=\Asp(G)=\SUC(G)=\{constants\}$ and every Asplund  representation of $G$ is trivial.
\item \cite{GM-AffComp} \ $G$ is representable on a (separable) Rosenthal space.
\item
\emph{(Uspenskij \cite[Example 4.4]{UspComp})}
\ $G$ is Roelcke precompact.
\item
$\UC(G) \subset \Tame(G)$,
that is, the Roelcke compactification of $G$ is tame.
\item $\Tame(G) \neq \UC (G)$.
\item $\Tame(G) \neq \RUC(G)$, that is, $G$ admits a transitive dynamical system which is not tame.
\item  \cite{MePolev}  \ $H_+[0,1]$ and $H_+(\T)$ are minimal groups.
\een
\end{thm}

In properties (5) and (6) we answer two questions of T. Ibarlucia
which are related to
 \cite{Ibar}.

\begin{proof}
(3)  See \cite{GM-AffComp} (or Theorem \ref{t:tame=R-repr}.1 with Lemma \ref{GrRepr}).

(5) (Sketch)
Consider the Roelcke G-compactification $G \to R$ of $G$. That is, the compactification of $G$ induced by the algebra $UC(G):=\LUC(G) \cap \RUC(G)$.
One can show that, in the present case, this compactification is a $G$-factor
of the Ellis compactification $G \to E$, where $E=E(G,[0,1])$ is the enveloping semigroup
for the action $G \times [0,1] \to [0,1]$, which is tame as we mentioned above.
As in \cite{GM-suc} one can apply a characterization of
Uspenskij, for elements of $R$; they can be identified with some special relations on $[0,1]$.  Namely,
those (connected) curves in the square $[0,1] \times [0,1]$ which connect the points $(0,0)$ and $(1,1)$ and never go down.
These are not functions in general and may have vertical intervals (as well as, horizontal intervals) in their graphs.

For the enveloping semigroup $E=E(G,[0,1])$ we know \cite{GM-AffComp}
that, as a compactum, it is naturally embedded into the Helly compactum
of nondecreasing selfmaps of $[0,1]$.
Each element $p$ of $E$ has at most countably many discontinuity points, where left and right limits both exist.
Our aim is to find a $G$-factor map $f: E \to R$.
Let $p \in E$.  At each discontinuity point $x \in [0,1]$ of the function $p: [0,1] \to [0,1]$,
add a vertical interval to the graph of $p$.
That is, we ``fill" the graph by joining the points $(x,y_1)$ and $(x,y_2)$,
where $y_1$ is the left limit of the function $p$ at $x$ and $y_2$ is the right limit.
Then after this operation, repeated at each discontinuity point, $p$ ``becomes"  an element of R which we denote by $f(p)$.
This defines a natural map $f: E \to R$ which is
$G$-equivariant, onto and continuous (but not 1-1) .

(6)  Define $f: G \to [0,1] , f(g)=g(\frac{1}{2})$. Then $f$ is tame (since the system $(G,[0,1])$ is tame) and not left uniformly continuous.

(7)
We will show that for some $g \in G$ the system $(\Exp[0,1],g)$,
induced on the space $\Exp[0,1]$ of closed subsets of $[0,1]$,
contains a subsystem which is isomorphic to a Bernoulli shift.

Define a homeomorphism $g \in G$ as follows.
First set $g(0)=0, g(1)=1$.
Now choose a two sided increasing sequence
$$
T=\{\dots, t_{-2}, t_{-1}, t_0, t_1, t_2, \dots\}
$$
with
$\lim_{i \to -\infty} t_i =0,\  \lim_{i\to\infty} t_i =1$,
and let $g$ map each of the closed intervals determined by this sequence
affinelly onto its right hand neighbor; i.e. $g[t_i,t_{i+1}] =
[t_{i+1},t_{i+2}]$. In particular then $g(t_i)=t_{i+1}$ and we conclude that
$g \rest  T \cup \{0,1\}$ defines a dynamical system which is isomorphic
to the two point compactification of the shift on the integers,
$(\Z_*,\sig)$, where $\Z_*= \Z \cup \{\pm \infty\}$.

Now it is well known (and easily seen) that the induced action on the space of closed subsets, $(\Exp(\Z_*),\sig)$, contains a copy of the full Bernoulli
shift on $\{0,1\}^\Z$.  Thus the same is true for $(\Exp[0,1],g)$.

\end{proof}

Regarding Theorem \ref{H_+[0,1]}.2 we note that recently Ben-Yaacov and Tsankov \cite{BenTsankov}
found some other Polish groups $G$ for which $\WAP(G) = \{constants\}$ (and which are therefore also reflexively trivial).

The group $H_+(\T)$ is Asplund-trivial.
Indeed, it is algebraically simple
\cite[Theorem 4.3]{Ghys} and
contains a copy of $H_+[0,1]=St(z)$ (a stabilizer group of some point $z \in \T$) which is Asplund-trivial \cite{GM-suc}.
Now, as in \cite[Lemma 10.2]{GM-suc}
use an observation of Pestov, which implies that then
any continuous Asplund representation of $H_+(\T)$ is trivial.


%

\begin{thm} \label{t:RP}
The Polish group $G=H_+(\T)$ is Roelcke precompact.
\end{thm}
\begin{proof}
First a general fact: if a topological group $G$ can be represented as
$G=KH$,
where $K$ is a compact subset and $H$ a Roelcke-precompact subgroup then $G$ is also Roelcke-precompact.
This is easy to verify either directly
or by applying \cite[Prop. 9.17]{RD}.
As was mentioned in Theorem \ref{H_+[0,1]}.4,  $H_+[0,1]$ is Roelcke precompact.
Now, observe that in our case
$G = KH$, where $H:=St(1) \cong H_+[0,1]$ is the stability group of
$1 \in \T$ and $K \cong \T$ is the subgroup
of $G$ consisting of the rotations of the circle.
Indeed, the coset space $G/H$ is homeomorphic to $\T$ and there exists a natural continuous section $s: \T \to K \subset G$.
\end{proof}

\subsection{Linearly ordered dynamical systems}

A map $f: (X,\leq) \to (Y,\leq)$ between two (partially) ordered sets is said to be \emph{order preserving} 
or \textit{monotonic} if
$x \leq x'$ implies $f(x) \leq f(x')$ for every $x,x' \in X$.

\begin{defin} \label{d:ord} \
\ben
\item (Nachbin \cite{Nach})
Let $(X,\tau)$ be a topological space and $\leq$ is a partial order on the set $X$. The triple 
$(X,\tau,\leq)$ is said to be a
\emph{compact ordered space} if $X$ is a compact space and the graph of the relation $\leq$ is closed in $X \times X$.
 \item (See \cite[p. 157]{Gl-thesis}) A compact dynamical $S$-system $(X,\tau)$ with a partial order $\leq$ is said to be a \emph{partially ordered dynamical system}
if the graph of $\leq$ is closed in $X \times X$ and every translation $\tilde{s}: X \to X$ is
an order preserving map.
\item
 For every linear order $\leq$ on a set $X$ we have the standard \emph{interval topology} which we denote by $\tau_{\leq}$.
The triple $(X,\tau_{\leq},\leq)$ is said to be a \emph{linearly ordered topological space}
(LOTS). Sometimes we write just $(X,\leq)$, or even simply $X$, where no ambiguity can occur.
\item
We say that a compact dynamical $S$-system $(X,\tau)$  is a \emph{linearly ordered} dynamical system
if there exists a linear order $\leq$ on $X$ such that $\tau=\tau_{\leq}$ is the interval topology and every $s$-translation $X \to X$ is
an $\leq$-order preserving map.
\een
\end{defin}

Corollary \ref{c:GLOTS} below implies that (4) is a particular case of (2).

For every compact $G$-system $X$ there is a natural partial ordering of inclusion on the
hyperspace
$2^X$.
 This makes $2^X$ is a compact partially ordered dynamical $G$-system. See \cite[Section 3]{Gl-thesis} for details and some applications.

Recall that for every linearly ordered set $(X,\leq)$ the rays $(a,\to)$, $(\leftarrow,b)$ with $a,b \in X$ form a subbase for the
 standard \emph{interval topology} $\tau_{\leq}$ on $X$.
It is well known that the interval topology is Hausdorff (and even normal). Moreover
it is easy to see that it is \emph{order-Hausdorff}
in the following sense.

\begin{lem} \label{ordHausd}
Let $(X, \leq)$ be a LOTS. Then for any
 two distinct points $u_1 < u_2$ in $X$
there exist disjoint $\tau_{\leq}$-open neighborhoods
$O_1$ and $O_2$ in $X$ of $u_1$ and $u_2$ respectively such that $O_1 < O_2$,
meaning that $x < y$ for every $(y,x) \in O_2 \times O_1$. In particular, the graph of $\leq$ is closed in $(X,\tau_{\leq}) \times (X,\tau_{\leq})$.
\end{lem}

 \begin{cor} \label{c:GLOTS}
 Any \emph{compact} LOTS is a
compact ordered space in the sense of Nachbin. 
 \end{cor}


\begin{thm} \label{monot}
 Every order preserving map $f: (X,\leq) \to (Y,\leq)$ between compact linearly ordered spaces is fragmented.
 \end{thm}
 \begin{proof} First note that the question can be reduced to the case of $Y:=[0,1]$.
Indeed, by Corollary \ref{c:GLOTS},  $(Y,\tau_{\leq})$
is an ordered space in the sense of L. Nachbin. 
Fundamental results from his book \cite[p. 48 and 113]{Nach} imply that
there exists a point separating family of \emph{order preserving} continuous maps $q_i: Y \to [0,1], \ i \in I$.
Clearly the composition of two order preserving maps is order preserving.
Now by 
Lemma \ref{r:fr1}.9
it is enough to show that every map $q_i \circ f$ is fragmented.
So we can assume that our order preserving function is of the form $f: X \to Y=[0,1]$.
We have to show that $f$ is fragmented. Assuming the contrary, by Lemma \ref{r:fr1}.8,
there exists a closed subset $K \subset X$ and $a <b$ in $\R$ such that 
$K \cap \{f \leq a\}$, $K \cap \{f \geq b\}$
are both dense in $K$.

Choose arbitrarily two distinct points $k_1 < k_2$ in $K$.
By Lemma \ref{ordHausd} one can choose disjoint open neighborhoods
$O_1$ and $O_2$ in $X$ of $k_1$ and $k_2$ respectively such that $O_1 < O_2$.

By our assumption we can choose $x \in O_1 \cap K$ such that $b \leq f(x)$. Similarly, there exists
$y \in O_2 \cap K$ such that $f(y) \leq a$. Since $a<b$ we obtain
$f(x) > f(y)$. On the other hand, $x < y$ (because $O_1 < O_2$),
contradicting our assumption that $f$ is order preserving.
\end{proof}


\sk

Let $(X,\leq)$ be LOTS. 
Denote by $M_+:=M_+(X,\leq)$ the set of order preserving real valued maps on $(X,\leq)$ and by
$C_+:=C_+(X,\leq)$ the set of order preserving continuous
real valued maps.


\begin{thm} \label{t:CLOTSisWRN} \
Let $(X, \leq)$ be a compact LOTS. Then  
\ben
\item $(X,\leq)$ is WRN.
\item  Any bounded subfamily $F \subset C_+(X,\leq)$ is a Rosenthal family for $X$. In  particular, 
$F$ is a tame family. 
\een
\end{thm}
\begin{proof} (2)
Since the natural order is closed in 
$\R^2$, we have $\cls_p(M_+) = M_+$.
By Theorem \ref{monot} we know that $M_+ \subset \F(X)$ (the set of fragmentable functions). Thus,
 $$
 \cls_p(F) \subset \cls_p(C_+) \subset \cls_p(M_+) = M_+ \subset \F(X).
 $$
This means that $F$ is a Rosenthal family for $X$ (Definition \ref{d:Ros-F}). In particular, $F$ does not contain an independent sequence by
Theorem \ref{f:sub-fr}.

(1) By results of Nachbin \cite{Nach} the bounded subset $F:= C_+((X,\leq),[0,1]) \subset C(X)$ separates points of $X$.
By (2), $F$ is a Rosenthal family for $X$. Thus we can apply Theorem \ref{t:WRN} to conclude that $X$ is WRN. 
\end{proof}


\begin{cor} \label{2arrows}
The two arrows space is WRN but not RN.
\end{cor}
\begin{proof}
The two arrows space is WRN by Theorem \ref{t:CLOTSisWRN}.1. It is not RN by a result of Namioka \cite[Example 5.9]{N}.
\end{proof}

\begin{thm} \label{t:OrderedAreTame}
Let $(X, \leq)$ be a compact linearly ordered  dynamical $S$-system.
Then
\ben
\item
The dynamical system $(S,X)$ is 
representable on a Rosenthal Banach space (i.e., WRN).
\item $(S,X)$ is tame.
\item  Any topological subgroup $G \subset H_+(X)$ is Rosenthal representable.
\een
\end{thm}
\begin{proof} (1) As in the proof of Theorem \ref{t:CLOTSisWRN} consider the (point separating and eventually fragmented)
family $F:= C_+((X,\leq),[0,1])$. Since $(S,X)$ is order preserving we obtain that $F$ is $S$-invariant (i.e., $FS =F$).
Now apply Theorem \ref{t:WRN} and Remark \ref{r:JCont1}.

(2) First proof: Apply (1) and Theorem \ref{t:tame}.

Second (direct) proof:
We have to show that every $p \in E(S,X)$ is a fragmented map. Choose a net $\{s_i\}$ in $S$ such that the net $\{j(s_i)\}$
converges to $p$, where $j: S \to E$ is the Ellis compactification.
Since every translation
$j(s) = \tilde{s}: X \to X$ is order preserving, and as the order is a closed relation (Corollary \ref{c:GLOTS}),
it follows that $p$ is also order preserving. Now from Theorem \ref{monot} we can conclude that $p$ is fragmented.

(3) By (1) the system $(G,X)$ is Rosenthal representable. Now by Lemma \ref{GrRepr} (taking into account Remark \ref{r:JCont1})
it follows that $G$ is Rosenthal representable.
 \end{proof}

\subsection{Orderly groups}

Theorem \ref{t:OrderedAreTame}.3 suggests the following:

\begin{defin} \label{d:pseudolinear}
	Let us say that a topological group $G$ is 
	\emph{orderly} if $G$ is a topological subgroup of $H_+(X,\leq)$ for some linearly ordered compact space.
\end{defin}

Thus, by Theorem \ref{t:OrderedAreTame}, every orderly topological group $G$ is Rosenthal representable.
For example, $\R$ is orderly as it can be embedded into $H_+([0,1])$, where $[0,1]$ is treated as the two-point compactification of $\R$.
Recall that it is unknown yet (see \cite{GM-AffComp,GM-survey}) whether every Polish group is Rosenthal representable.
A sufficient condition is that $G$ is embedded into a product of c-orderly topological groups.

%


Recall that an abstract group $G$ is \emph{left-ordered} (left-c-ordered) if there exists a linear order (c-order) on $G$ which is invariant under left translations; see, for example, \cite{DNR}.


\begin{prop} \label{p:orderly} 
	A discrete group $G$ is orderly if and only if $G$ is left-ordered. 
\end{prop}
\begin{proof}  Let $G$ be orderly. 
Thus,
$G$ is a subgroup of $H_+(K)$ for some compact ordered $K$. Then $G$ acts effectively on a linearly ordered set $K$. It is well known that this is equivalent to saying that $G$ is left-ordered. The idea is to use a \emph{dynamically lexicographic order} (see \cite{DNR}) on $G$. 
	
	Conversely, let $(G,\leq)$ be left-ordered. Consider the following bounded order preserving function $f: G \to [0,1]$ 
	with $f(x)=0 \ \forall x < e, f(e)=\frac{1}{2}, f(x)=1 \ \forall e < x$.  
	Then $\pi_f: G \to G_f$ is an order preserving $G$-compactification (by Lemma \ref{l:LO-X_f}). Moreover, $\pi_f(G)$ is  also discrete. 
	Then the induced homomorphism $G \to H_+(G_f)$ is an embedding of discrete $G$.
	 
For an additional proof of this direction observe that the Nachbin's compactification $\nu: G \to Y$ is an order preserving proper $G$-compactification which induces a topological embedding of (discrete) $G$ into $H_+(Y)$. 
\end{proof}

\begin{prop} \label{l:LO-X_f}
	Let $X$ be a linearly ordered $G$-space and $f: X \to [a,b]$ be a $\mathrm{RUC}$ bounded order preserving function.  
	Then the corresponding cyclic $G$-system $X_f$ (Definition \ref{d:cyclic}) is linearly ordered and the function $\tilde{f}: X_f \to [a,b]$ 
	is order preserving. 
\end{prop}
\begin{proof}
	$X \to X_f \subset [a,b]^G$ can be defined as the diagonal function induced by the family of functions $fG$. 
	Every $fg: X \to \R$ is c-order preserving. Then $[a,b]^G$ carries a natural partial order $\ga$. Moreover, it is easy to see that the restriction of $\ga$ on $Y=\cls(X)$ is a linear order which extends the given order on $X$. 
%
%
%
%
%
\end{proof}

\subsection{Orientation preserving actions on the circle}

The following definition is 
the starting point for
one of the approaches to monotone functions 
on the circle. See, for example, \cite{RT}. 

\begin{defin} \label{d:lift} 
Let $f: \T \to \T$ be a not necessarily continuous selfmap on the circle $\T$. 
We say that $f$ is \emph{orientation preserving} (notation: $f \in M_+(\T,\T)$, or $f \in M_+$) 
if there exists, a not necessarily continuous, 
map $F: \R\to \R$ which is a monotonic lift of degree 1. 
More precisely $F$ satisfies the following conditions: 
\ben
\item  $q\circ F= f \circ q$, where $q: \R \to \T=\R / \Z$ is the 
quotient map;
\item  $F: \R\to \R$ is order preserving; 
\item $F(x+1)=F(x)+1$ for every $x \in \R$. 
\een 
In this case we
say that $F$ is a {\it lift} of $f$.   
\end{defin}

\begin{remark} \label{r:lifts}
Let $k$ be a fixed integer. Then 
$F$ is a lift of $f$ iff $F+k$ is. Therefore, 
among
all 
the possible
lifts of $f$, 
one may choose $F$ such that $F(0) \in [0,1)$. 
Clearly, $F(1)=F(0)+1 <2$ and $F(x)\leq F(1) < 2$ for every $x \in [0,1]$. 
The restriction $F^*:[0,1] \to [0,2]$ of $F$ 
to 
$[0,1]$ 
uniquely reconstructs $F$. Indeed, it is easy to see that 
\begin{equation}\label{k} 
F(x):=F^*(\{x\}) + n \ \ \ \ \ \ \forall n \leq x < n+1
\end{equation}  
with $n \in \Z$, where $\{x\}$ is the fractional part of $x \in \R$. 
Equivalently, $F(x)=F^*(\{x\}) + [x]$, where $[x]$ is the integer part. 
 We say that $F$ is a {\it canonical lift} of $f$ and 
that 
$F^*$ is its {\it kernel}.  

Note that for 
an
arbitrary order preserving function $h: [0,1] \to [0,2]$ with 
$h(1)=h(0)+1$ and $q_2 \circ h= f \circ q_1$, 
the function $F(x):=h(\{x\}) + [x], x \in \R$ is a lift of $f$. Observe that $h$ is the kernel $F^*$ of $F$ iff $h(0) <1$. Otherwise, if $h(0)=1$ then $F^*=h-1$. 
\end{remark}

\begin{lem} \label{l:liftP} \
	\ben 
	\item 
Every orientation preserving $f: \T \to \T$ is Baire 1. That is, $M_+(\T,\T) \subset \B_1(\T,\T)$. 
		\item $M_+$ is pointwise closed in $\T^{\T}$;
		\item $M_+$ is a compact right topological submonoid of $\T^{\T}$ with respect to the composition. 
					\een
\end{lem}
\begin{proof}
(1) Let $F$ be the canonical lift of $f$ 
and $F^*: [0,1] \to [0.2]$ be its kernel.  
Let 
$q_1: [0,1] \to \T$ and $q_2: [0,2] \to \T$ be the restrictions of $q: \R \to \T$. 
Then  $ q_2 \circ F= f  \circ q_1$. 
By Lemma \ref{r:fr1}.10  
we conclude that $f$ is fragmented (hence, Baire 1, because $f$ is a map between Polish spaces). 

(2) Let $f \in \cls(M_+)$. Then $f$ is a pointwise limit of some net $f_i \in M_+$. 
For every $f_i$ consider the canonical lifting $F_i$ and its kernel $F_i^*: [0,1] \to [0,2]$. 
Passing to subnets if necessary one may assume that $F_i^*$ pointwise converges in $[0,2]^{[0,1]}$ to some $h: [0,1] \to [0,2]$. 
Then $h$ is order preserving, too. Moreover, since $0 \leq F_i(0) < 1$ we have $h(0) \in [0,1]$ and $h(1)=h(0)+1 \in [0,2]$. It is easy to show 
that
$$
q_2(h(x))=\lim q_2(F_i^*(x))= \lim f_i(q_1(x)) = f(q_1(x)).
$$
for every $x \in [0,1]$.
Now, as in Remark \ref{r:lifts}, 
define $F(x):=h(\{x\}) + [x]$.  
Then $F$ is a lift
of $f$ in the sense of Definition \ref{d:lift}. 
Note that $F$ is not necessarily the canonical lift of $f$ 
(though $h(0) \in [0,1]$ but it is possible that $h(0)=1$).  

(3) Clearly, $F:=id_{\R}$ is a 
lift of $f:=id_{\T}$. So, $id_{\T} \in M_+$. It is plain to show that if $f_1, f_2 \in M_+$ with  
lifts $F_1,F_2$. Then $F_1 \circ F_2$ is a lift for $f_1 \circ f_2$. 
\end{proof}

  Recall the definition of the natural cyclic ordering on $\T$. Identify
  $\T$, as a set, with $[0,1)$ and
  define a ternary relation, a subset $R \subset [0,1)^3$.
  We say that an ordered triple of pairwise disjoint points
  $z, y, x \in [0,1)$
  has cyclic ordering (and write $[z,y,x] \in R$) if
  $(x-y)(y-z)(x-z) > 0$. 
  An injective selfmap $f: \T \to \T$ is said to be (cyclic) order preserving if $f$ preserves $R$, meaning that
  $[z,y,x] \in R$ implies $[f(z),f(y),f(x)] \in R$.
  

 The following lemma is a version of Lemma 1 in \cite[Section 3]{Koz}.

\begin{lem} \label{lift}
Every injective cyclic order preserving selfmap (e,g., order preserving homeomorphism) $f: \T \to \T$ is orientation preserving in the sense of Definition \ref{d:lift}.
\end{lem}
\begin{proof}
Treating the set $\T$ as $[0,1)$ (so that $f$ is defined as a map $[0,1) \to [0,1)$) consider the partition $[0,1)=I_+(f) \cup I_-(f)$, where
$$I_+(f):=\{x \in [0,1): f(x) \geq f(0)\}, \ \ I_-(f):=\{x \in [0,1): f(x) < f(0)\}.$$
Define $F^*:[0,1] \to [0,2]$ by
$$F^*(x)=f(x), \ \text{if} \ x \in I_+, \  \ F^*(x)= f(x)+1, \ \text{if} \ x \in I_-, \ F^*(1)=f(0)+1.$$
 It is easy to see (using the circle ordering) that $I_+(f),I_-(f)$ are intervals, $F^*: [0,1] \to [0,2]$ is order preserving and $q_2 \circ F^*=f \circ q_1$. Then $F$, defined as in \ref{k}, is the desired 
 lift of $f$.  
 
%
 	
\end{proof}

Let $C_+(\T, \T)$ be the topological monoid of all orientation preserving \emph{continuous} selfmaps $\T \to \T$ endowed with the compact open topology.
Then, for every submonoid $S$ (in particular, for any subgroup $G \leq H_+(\T)$) we have a corresponding (orientation preserving) dynamical system $(S,\T)$.

\sk

\begin{thm} \label{GenH_+} \
	\ben
	\item For every submonoid $S$ of $C_+(\T, \T)$ the dynamical system
	$(S,\T)$ is tame. In particular, this is true for any subgroup $S:=G$ of  $H_+(\T)$.
	\item
	$H_+(\T)$ is Rosenthal representable as a Polish topological group.
	\een
\end{thm}
\begin{proof}
	Part (2) follows from (1), Theorem \ref{t:tame=R-repr} and Lemma \ref{GrRepr}.
	For (1) we have to show that every $p \in E(\T)$ is a Baire 1 class function $\T \to \T$.
	By our assumption, $S \subset C_+(\T, \T) \subset M_+(\T, \T)$ and  $M_+(\T, \T)$ is pointwise closed (Lemma \ref{l:liftP}.2). So, we obtain $\cls(S)=E(S,\T) \subset M_+(\T,\T)$.  
	By Lemma \ref{l:liftP}.1 we have $M_+(\T,\T) \subset \B_1(\T,\T)$. Therefore, $p\in E(S,\T) \subset \B_1(\T,\T)$.	
\end{proof}



The Ellis compactification $j: G \to E(G,\T)$ of the group $G=H_+(\T)$ is a topological embedding.
In fact, observe that
the compact open topology on $j(G) \subset C_+(\T, \T)$ coincides
with the pointwise topology.
This observation implies, by \cite[Remark 4.14]{GM-survey} that
$\Tame(G)$ separates points and closed subsets.

Although $G$ is representable on a (separable) Rosenthal Banach space,
we have $\Asp(G)=\{constants\}$ and therefore any Asplund representation of this group is trivial
(this situation is similar to the case of the group $H_+[0,1]$, \cite{GM-suc}).
Indeed, we have $\SUC(G)=\{constants\}$ by \cite[Corollary 11.6]{GM-suc} for $G=H_+(\T)$, and we
recall that for every topological group
$\Asp(G) \subset \SUC(G)$.


\subsubsection{Functions of bounded variation}

\begin{defin} \label{d:BV} 
	Let $(X, \leq)$ be a linearly ordered set. We say that a bounded function $f: X \to \R$ has variation not greater than $r$ if 
	$$
	\sum_{i=0}^{n-1} |f(x_{i+1})-f(x_i)| \leq r
	$$
	for every choice of $x_0 \leq x_1 \leq \cdots \leq x_n$ in $X$. 
	The least upper bound of all such possible sums is the  {\it variation} of $f$. Notation:  $\Upsilon(f)$. If  $\Upsilon(f) \leq r$ then we write $f \in BV_r(X)$. If $f(X) \subset [c,d]$ for some reals $c \leq d$ then we write also $f \in BV_r(X,[c,d])$. 
\end{defin}

Denote by $M_+(X,[c,d])$ the set of all order-preserving functions $X \to [c,d]$. 
Then $M_+(X,[c,d]) \subset BV_r(X,[c,d])$ for every $r \geq d-c$. 

\begin{thm} \label{newprinciple} \cite{Me-Helly} 
	For every linearly ordered set $X$ the set of functions $BV_r(X,[c,d])$ is a tame family. 
	In particular, this	is true also for $M_+(X,[c,d])$. 
\end{thm}

This result together with Lemma \ref{l:tametype} 
of the present work leads to the following application. 
See also Corollary \ref{l:NoIndCircle}. 

\begin{thm} \label{t:monot}
	Let $X$ be a linearly ordered set, $f: X \to \R$ a (not necessarily continuous) function in $BV_r(X)$ and  
	$S \subset  M_+(X,X)$ a semigroup of order preserving (not necessarily continuous) selfmaps. Then for every point $z \in X$ the coding function
	$$
	m(f,z) : S \to \R, \ s \mapsto f(s(z))
	$$
	is tame on the discrete copy of $S$. 
\end{thm}
\begin{proof} 
	For every $s \in S$ and $x_0 \leq x_1 \leq \cdots \leq x_n$ in $X$ we have $sx_0 \leq sx_1 \leq \cdots \leq sx_n$. Then 
		$
		\sum_{i=0}^{n-1} |f(sx_{i+1})-f(sx_i)| \leq r
		$
		because $f \in BV_r(X)$. It follows that $fs \in BV_r(X)$. So the 
	 orbit $fS$ of $f$ is a bounded family of functions with bounded total variation. Then $fS$ is a tame family 
	 by Theorem \ref{newprinciple}. Hence by Lemma \ref{l:tametype}.6 the function $m(f,z)$ is tame. 
\end{proof}

\begin{defin} 
Let $f: \T\to \R$ be a function. We say that it has 
total variation $\leq r$ and write 
$f \in BV_r(\T)$ if the induced function  
$f \circ q_1: [0,1] \to \R$ belongs to $BV_r[0,1]$. 
\end{defin}

\begin{lem} \label{l:BV}
	For every $\phi \in M_+(\T,\T)$ and every $f \in BV_r(\T)$ we have $f \circ \phi \in BV_{2r}(\T)$. 
\end{lem}
\begin{proof}
	Let $\Phi$ be the canonical lift of $\phi$ and  $\Phi^*: [0,1] \to [0,2]$ be its kernel. Then we have to show that $f \circ \phi \circ q_1: [0,1] \to \R$ belongs to $BV_{2r}[0,1]$. 
	Since, $\phi \circ q_1= q_2 \circ \Phi^*$, it is equivalent to 
showing	
	that $f \circ q_2 \circ \Phi^* \in BV_{2r}[0,1]$. 
	Since $\Phi^*$ is 
monotone	
	it suffices to show that 
	$f \circ q_2 \in BV_{2r}[0,2]$.   
	Now it is enough to see that the restrictions of $f \circ q_2$ 
to	
	the subintervals $[0,1]$ and $[1,2]$ 
are in $BV_{r}$. The first is clear by the definition of $f \in BV_r[0,1]$. 
	The rest is easy using the fact that $q_2$ has period $1$. 
	\end{proof}

\begin{cor} \label{l:NoIndCircle} 	 
	Any bounded family of functions 
	 $\Ga:=\{f_i: \T \to \R\}_{i \in I}$ with 
	 finite total variation 
	 is tame. 
	 
\end{cor}
\begin{proof}
	Assuming the contrary let $f_n \in BV_r(\T)$ be an independent sequence of functions. Then since $q_1: [0,1] \to \T$ is onto we obtain by Lemma \ref{l_1}.2 that the sequence 
	$f_n \circ q_1$ of functions on $[0,1]$ is independent, too. This contradicts 
Theorem \ref{newprinciple} because $f_n \circ q_1: [0,1] \to \R$ is a bounded family of bounded total variation.  
	
	\end{proof}

\begin{thm} \label{t:circleWalk} Let $f: \T \to \R$ be a (not necessarily continuous) function in $BV_r(\T)$.
	\ben
	\item 
	Let $S \subset  M_+(\T,\T)$ be a semigroup of orientation  preserving (not necessarily continuous) selfmaps. Then for every point $z \in \T$ the coding function
	$$
	m(f,z) : S \to \R, \ s \mapsto f(s(z))
	$$
	is tame on the discrete copy of $S$. 
		\item
	Let $\sigma: \T \to \T$ be an orientation preserving selfmap.  
	Then the coding function
	$$
	m(f,z) : \N \cup \{0\} \to \R, \ n \mapsto f(\s^n(z))
	$$
	is tame.  If $\s$ is a bijection then one can replace $\N \cup \{0\}$ by $\Z$. 
	 \een
\end{thm}
\begin{proof} (1) By Lemma \ref{l:BV} we have $fS \subset BV_{2r}(\T)$. Clearly, $fS$ is bounded. 
	Then by Corollary \ref{l:NoIndCircle} 
	$fS$ is a tame family. 
	Now Lemma \ref{l:tametype}.6 finishes the proof.  
	
	(2) 
This	
	is a particular case of (1). 
	\end{proof}

\begin{remark} \label{applicBV}  
	(Coloring function on the circle) 	
	
	Consider a finite partition of 
	$\T=\cup_{i=0}^{d} I_i$, 
	where each $I_i$ is an arc on $\T$ (open, closed or containing one of the boundary points). Then any coloring of this partition, that is, any 
	function $f: \T \to \R$ which is constant on 
	each $I_i$ is of finite variation. In particular Theorem \ref{t:circleWalk} gives now an additional way of proving Theorem \ref{multi}, which asserts that the multidimensional Sturmian sequences of Example \ref{d:mSturm} are tame. 
\end{remark}


\begin{remark} Several results of this subsection 
	(the case of $\T$) can be generalized to general cyclically ordered sets and orientation preserving actions. 
	We intend to deal with this issue in some future publication. 
	See \cite{GM-c}. 
\end{remark}

\section{Intrinsically tame groups}\label{sec,int}

Recall that for every topological group $G$ there exists a
unique universal minimal $G$-system $M(G)$.
Frequently 
 $M(G)$ is nonmetrizable. For example, this is the case for every locally compact noncompact $G$. On the other hand,
many interesting massive Polish groups are extremely amenable that is, having trivial $M(G)$.
See for example \cite{Pest98, PestBook, UspComp}.
The first example of a nontrivial yet computable small $M(G)$ was found by Pestov.
In \cite{Pest98} he shows that for $G:=H_+(\T)$ the universal minimal system $M(G)$
can be identified with the natural action of $G$ on the circle $\T$.
Glasner and Weiss \cite{GWsym,GW-Cantor} gave an explicite description
of $M(G)$ for the symmetric group $S_{\infty}$ and for $H(C)$
(the Polish group of homeomorphisms of the Cantor set $C$).
Using model theory Kechris, Pestov and Todor\u{c}evi\'{c} gave in \cite{KPT} many new
examples of computations of $M(G)$ for various subgroups $G$ of $S_{\infty}$.

\sk

\begin{defin} \label{d:int-tame} 
We say that a topological group $G$ is \emph{intrinsically tame}
if the universal minimal $G$-system $M(G)$ is tame.
Equivalently, if
every continuous action of $G$ on a compact space $X$ admits a $G$-subsystem $Y \subset X$ which is tame.
\end{defin}

By Lemma \ref{l:tame-prop}.1 every $G$-system $X$ admits a 
largest
tame $G$-factor. Every topological group $G$ has a universal minimal tame system $M_t(G)$  (which is the 
largest
tame $G$-factor of $M(G)$). So $G$ is intrinsically tame iff the natural projection $M(G) \to M_t(G)$ is an isomorphism.  

The $G$-space $M_t(G)$ can also be described as a minimal left ideal in the 
universal space $G^{\Tame}$. Recall that the latter is isomorphic to its own enveloping semigroup
and thus has a structure of a compact right topological semigroup. Moreover, any
two minimal left ideals there are isomorphic as dynamical systems.


In \cite{GM1} we defined, for a topological group $G$ and
a dynamical property $P$, the notion of $P$-fpp
($P$ fixed point property). Namely $G$ has the $P$-fpp if every $G$-system which
has the property $P$ admits a $G$ fixed point.
Clearly this is the same as demanding that every minimal $G$- system with the property
$P$ be trivial. Thus for $P=\Tame$ a group $G$ has the tame-fpp iff $M_t(G)$ is trivial.

\sk 
We will need the following
theorem which extends a result in \cite{Gl-str}.

\begin{thm}\label{pd}
Let $(G,X)$ be a metrizable minimal tame dynamical system and suppose
it admits an invariant probability measure. Then $(G,X)$ is point distal.
If moreover, with respect to $\mu$ the system $(G,\mu,X)$ is weakly
mixing then it is a trivial one point system.
\end{thm}

\begin{proof}
With notations as in \cite{Gl-str} we observe that for any minimal
idempotent $v \in E(G,X)$ the set $C_v$ of continuity points of $v$
restricted to the set $\ov{vX}$, is a dense $G_\del$ subset of $\ov{vX}$
and moreover $C_v \subset vX$ (\cite[Lemma 4.2.(ii)]{Gl-str}).
Also, by \cite[Proposition 4.3]{Gl-str} we have $\mu(vX)=1$, and it follows
that $\ov{vX}=X$. The proof of the claim that $(G,X)$ is point distal
is now finished as in \cite[Proposition 4.4]{Gl-str}.

Finally, 
if 
the measure preserving system $(G,\mu,X)$ is weakly
mixing it follows that it is also topologically weakly mixing. By the
Veech-Ellis structure theorem for point distal systems
\cite{V, E}, if $(G,X)$ is nontrivial it admits a nontrivial equicontinuous
factor, say $(G,Y)$. However $(G,Y)$, being a factor of $(G,X)$, is
at the same time also topologically weakly mixing which is a contradiction.
\end{proof}

\begin{remark}
It seems that this observation, namely that the existence of
an invariant measure can replace the assumption that $G$ is abelian
in proving point distally, can be pushed to a proof of the full statement
of Proposition 5.1 in \cite{Gl-str} (modulo some obvious modifications) under the assumption that $X$
supports an invariant measure.
\end{remark}


\begin{thm} \label{thm4} \
\begin{enumerate}
\item
Every extremely amenable group is intrinsically tame.
\item
The Polish group $H_+(\T)$ of orientation preserving homeomorphisms of the circle is intrinsically tame.
\item
The Polish groups $\Aut(\mathbf{S}(2))$ and 
$\Aut(\mathbf{S}(3))$, of automorphisms 
of the circular directed graphs $\mathbf{S}(2)$ and $\mathbf{S}(3)$, are intrinsically tame.
\item
A discrete group which is intrinsically tame is finite.
\footnote{Modulo an extension of Weiss' theorem, which does not yet exist, a similar idea will work for any locally compact group.
The more general statement would be: A locally compact group which is
intrinsically tame is compact.}
\item
For an abelian infinite countable discrete group $G$, its universal
minimal tame system $M_t(G)$ is a highly proximal extension of
its Bohr compactification $G^{AP}$ (see e.g. \cite{Gl-str}). 
\item
The Polish group $H(C)$, of homeomorphisms of the Cantor set,
is not intrinsically tame.
\item
The Polish group $G=S_{\infty}$, of permutations of the natural numbers,
is not intrinsically tame. In fact $M_t(G)$ is trivial; i.e. $G$ has the tame-fpp.
\end{enumerate}
\end{thm}

\begin{proof}
Claim (1) is trivial and claim (2) follows from Pestov's theorem \cite{Pest98} which
identifies, for $G = H_+(\T)$, the universal minimal
dynamical system $(G,M(G))$ with the tautological action $(G,\T)$, and from Theorem \ref{GenH_+}
which asserts that this action is tame.

(3) 
The universal minimal $G$-systems for the groups $\Aut(\mathbf{S}(2))$ and 
$\Aut(\mathbf{S}(3))$ are computed in \cite{van-the}.
In both cases it is easy to check that every element of the enveloping semigroup
$E(M(G))$ is an order preserving map. As there are only $2^{\aleph_0}$ order
preserving maps, it follows that the cardinality of $E(M(G))$ is
$2^{\aleph_0}$, whence, in both cases, the dynamical system $(G, M(G))$ is tame.

In order to prove Claim (4) we assume, to the contrary, that $G$ is infinite
and 
apply a result of B. Weiss  \cite{W}, to obtain
a minimal model, say $(G,X,\mu)$, of the
Bernoulli 
probability measure preserving
system $(G,\{0,1\}^G,  (\frac12 (\del_0 + \del_1))^G)$.
Now $(G,X,\mu)$ is metrizable, minimal and tame, and it carries
a $G$-invariant probability measure with respect to which the system is weakly mixing.
Applying Theorem \ref{pd} we conclude that $X$ is trivial.
This contradiction finishes the proof.

(5) 
In \cite{H}, \cite{KL} and \cite{Gl-str} it is shown that a metric minimal tame $G$-system
is an almost one-to-one extension of an equicontinuous system.
(Note that not every minimal almost one-to-one extension of
a minimal equicontinuous $G$-system is tame,
such systems e.g. can have positive topological entropy.)
Of course every minimal equicontinuous $G$-system is tame. 
Now tameness is preserved under sub-products, and because our group $G$ is
countable, it follows that $M_t(G)$ is a minimal sub-product of all the
minimal tame metrizable systems. In turn this implies that $M_t(G)$ is
a (non-metrizable)  highly proximal extension of the Bohr compactification
$G^{AP}$ of $G$.

(6) 
To see that $G = H(C)$ is not intrinsically tame
it suffices to show that the tautological action $(G,C)$, which is a factor of $M(G)$, is not tame. To that end note that the shift transformation
$\sig$ on $X = \{0,1\}^\Z$ is a homeomorphism of the
Cantor set. Now the enveloping semigroup $E(\sig,X)$ of the cascade
$(\sig,X)$, a subset of $E(G,X)$, is homeomorphic to $\beta\N$.

(7) 
To see that $G = S_{\infty}$ is not intrinsically tame we recall first that,
by \cite{GW}, the universal minimal dynamical system for this group
can be identified with the natural action of $G$ on the compact metric
space $X = LO(\N)$ of linear orders on $\N$. Also, it follows from the
analysis of this dynamical system that for any minimal idempotent
$u \in E(G,X)$ the image of $u$ contains exactly two points,
say $uX = \{x_1,x_2\}$.
A final fact that we will need concerning the system $(G,X)$ is that
it carries a $G$-invariant probability measure $\mu$ of full support \cite{GW}.
Now to finish the proof, suppose that $(G,X)$ is tame. Then there
is a {\bf sequence} $g_n \in G$ such that $g_n \to u$ in $E(G,X)$.
If $f \in C(X)$ is any continuos real valued function, then we have,
for each $x \in X$,
$$
\lim_{n \to \infty} f(g_n x) = f(ux) \in \{f(x_1), f(x_2)\}.
$$
But then, choosing a function $f \in C(X)$ which vanishes at the points\ $x_1$ and $x_2$ and with $\int f\,d\mu=1$, we get, by Lebesgue's theorem,
$$
1 = \int f \, d\mu = \lim_{n\to \infty} \int f(g_nx) \, d\mu =
\int f (ux) \, d\mu =0.
$$

Finally, the property of supporting an invariant measure, as well as the 
fact that the cardinality of the range of minimal idempotents is
$\le 2$, are inherited by factors and the same argument shows that $M(G)$
admits no nontrivial tame factor.
Thus $M_t(G)$ is trivial.  
\end{proof}

\begin{remark} \label{r:11}
A theorem of Huang, Kerr-Li and Glasner (\cite{H}, \cite{KL}, \cite{Gl-str})
asserts that:
for G abelian any 
metrizable minimal tame action is 
almost automorphic; i.e. an 
almost 1-1 extension of an equicontinuous system.
The fact that the minimal action of
$H_+(\T)$ on $\T$ is tame
shows that
some restrictive assumption on the group $G$ is
really necessary. Other non abelian minimal tame actions which are not 
almost automorphic are given in Example \ref{e:tameNOThns}(1)
and  Theorem \ref{thm4}(3).

\end{remark}

It would be interesting to find other examples of intrinsically tame Polish groups.

\vspace{0.3cm}

The (nonamenable) group $G=H_+(\T)$ has one more remarkable property. Besides $M(G)$,
one can also effectively compute the affine analogue of $M(G)$.
Namely, the \emph{universal irreducible affine system} of $G$ (we denote it by $I\!A(G)$) which was defined
and studied in \cite{Gl-thesis, Gl-book1}.
It is uniquely determined up to affine isomorphisms.  The corresponding affine compactification $G \to I\!A(G)$ is equivalent to the
affine compactification $G \to P(M_{sp}(G))$, where, $M_{sp}(G)$ is the \emph{universal strongly proximal minimal system} of $G$
and $P(M_{sp}(G))$ is the space of probability measures on the compact
space $M_{sp}(G)$. 
For more information regarding affine compactifications of dynamical systems we refer to \cite{GM-AffComp}.

\begin{defin} \label{d:ConvIntTame}
We say that $G$ is \emph{convexly intrinsically tame}
 (or \emph{conv-int-tame} for short) 
if the $G$-system $I\!A(G)$ is tame.
\end{defin}
Note that this condition holds 
iff every compact affine dynamical system
$(G,Q)$ admits an affine tame $G$-subsystem, iff
every compact affine dynamical system
$(G,Q)$ admits a tame $G$-subsystem.
The latter assertion follows from the fact that
$P(X)$ is tame whenever $X$ is \cite[Theorem 6.11]{GM-rose}, and by the affine universality of $P(X)$.
In particular, it follows that any intrinsically tame group 
 is convexly intrinsically tame.
Of course any amenable topological group (i.e. a group with trivial $I\!A(G)$) has this property.
Thus we have the following diagram which
emphasizes the analogy between the two pairs of properties:
\begin{equation*}
\xymatrix
{
 \text{extreme amenability}\  \ar@2{->} [d] \ar@2{->}[r]\  &  \  \text{intrinsically tame} \ar@2{->}[d]\  \\
 \text{amenability}\  \ar@2{->}[r] &  \  \text{convexly intrinsically tame}
}
\end{equation*}

\begin{remark} 
Given a class $P$ of compact $G$-systems which is stable under subdirect products, one 
can define the notions of intrinsically $P$-group and convexly intrinsically $P$ group in a manner analogous to the one we adopted for $P=\Tame$.
We then note that 
in this terminology a group is convexly intrinsically HNS (and, hence, also conv-int-WAP) iff it is amenable. This follows easily from the fact that the algebra $\Asp(G)$ is left amenable, \cite{GM-fp}. 
This ``collapsing effect" together with the special role of tameness in the dynamical BFT dichotomy \ref{D-BFT} suggest that the 
notion of conv-int-tameness is a natural analogue of amenability.

\end{remark}

At least for discrete groups, if $G$ is intrinsically HNS then it is finite.
In fact, for any group, an HNS minimal system is equicontinuous 
(see \cite{GM1}),
 so that
for a group $G$ which is intrinsically HNS the universal minimal system $M(G)$ coincides
with its Bohr compactification $G^{AP}$.
Now for a discrete group, it is not hard to show that an infinite minimal equicontinuous system admits a nontrivial almost one to one (hence proximal) 
extension which is still minimal.
Thus $M(G)$ must be finite.
However, by a theorem of Ellis \cite{Ellis}, for discrete groups the group $G$ acts freely on $M(G)$,
so that $G$ must be finite as claimed. Probably similar arguments will show that 
a locally compact intrinsically HNS group is necessarily compact.

%

The Polish group $S_{\infty}$
is amenable (hence convexly intrinsically tame) but not intrinsically tame.

The group $H(C)$ is not convexly intrinsically tame.
In fact, its natural action on the Cantor set $C$ is minimal and strongly proximal, but this action is not tame; it contains as a subaction a copy of the full shift $(\Z,C)=(\sigma,\{0,1\}^\Z)$.
 The group $H([0,1]^{\N})$ is a universal Polish group (see Uspenskij \cite{UspUn}).  It also is not convexly intrinsically tame.
 This can be established by observing that the action of this group on the Hilbert cube is minimal, strongly proximal and not tame.
 The strong proximality of this action can be
 easily checked.
The action is not tame because it is a \emph{universal action} (see \cite{MeNZ}) for all Polish groups on compact metrizable spaces.

The (universal) minimal $G$-system $\T$ for $G=H_+(\T)$ is strongly proximal. Hence, $I\!A(G)$ in this case is easily computable
and it is exactly $P(\T)$ which, as a $G$-system, is tame (by Theorem \ref{thm4} and a remark above).
So, $H_+(\T)$ is a (convexly) intrinsically tame nonamenable topological group. 

Another example of a Polish group which is nonamenable yet
convexly intrinsically tame (that is, with tame
$I\!A(G)$) is any semisimple Lie group $G$ with finite center
and no compact factors.
Indeed, by Furstenberg's result \cite{Furst-63-Poisson} the universal minimal strongly proximal system
$M_{sp}(G)$ is the homogeneous space $X=G/P$, where $P$ is a minimal parabolic subgroup (see \cite{Gl-book1}).
Results of Ellis \cite{Ellis93} and Akin \cite{Ak-98} (Example \ref{e:tameNOThns}.1) show that the enveloping semigroup $E(G,X)$ in this case is a Rosenthal compactum, whence the  system $(G,X)$ is tame by the dynamical BFT dichotomy (Theorem \ref{D-BFT}).


\sk

\begin{question} \label{q1}
For which (Polish) groups $G$
the following universal constructions
lead to tame $G$-systems:
\bit
\item [(a)] The  Roelcke compactification $R(G)$ (of a Roelcke precompact group);

\item [(b)]  the universal minimal system $M(G)$;

\item [(c)]  the universal irreducible affine system $I\!A(G)$ ?
\eit
\end{question}

A 
related 
question 
is to compute the largest tame factor for 
these (and some additional) 
compactifications of $G$.


\section{Appendix}\label{app}

The proof of the following result was communicated to
us by Stevo Todor\u{c}evi\'{c}.

\begin{thm} \label{main}
	$\beta \N$ is not WRN.
\end{thm}

For the proof we will need several definitions and lemmas.

\begin{defin} \label{d:indPairs}
	A family of disjoint pairs of subsets $\{(F_i,G_i) : i \in I\}$ of a set $S$ is
	{\em independent} if for every two disjoint finite sets $K, L \subset I$
	$$
	\bigcap_{i \in K} F_i \  \cap \ \bigcap_{i \in L} G_i \not=\emptyset.
	$$
\end{defin}

\begin{defin}
	A sequence $\{F_n : n \in \N\}$ of subsets of a set $S$ is
	{\em convergent} if for every $s \in S$ there is
	$n_0$ such that, either $s \in F_n$ for every $n \ge n_0$, or
	$s \not\in F_n$ for every $n \ge n_0$.
\end{defin}

%
%

\begin{lem}\label{beta}
	There exists a family $\{(F_r,G_r) : r \in \R\}$ of pairs of disjoint closed sets of $\beta \N$ which is
	independent.
\end{lem}

\begin{proof}
	As is well known the Cantor cube $\Om=\{0,1\}^\R$ contains a dense countable sequence, say
	$\{\om_n\}_{n \in \N}$.
	Let $\rho : \beta \N \to \Om$ be the unique continuous extension to $\beta\N$ of the map $\N \to \Om$,
	$n \mapsto \om_n$. Then $\rho$ is a continuous surjection and for $r \in \R$ we set
	$$
	F_r = \{x \in \beta \N : \rho(x)(r) = 0\}, \qquad G_r = \{x \in \beta \N : \rho(x)(r) = 1\}.
	$$
\end{proof}

The next lemma is a crucial tool. Its proof is very similar to the proof of
a theorem of Rosenthal \cite{Ro}
(see also \cite{Far} and \cite[page 100]{LT}).

\begin{lem}\label{product}
	Let $\{(A_i,B_i) : i \in \om\}$ be an independent family
	of disjoint pairs of subsets of a set $S$.
	Suppose there is a positive integer $k \ge 1$,
	and $k$ families of disjoint pairs $\{(A_{ij}, B_{ij}) : i \in \om\},\ 1 \le j \le k$,
	such that:
	$$
	A_i \times B_i \ \subset \ \bigcup_{j=1}^k A_{ij} \times B_{ij}.
	$$
	Then there is an infinite $M \subset \om$ and $j_0 \in \{1,2,\dots,k\}$ such that
	the family
	$$
	\{(A_{ij_0},B_{ij_0}) : i \in M\}
	$$
	is an independent family.
\end{lem}

\begin{proof}
	We let $[\om]^\om$ denote the collection of infinite subsets of $\om$. More generally,
	if $M \in [\om]^\om$ then $[M]^\om$ denotes the collection of infinite subsets of $M$.
	The space $[\om]^\om$ carries a natural topology when we identify it with the subset of
	the Cantor space $\{0,1\}^\om$ consisting of sequences with infinitely many $1$'s.

	For $1 \le j \le k$ let
	\begin{align*}
	\mathcal{X}_j =
	\{ M =\{m_1 <   m_2  & < \cdots\}   \in [\om]^\om : \  \forall n < \om, \ \\
	& \bigcap_{l=1}^n A_{m_{2l},j} \ \cap \ \bigcap_{l=1}^n B_{m_{2l+1},j} \not=\emptyset \}
	\end{align*}
	Each $\mathcal{X}_j$ is a closed subset of $[\om]^\om$ and the Galvin-Prikry theorem
	\cite{GP} implies that
	either:
	
	(i) there is some $1 \le j \le k$ and $M \in [\om]^\om$ such that $[M]^\om \subset \mathcal{X}_j$, or
	
	(ii) there is an $M \in [\om]^\om$ such that for every $1 \le j \le k$,
	$[M]^\om \cap \mathcal{X}_j =\emptyset$.
	
	In the first case, where $[M]^\om \subset \mathcal{X}_j$,
	let $M = \{m_n : n < \om\}$ and set $N = \{m_{2n} : n < \om\}$.
	Then $\{(A_{n,j}, B_{n,j})  : n \in N\}$ is independent. In fact, given
	a positive integer $u \ge 1$ and two disjoint finite sets
	$K, L  \subset N$, such that $K \cup L = \{m_2, m_4, \dots, m_{2u}\}$,
	construct a sequence $N_1 =\{n_h\}_{h  \in \N} \subset N$ which contains the integers
	$\{m_2, m_4, \dots, m_{2u}\}$, scattered among $\{n_1, n_2, \dots , n_{2u}\}$ in such a way that
	for $m_{2p} \in K, m_{2p}  = n_{2h}$ for some $1 \le h \le u$,
	and for $m_{2p} \in L, m_{2p}  = n_{2h+1}$ for some $1 \le h \le u$.
	Since $N_1 \in \mathcal{X}_j$ we now have
	$$
	\bigcap_{m_{2p} \in K} A_{m_{2p},j} \ \cap \  \bigcap_{m_{2p} \in L} B_{m_{2p},j}
	\ \supset \
	\bigcap_{h =1}^u A_{n_{2h},j} \ \cap \  \bigcap_{h =1}^u B_{n_{2h+1},j} \ne \emptyset.
	$$

	%
	
	Our proof will be complete when we show next that in our situation the case (ii)
	can not occur.
	In fact, we show that if $[M]^\om \cap \mathcal{X}_j =\emptyset$, then the sequence
	$\{A_{ij} : i \in M\}$ converges. For otherwise we can find a point $s \in S$ and an infinite subsequence
	$N = \{n_1 < n_2 < \cdots \}\subset M$ such that $s \in A_{n_{2l},j}$ and $s \not\in A_{n_{2l+1},j}$
	for every $l \ge 1$. But then, $N \in [M]^\om \cap \mathcal{X}_j$, contradicting our assumption.
	
	Thus under the assumption that (ii) holds, for every $1 \le j \le k$, the sequence $\{A_{ij} : i \in M\}$
	converges. This however clearly contradicts the independence of the family  $\{(A_i,B_i) : i \in \om\}$,
	and our proof is complete.
\end{proof}

\begin{lem} \label{ind-ell1}
	\emph{(Rosenthal \cite[Proposition 4]{Ro})}
	Let $S$ be a set and $\{f_n\}_{n \in \N} \subset \R^S$ a uniformly bounded sequence
	of functions on $S$. Suppose there are real numbers
	$p < q$ such that the pairs of sets
	$$
	F_n = \{s \in S : f_n(s) \le p\}\ \ {\text{and}}\ \ G_n = \{s \in S : f_n(s) \ge q\}
	$$
	form an independent family. Then the sequence $\{f_n\}_{n \in \N}$ is an $\ell_1$-sequence
	in the Banach space $\ell_\infty(S)$.
\end{lem}

\sk

\begin{proof}[Proof of Theorem \ref{main}]
	Suppose to the contrary that $V$ is a Rosenthal Banach space and that $B_{V^*}$, equipped with its
	weak$^*$ topology, contains a copy $\Phi$ of $\beta \N$.
	By Lemma \ref{beta} there exists a family $\{(F_i,G_i) : i \in \R\}$ of pairs of disjoint closed subsets
	of $\Phi$ which is independent. By the nature of the weak$^*$ topology,
	for each $i$ there exist a finite set $\{v_{ij}: 1 \le j \le k_i\} \subset B_V$, the unit ball of $V$,
	and a finite set of pairs $\{q_{ij} < q'_{ij} : 1 \le j \le k_i\} \subset \Q$,
	such that the sets
	$$
	Q_i = \bigcap_j \{\phi \in V^* : \phi(v_{ij}) \le q_{ij} \}, \qquad
	Q'_i = \bigcap_j \{\phi \in V^* : \phi(v_{ij}) \ge q'_{ij} \},
	$$
	separate the pair $(F_i,G_i)$; i.e. $F_i \subset Q_i$ and $G_i \subset Q'_i$.
	Now, as $|\R| = \bf{c}$ and there are only countably many possible choices 
	for the values $k_i$
	we conclude,
	by the pigeon holes principle, that
	there exists a finite positive integer $k \ge 1$ and
	an uncountable subset $D \subset \R$ with $k_i =k$ for every $i \in D$.
	Next we chose an arbitrary infinite countable subset $C \subset D$ and we now have
	a countable subfamily $\{(F_i, G_i) : i \in C\}$ (for a countable subset $C \subset \R$)
	such that $k_i = k$ for every $i \in C$.
	Clearly the family $\{(Q_i \cap \Phi, Q'_i \cap \Phi) : i \in C\}$ is a {\em countable} independent family.

	Applying Lemma \ref{product} to the sets
	\begin{gather*}
	A_i = Q_i \cap \Phi, \qquad B_i = Q'_i \cap \Phi,   \\
	A_{ij} = \{\phi \in \Phi : \phi(v_{ij}) \le q_{ij} \}, \qquad
	B_{i,j} = \{\phi \in \Phi : \phi(v_{ij}) \ge q'_{ij} \},
	\end{gather*}
	with $i \in C$ and $1 \le j \le k$,
	we conclude that for some infinite $M \subset C$ and some $j_0 \in \{1,2,\dots,k\}$
	the family
	$$
	\{(A_{i j_0},B_{i j_0}) : i \in M\}
	$$
	is an independent family.
	
	Applying Lemma \ref{ind-ell1} to the sequence of functions $f_i = v_{i j_0} \rest \Phi : \Phi \to \R$,
	we conclude that this sequence is an $\ell_1$-sequence in the Banach space
	$C(\Phi)$. However, as the restriction map $v \mapsto f_v := v \rest \Phi$, $V \to C(\Phi)$ satisfies
	$$
	\|v\| = \sup_{\phi \in B_{V^*}}  |\phi(v)| \ge \|f_v\| = \sup_{\phi \in \Phi}  |\phi(v)|,
	$$
	it follows that
	$\{v_{i j_0} : i \in M\}$ is also an $\ell_1$-sequence in $V$.
	This contradicts our assumption that $V$ is a Rosenthal space and the proof is complete.
\end{proof}

\begin{remark}
	The proof of Theorem \ref{main} actually shows that the cube $Q(\om_1) =[0,1]^{\om_1}$, as well as any compactum $K$ which maps continuously onto $Q(\om_1)$, is not a WRN compactum.
\end{remark}

\bibliographystyle{amsplain}

\begin{thebibliography}{10}


\bibitem{Ak-98}
E. Akin, {\em Enveloping linear maps\/}, in: Topological dynamics
and applications, Contemporary Mathematics {\bfseries 215}, a
volume in honor of R.~Ellis, 1998, pp. 121-131.

\bibitem{AG-16}
Ethan Akin and Eli Glasner
WAP Systems and Labeled Subshifts,
arXiv:1410.4753.




\bibitem{AK}
A. Aviles and P. Koszmider,
\emph{A continuous image of a Radon-Nikod\'ym compact which is not Radon-Nikod\'ym}, Duke Math. J. {\bf 162} (2013), no. 12, 2285-2299.


\bibitem{Auj}
J. B. Aujogue,
\emph{Ellis enveloping semigroup for almost canonical model sets of an Euclidean space}, 
Algebraic \& Geometric Topology {\bfseries 15} (2015) 2195-2237.


\bibitem{AuYo}
J. Auslander and J. Yorke, {\it Interval maps, factors of maps and
chaos}, Tohoku Math. J. {\bfseries 32} (1980), 177-188.




%




%




\bibitem{BenTsankov}
I. Ben Yaacov and T. Tsankov,
\emph{Weakly almost periodic functions, model-theoretic stability, and
minimality of topological groups}, 
Trans. of AMS, {\bf 368} (2016), 8267-8294.

\bibitem{BRW}
Y. Benyamini,  M. E.  Rudin and M. Wage,
{\it Continuous images of weakly compact subsets of Banach spaces},
Pacific J. Math. {\bf 70}, (1977), 309-324.

\bibitem{BFZ}
V. Berthe, S. Ferenczi, L.Q. Zamboni,
\emph{Interactions between Dynamics, Arithmetics and
Combinatorics: the Good, the Bad, and the Ugly}, Contemporary Math., {\bf 385} (2005), 333-364.


\bibitem{BV}
V. Berthe, L. Vuillon, \emph{Palindromes and two-dimensional Sturmian sequences,}
J. Automata, Languages and Combinatorics, {\bf 6} (2001), 121-138.

\bibitem{BJM}
 J.F. Berglund, H.D. Junghenn and P. Milnes,
{\it Analysis on Semigroups}, Wiley, New York, 1989.

\bibitem{BFT} J. Bourgain, D.H. Fremlin and M. Talagrand,
{\it Pointwise compact sets in Baire-measurable functions}, Amer.
J. of Math., \textbf{100:4} (1977), 845-886.


%



\bibitem{CNO} 
B. Cascales, I. Namioka and J. Orihuela, \emph{The Lindelof property in
Banach spaces}, Studia Math. {\bf 154} (2003), 165-192. 

\bibitem{DFJP}
W.J. Davis, T. Figiel, W.B. Johnson and A. Pelczy\'nski, {\em
Factoring weakly compact operators\/}, J. of Funct. Anal.,
{\bf 17} (1974), 311-327.


\bibitem{DNR}
B. Deroin, A. Navas and C. Rivas, 
\emph{Groups, Orders, and Dynamics},  2014, arXiv:1408.5805. 



\bibitem{Dulst}
D. van Dulst, \emph{Characterizations of Banach spaces not containing $l^1$}.
Centrum voor Wiskunde en Informatica, Amsterdam, 1989.

\bibitem{C-W}
D. Cenzer,  A. Dashti, F. Toska and S. Wyman,
{\em Computability of countable subshifts in one dimension},
Theory Comput. Syst. {\bf 51} (2012), no. 3, 352-371.


\bibitem{Ellis}
R. Ellis, 
{\em Universal minimal sets}, 
Proc. Amer. Math. Soc. {\bf 11} (1960),  540-543.


\bibitem{E}
R. Ellis,
{\em The Veech structure theorem},
Trans. Amer. Math. Soc. {\bf 186} (1973), 203-218 (1974).

\bibitem{Ellis93}
R. Ellis, {\it The enveloping semigroup of projective flows},
Ergod. Th. Dynam. Sys. {\bf 13} (1993), 635-660.



\bibitem{Eng}
R. Engelking, {\em  General topology\/}, revised and completed
edition, Heldermann Verlag, Berlin, 1989.

%

\bibitem{Fa}
M. Fabian, {\em Gateaux differentiability of convex functions and
topology. Weak Asplund spaces\/},\ Canadian Math.\ Soc.\ Series of
Monographs and Advanced Texts, 
New York, 1997.

\bibitem{Far}
J. Farahat,
{\em Espaces de Banach contenant $\ell^1$, d'aprs H. P. Rosenthal},
(French) Espaces $L^p$, applications radonifiantes et g\'eom\'etrie des espaces de Banach,
Exp. No. {\bf 26}, 6 pp. Centre de Math., \'Ecole Polytech., Paris, 1974.

\bibitem{Fed} 
V. Fedorchuk, \textit{Ordered spaces,} Soviet Math. Dokl., 1011-1014, 1966. 


\bibitem{Fernique}
T. Fernique, \emph{Multi-dimensional Sturmian sequences and generalized substitutions},
Int. J. Found. Comput. Sci., \textbf{17} (2006), pp. 575-600.



\bibitem{Furst-63-Poisson}
H. Furstenberg,
{\em A Poisson formula for semi-simple Lie groups\/},
Ann.\ of Math.\  {\bfseries 77} (1963), 335-386.

\bibitem{GP}
F. Galvin and K. Prikry,
{\em Borel sets and Ramsey's theorem},
J. Symbolic Logic, {\bf 38} (1973), 193--198.

\bibitem{Ghys}
E. Ghys, \emph{Groups acting on the circle}, Enseign. Math. (2) \textbf{48} (2001), 329-407.

\bibitem{Gl-thesis}
S. Glasner,
{\em Compressibility properties in topological dynamics},
Amer. J. Math., {\bf 97} (1975), 148--171.

\bibitem{Gl-book1}
E. Glasner, \emph{Proximal flows,} Lect. Notes, 517, Springer, 1976.

\bibitem{Gl-03}
E. Glasner,
{\it Ergodic Theory via joinings}, Math. Surveys and
Monographs, AMS, {\bf 101}, 2003.

\bibitem{Gl-tame}
E. Glasner, {\it On tame dynamical systems}, Colloq. Math.
\textbf{105} (2006), 283-295.

\bibitem{Gl-str} E. Glasner, {\it The structure of tame minimal
dynamical systems}, Ergod. Th. and Dynam. Sys. {\bf 27} (2007),
1819--1837.

 
 
\bibitem{Gl-tf}
E. Glasner,
{\em Translation-finite sets}, (2011),
ArXiv: 1111.0510.

\bibitem{GM1}
E. Glasner and M. Megrelishvili, \emph{Linear representations of
hereditarily non-sensitive dynamical systems}, Colloq. Math.,
\textbf{104} (2006), no. 2, 223--283.

\bibitem{GM-suc}
E. Glasner and M. Megrelishvili, {\it New algebras of functions on
topological groups arising from $G$-spaces}, Fundamenta Math., \textbf{201}
(2008), 1-51.

\bibitem{GM-rose}
E. Glasner and M. Megrelishvili, {\it Representations of dynamical
systems on Banach spaces not containing $l_1$},
Trans. Amer. Math. Soc.,  \textbf{364} (2012),  6395-6424.

\bibitem{GM-fp}
E. Glasner and M. Megrelishvili, {\em On fixed point theorems and
nonsensitivity},
Israel J. of Math., \textbf{190} (2012), 289-305.

\bibitem{GM-AffComp}
E. Glasner and M. Megrelishvili,
\emph{Banach representations and affine compactifications of dynamical systems},
in: Fields institute proceedings dedicated to the 2010 thematic program on asymptotic geometric analysis,
M. Ludwig, V.D. Milman, V. Pestov, N. Tomczak-Jaegermann (Editors), Springer, New-York, 2013. 
ArXiv version: 1204.0432.

\bibitem{GM-survey}
E. Glasner and M. Megrelishvili,
\emph{Representations of dynamical systems on Banach spaces,}
in: Recent Progress in General Topology III, 
(Eds.: K.P. Hart, J. van Mill, P. Simon),  
Springer-Verlag, Atlantis Press, 2014, 399-470. 

\bibitem{GM-c}
E. Glasner and M. Megrelishvili,
\emph{Circularly ordered dynamical systems,} 2016, arXiv:1608.05091. 


\bibitem{GMU} E. Glasner, M. Megrelishvili and V.V. Uspenskij,
\emph{On metrizable enveloping semigroups}, Israel J. of Math.
\textbf{164} (2008), 317-332.

\bibitem{GW1}
E. Glasner and B. Weiss,
{\it Sensitive dependence on initial conditions},
Nonlinearity {\bf 6} (1993), 1067-1075.

\bibitem{GW}
E. Glasner and B. Weiss,
{\em Quasifactors of zero-entropy systems\/},
J.\ of Amer.\ Math.\ Soc.\ {\bfseries 8}
(1995), 665--686.

\bibitem{GW-Hi}
E. Glasner and B. Weiss,
{\em On Hilbert dynamical systems},
Ergodic Theory Dynam. Systems, {\bf 32} (2012), no. 2, 629-642.

\bibitem{GWsym}
E. Glasner and B. Weiss,
{\em  Minimal actions of the group $S(\Z)$ of permutations of the integers}, Geom. Funct. Anal., {\bf 12} (2002),  964-988.

\bibitem{GW-Cantor}
E. Glasner and B. Weiss,
{\em The universal minimal system for the group of homeomorphisms of the Cantor set},
Fund. Math., {\bf 176} (2003),  277-289.

\bibitem{GY}
E. Glasner and X.  Ye,
{\em Local entropy theory},
Ergodic Theory Dynam. Systems {\bf 29} (2009),  321-356.





\bibitem{H}
W. Huang, {\em Tame systems and scrambled pairs under an abelian
group action\/}, Ergod. Th. Dynam. Sys. {\bf 26} (2006),
1549--1567.



\bibitem {JOPV}
J.E. Jayne, J. Orihuela, A.J. Pallares and G. Vera, {\it
$\sigma$-fragmentability of multivalued maps and selection
theorems}, J. Funct. Anal. {\bf 117} (1993), no. 2, 243-273.


\bibitem{Ibar}
T. Ibarlucia,
\emph{The dynamical hierachy for Roelcke precompact Polish groups}, ArXiv:1405.4613v1, 2014,
 Israel J. of Math., to appear.

\bibitem{Kau}
R. Kaufman, \textit{Ordered sets and compact spaces,} Colloq. Math., 
{\bf 17} (1967), 35-39. 

\bibitem{Kech}
A.S. Kechris, {\em Classical descriptive set theory\/}, Graduate
texts in mathematics, {\bfseries 156}, 1991, Springer-Verlag.

\bibitem{KPT}
A.S. Kechris, V.G. Pestov, and S. Todor\u{c}evi\'{c},
Fra\"iss\'e limits, Ramsey theory, and
topological dynamics of automorphism groups, Geom. Funct. Anal. \textbf{15} (2005), no. 1, 106-189.


\bibitem {KM}
P.S. Kenderov and W.B. Moors,
\emph{Fragmentability of groups and metric-valued function spaces},
Top. Appl., \textbf{159} (2012), 183-193.

\bibitem{KL05}
D. Kerr and H. Li,
{\it Dynamical entropy in Banach spaces}, Invent. Math.
{\bf 162}  (2005), 649-686.


\bibitem{KL}
D. Kerr and H. Li,
{\it Independence in topological and $C^*$-dynamics}, Math. Ann.
{\bf 338}  (2007), 869-926.

\bibitem{Ko}
A. K\"{o}hler, {\em Enveloping semigrops for flows}, Proc.
of the Royal Irish Academy, {\bf 95A} (1995), 179--191.




\bibitem{Mich+}
A. Komisarski, H. Michalewski, P. Milewski,
\emph{Bourgain-Fremlin-Talagrand dychotomy and dynamical systems,}
preprint, 2004.

\bibitem{Koz}
V.S. Kozyakin,
\emph{Sturmian sequences generated by order preserving circle maps},
Preprint No. 11/2003, May 2003, Boole Centre for Research in
Informatics, University College Cork — National University of Ireland,
Cork, 2003.


\bibitem{KMa}
J.-L. Krivine and B. Maurey,
\emph{Espaces de Banach stables,} 
Israel J. Math., \textbf{4} (1981), 273-295.

\bibitem{LNO}
A. Lima, O. Nygaard, E. Oja,
\emph{Isometric factorization of weakly compact operators and the approximation property,}
Israel J. Math., \textbf{119} (2000), 325-348.

\bibitem{LT}
J, Lindenstrauss and L. Tzafriri,
{\em Classical Banach spaces. I. Sequence spaces},
Ergebnisse der Mathematik und ihrer Grenzgebiete, Vol. {\bf 92},
Springer-Verlag, Berlin-New York, 1977.

\bibitem{Mart-Cerv} 
G. Martinez-Cervantes, 
On weakly Radon-Nikod\'ym compact spaces, ArXiv, September, 2015. 






\bibitem{MeNZ} M. Megrelishvili,
{\it Free topological $G$-groups}, New Zealand J. of Math., {\bf
	25} (1996), 59-72.

\bibitem{me-fr} M. Megrelishvili,
{\it Fragmentability and continuity of semigroup actions},
Semigroup Forum, {\bf 57} (1998), 101-126.


\bibitem{Me-op}
M. Megrelishvili, {\it Operator topologies and reflexive
	representability}, In: ``Nuclear groups and Lie groups" Research
and Exposition in Math. series, vol. {\bf 24}, Heldermann Verlag
Berlin, 2001, 197-208.




\bibitem{Me-nz}
M. Megrelishvili, {\em Fragmentability and representations of
	flows\/}, Topology Proceedings, {\bfseries 27:2} (2003), 497-544.
See also: www.math.biu.ac.il/ ${\tilde{}}$ megereli.


\bibitem{Me-opit}
M. Megrelishvili, \emph{Topological transformation groups:
selected topics,} in: Open Problems In Topology II (Elliott Pearl,
ed.), Elsevier Science, 2007, pp. 423-438.





\bibitem{Me-Helly}
M. Megrelishvili, \emph{A note on tameness of families having bounded variation}, ArXiv, 2014. 

\bibitem{MePolev}
M. Megrelishvili and L. Polev, \emph{Order and minimality of some topological groups}, 
Topology Applications, {\bf 201} (2016), 131-144. 
%
%

\bibitem{Nach}
L. Nachbin, \emph{Topology and order,} Van Nostrand Math. Studies, Princeton, New Jersey, 1965.



\bibitem{N}
I. Namioka, {\em Radon-Nikod\'ym compact spaces and
fragmentability\/}, Mathematika {\bfseries 34} (1987), 258-281.


\bibitem{NP}
I. Namioka and R.R. Phelps, {\it Banach spaces which are Asplund
spaces}, Duke Math. J., {\bf 42} (1975), 735-750.

\bibitem{Natanson} 
I.P. Natanson, {\it Theory of functions of real variable}, v. I, New York, 1964. 




\bibitem{OR}
E. Odell and H. P. Rosenthal, {\em A double-dual characterization
of separable Banach spaces containing $l\sp{1}$}, Israel J. Math.,
{\bf 20} (1975), 375-384.



\bibitem{Pest98}
V.G. Pestov,
{\em On free actions, minimal flows, and a problem by Ellis},
Trans. Amer. Math. Soc., {\bf 350} (1998),  4149-4165.

\bibitem{PestBook}
V.G. Pestov, \emph{Dynamics of infinite-dimensional groups. The
Ramsey-Dvoretzky-Milman phenomenon.} University Lecture Series, {\bf 40}. American Mathematical Society, Providence, RI, 2006.


\bibitem{Pikula}
R. Pikula,
\emph{Enveloping semigroups of affine skew products and sturmian-like systems}, Dissertation,
The Ohio State University, 2009.


\bibitem{Rayn}
Y. Raynaud, \emph{Espaces de Banach superstables, distances stables et homeomorphismes uniformes},
Israel J. Math., \textbf{44} (1983), 33-52.

\bibitem{RD}
W. Roelcke and S. Dierolf, {\em Uniform structures on topological
groups and their quotients\/}, McGraw-Hill, 1981.

\bibitem{RT}
F. Rhodes and C.L. Thompson, {\em Rotation numbers for monotone
functions of the circle}, J. Lond. Math. Soc., \textbf{34} (1986), 360-368.


\bibitem{Ro} H.P. Rosenthal,
\emph{A characterization of Banach spaces containing $l_1$}, Proc.
Nat. Acad. Sci. U.S.A., \textbf{71} (1974), 2411-2413.



%


\bibitem{Rup-WAPsets} W. Ruppert,
\emph{On weakly almost periodic sets,} Semigroup Forum, \textbf{32} (1985) 267--281.

\bibitem{SS}
E. Saab and P. Saab, \emph{A dual geometric characterization of
Bancah spaces not containing $l_1$}, Pacific J. Math.,
\textbf{105:2} (1983), 413-425.

%

\bibitem{Sh}
S.A. Shapovalov,
{\em A new solution of one Birkhoff problem},
J. Dynam. Control Systems, {\bf 6} (2000), no. 3, 331-339.


\bibitem{Tal}
M. Talagrand, \emph{Pettis integral and measure theory,} Mem. AMS No.
\textbf{51}, 1984.


\bibitem{TodBook}
S. Todor\u{c}evi\'{c}, {\em Topics in topology}, Lecture Notes in
Mathematics, {\bfseries 1652}, Springer-Verlag, 1997.

%



\bibitem{UspUn} V.V. Uspenskij,
{\it A universal topological group with countable base}, Funct.
Anal. Appl., {\bf 20} (1986), 160-161.


\bibitem{UspComp}
V.V. Uspenskij, {\em Compactifications of topological groups},
Proceedings of the Ninth Prague Topological Symposium (Prague,
August 19--25, 2001). Edited by P. Simon. Published April 2002 by
Topology Atlas (electronic publication). Pp. 331-346,
ArXiv:math.GN/0204144. 

\bibitem{van-the}
L. Nguyen van Th\'{e},
{\em More on the Kechris-Pestov-Todorcevic correspondence:
precompact expansions}, 
Arxiv: 1201.1270v3.

\bibitem{V}
W. A. Veech,
{\em Point-distal flows},
Amer. J. Math., {\bf 92} (1970), 205-242.

\bibitem{W}
B. Weiss,
{\em Minimal models for free actions. Dynamical systems and group actions},
249-264, Contemp. Math., {\bf 567}, Amer. Math. Soc., Providence, RI, 2012.



\end{thebibliography}

\end{document}